\documentclass{article}
\usepackage{amsmath}
\usepackage{amssymb}
\usepackage{amsfonts}
\usepackage{amsthm}
\usepackage{mathrsfs}   	
\usepackage{amsbsy}		% for bold \pmb{}
\usepackage{graphicx}
\usepackage{makeidx}
\usepackage{bm}
\usepackage[all]{xy}
\usepackage[left,modulo]{lineno}
\usepackage{circuitikz}		% for electricity
\usepackage{mathdots}
\usepackage[ps2pdf]{hyperref} % ps2pdf needed for compilation via Latex -> dvips -> ps2pdf
\hypersetup{ colorlinks=true,linkcolor=blue, citecolor=blue, urlcolor=blue  }

\theoremstyle{plain}
\newtheoremstyle{myexstyle}
{}{}{}{}{\bfseries}{.}{ }{}
\newtheorem{theorem}{Theorem}[section]
\newtheorem{corollary}[theorem]{Corollary}
\newtheorem{definition}[theorem]{Definition}
\newtheorem{lemma}[theorem]{Lemma}
\newtheorem{lemmadefinition}[theorem]{Lemma-Definition}
\newtheorem{notation}[theorem]{Notation}

\newtheorem{conventionnotation}[theorem]{Conventions and Notations}
\newtheorem{observation}[theorem]{Observation}
\newtheorem{proposition}[theorem]{Proposition}
\theoremstyle{myexstyle}
\newtheorem{remark}[theorem]{Remark}
\newtheorem{example}[theorem]{Example}

\numberwithin{equation}{section}

\renewenvironment{proof}[1][Proof]{\noindent\textbf{#1.} }{\ \rule{0.5em}{0.5em}}
\numberwithin{equation}{section}

\makeindex

\begin{document}

\title{Partial Dirac structures in infinite dimension}

\author{Fernand Pelletier \& Patrick Cabau }
\maketitle


\noindent\textbf{Classification MSC 2010.--}  
22E65, %Infinite-dimensional Lie groups and their Lie algebras: general properties 
46T05, %Infinite-dimensional manifolds
55R10, %Fiber bundles in algebraic topology
70G45  %Differential geometric methods (tensors, connections, symplectic, Poisson, contact, Riemannian, nonholonomic, etc.) for problems in mechanics
.

\date{}

\begin{abstract}
In this paper the notion of Dirac structure in finite dimension is extended to the convenient setting. In particular, we introduce the notion of partial Dirac structure on convenient Lie algebroids and manifolds. We then look for those structures whose classical geometrical results in finite dimension can be extended to this infinite dimensional context. Finally, we are interested in the projective and direct limits of such structures.
\end{abstract}

\tableofcontents

\bigskip

Dirac structures were introduced in \cite{Cou90} and \cite{CoWe88}, in the finite dimensional setting, in order to treat both weak symplectic  and Poisson structures in a unified framework say as subbundles of the Pontryagin bundle
\[
{T}^{\mathfrak{p}}(M) = TM \oplus T^\ast M.
\]
The notion was intensively developped from the geometric point of view in connection with the study of nonholonomic mechanical systems or mechanical systems with constraints (cf. \cite{JoRa12} and \cite{YJM10}). It was also motivated by the type of structures encountered in Kirchhoff current laws for electric circuits (cf. \cite{YoMa06} and \cite{CDK87}).\\

In this paper, we introduce, in the convenient framework, a weaker notion on the partial Pontryagine bundle
\[
{T}^{\mathfrak{p}}(M) = TM \oplus T^\flat M
\]
where $T^\flat M$ is a weak sub-bundle of the kinematic cotangent bundle $T'M$. 

\section{Dirac structures on finite dimensional manifolds}
\label{__DiracStructuresOnFiniteDimensionalManifolds}

Let $M$ be a manifold of dimension $n$, $TM$ its tangent bundle and $T^\ast M$ its cotangent bundle. 

\subsection{Dirac structures}

The Pontryagin bundle\index{Pontryagin bundle} 
${T}^{\mathfrak{p}}(M) = TM \oplus T^\ast(M)$ is the Whitney sum bundle over $M$, i.e. it is the bundle over the base $M$ where the fibres over the point $x \in M$ equal to $T_x M \oplus T_x^\ast M$. This bundle is equipped with the natural projections
\[
\operatorname{pr}_T : {T}^{\mathfrak{p}}(M) \to TM
\textrm{  et  }
\operatorname{pr}_{T^\ast} : {T}^{\mathfrak{p}}(M) \to T^\ast(M)
\]
as well as
\begin{description}
\item[$\bullet$]
a nondegenerate, symmetric fibrewise bilinear form $<.,.>$ defined for any $x \in M$, any pair $(X_x,Y_x)$ of $T_xM$ and any pair $(\alpha_x,\beta_x)$ of $T_x^\ast M$ by
\begin{eqnarray}
\label{eq_DirectSymmetricBilinearFormDiracStructures}
<(X_x,\alpha_x),(Y_x,\beta_x)>
= 
\beta_x(X_x) + \alpha_x(Y_x)
\end{eqnarray}
\item[$\bullet$]
the \emph{Courant bracket}\index{Courant bracket}\index{bracket!Courant} defined for any pair $(X,Y)$ of vector fields and any pair $(\alpha,\beta)$ of $1$-forms by
\begin{eqnarray}
\label{eq_CourantBracketDiracStructures}
[[
(X,\alpha),(Y,\beta)
]]
=
\left( 
[X,Y], 
L_X \beta - L_Y \alpha 
+\dfrac{1}{2} d \left( \alpha(Y) - \beta(X) \right) 
\right)
\end{eqnarray}
\end{description}
The Courant bracket satisfies the identity:
\[
\circlearrowleft [[[[a1, a2]], a3]]  
=
-
\circlearrowleft \dfrac{1}{3} d( < [[a1, a2]], a3 > )
\]
where $\circlearrowleft$ stands for the sum of cyclic permutations.

\begin{notation}
\label{N_OrthogonalDiracStructure}
For a vector subbundle $D$ of ${T}^{\mathfrak{p}}M$, we denote $D^\perp$ the orthogonal of $D$ with respect to $<.,.>$:
\[
D^\perp 
= 
\{ (U,\eta) \in {T}^{\mathfrak{p}}M :\;
\forall (X,\alpha) \in D, < (X,\alpha),(U,\eta) > =0  \}
\]
\end{notation}
\begin{notation}
\label{N_SetOfSectionsOfASubbundle}
We denote $\Gamma(D)$ the set of smooth sections of a vector subbundle $D$ of ${T}^{\mathfrak{p}}M$.
\end{notation}

\begin{definition}
\label{D_DiracStructures}
An \emph{almost Dirac structure}\index{almost!Dirac structure} on $M$ is a vector sub-bundle $D$ of ${T}^{\mathfrak{p}}M$ satisfying
\begin{description}
\item[\textbf{(ADS)}]
{\hfil 
$D = D^\perp$
}
\end{description}
\end{definition}
Moreover, if the distribution $\Gamma(D)$\index{GammaL@$\Gamma(D)$ (distribution)} fulfills the condition
\begin{description}
\item[\textbf{(DS)}]
{\hfil 
$[[\Gamma(D), \Gamma(D)]] \subset \Gamma(D)$
}
\end{description}
$D$ is a \emph{Dirac structure}\footnote{In \cite{CoWe88}, the condition \textbf{(DS)} is replaced by
\[
\forall 
\left( 
\left( X_1,\alpha_1 \right) , 
\left( X_2,\alpha_2 \right) , 
\left( X_3,\alpha_3 \right) 
\right)
\in \left( \Gamma(D) \right) ^3,\;
\mathbf{\sigma} \left\langle L_{X_1} \alpha_2 , X_3
 \right\rangle = 0 .
\]
}
\index{Dirac structure} on $M$.\\
%\begin{eqnarray}
%\label{eq_BracketDiracStructures}
%[[
%(X,\alpha),(Y,\beta)
%]]
%=
%\left( [X,Y], L_X \beta - i_Y d\alpha \right)
%\end{eqnarray}
Since the pairing $<.,.>$ has split signature, \textbf{(ADS)} is equivalent to
\begin{enumerate}
\item[$\bullet$]
$<.,.>_{|D} = 0$
\item[$\bullet$]
$\operatorname{rank} D = n$ 
\end{enumerate}
For any Lagrangian subbundle $L \subset TM$, the expression
\[
\mathbf{T}_{L} \left( a_1, a_2, a_3 \right) := <[[a_1, a_2]], a_3>
\]
for any $a_1$, $a_2$ and $a_3$ in  $\Gamma(L)$, defines a $3$-tensor called the \emph{Courant tensor}\index{Courant!tensor} of $L$.\\ 
For a  sub-bundle $D$, the condition of involutivity  \textbf{(DS)} is equivalent to $\mathbf{T}_D = 0$. 

\subsection{Basic examples}
\label{___BasicExamples}

\begin{example}
\label{Ex_GraphOfABiVectorField}
\textsf{Graph of a bivector.} \\
Any bivector field $\underline{P} \in \Gamma \left( \Lambda^2(TM) \right) $ on a finite dimensional manifold defines a Lagrangian subbundle $D_{\underline{P}}$ of ${T}^{\mathfrak{p}}M$ 
\[
D_{\underline{P}} = \left\lbrace  \left( P^\sharp \alpha, \alpha \right), \; \alpha \in T^\ast M \right\rbrace
\]
where $P^{\sharp}$ is a morphism from $T^\ast M$ to $TM$ defined by  $ P^\sharp (\alpha) = \iota_{\alpha} P$ where $P$ is the corresponding mapping from $T^\ast M \times T^\ast M$ to $\mathbb{R}$.\\
The bivector field $\underline{P}$ defines the skew-symmetric bracket 
$ \{f,g\} = \underline{P} (df, dg)$. \\
For any $a_i = \left( P^{\sharp}\left( df_i \right) , df_i \right) $ for $i \in \{1,2,3\}$, we have 
\[
\mathbf{T}_{D_{\underline{P}}}  \left( a_1,a_2,a_3 \right) 
 =\circlearrowleft \{ f_1, \{ f_2,f_3 \} \} .
\]
So $D_{\underline{P}}$ is a Dirac structure if and only if $\underline{P}$ is a Poisson tensor\index{Poisson!tensor}.
\end{example}

\begin{example}
\label{Ex_GraphOfA2Form}
\textsf{Graph of a $2$-form.}\\
Any $2$-form $\underline{\omega}$ on a finite dimensional manifold $M$ defines the Lagrangian subbundle
\[
D_\omega = \left\lbrace \left( X,\omega^{\flat} (X) \right),\; X \in TM \right\rbrace
\]
where $\omega^{\flat} : TM \to T^\ast M$ is the associated morphism to $\underline{\omega}$.\\
For any $a_i = \left( X_i,\omega^\flat X_i \right) $, we have
\[
\mathbf{T}_{D_\omega}  \left( a_1,a_2,a_3 \right) 
= d\omega \left( X_1,X_2,X_3 \right)
\] 
So $D_\omega$ is a Dirac structure if and only if $\omega$ is presymplectic. 
\end{example}

\subsection{Induced Dirac structures on the cotangent bundle}
An important class of Dirac structures is the induced Dirac structure on the cotangent bundle of a finite dimensional manifold $Q$ (cf. \cite{LeOh11}, 2.2). \\
Let $\pi_Q : T^\ast M \to M$ be the cotangent bundle and $\Omega^\flat : T T^\ast M \to T^\ast T^\ast M$ the morphism associated to the Liouville $2$-form $\Omega$ (canonical  symplectic structure on $T^\ast Q$).\\
Given a constant-dimensional distribution $\Delta \subset TQ$, one defines the lifted distribution
\[
\Delta_{T^\ast Q} := \left( T \pi_Q \right) ^{-1} (\Delta) \subset T T^\ast Q
\] 
The annihilator $\Delta_{T^\ast Q}^{\operatorname{o}}$ is also given by $\Delta_{T^\ast Q}^{\operatorname{o}} = \pi_Q^\ast \left( \Delta \right) $.\\
Then the subbundle $\mathcal{D}_\Delta$ of $T T^\ast Q \oplus T^\ast T^\ast Q$ defined by
\[
\mathcal{D}_\Delta 
=
\{
(u,\alpha) \in T T^\ast Q \oplus T^\ast T^\ast Q: \;
u \in \Delta_{T^\ast Q},\; \alpha - \Omega^\flat (u) 
\in \Delta_{T^\ast Q}^{\operatorname{o}}
\}
\]
is a Dirac structure on $T^\ast Q$.

\begin{example}
\label{Ex_VerticalRollingDisk}
\textsf{The vertical rolling disk.}\\
The \emph{vertical rolling disk}\index{vertical rolling disk} on a plane is one of the most fundamental example of mechanical systems with non holonomic constraints described in  \cite{BKMM96} and \cite{Blo03}. It is a homogeneous disk of radius $R$ and mass $m$ rolling without slipping on a horizontal plane.\\
The circle $\mathbb{S}^1$ of radius $1$ is parametrized by an angular variable, i.e. a variable which is $2\pi$-periodic.\\
The configuration space is $Q = \mathbb{R}^2 \times \mathbb{S}^1 \times \mathbb{S}^1$\index{R2S1S1@$\mathbb{R}^2 \times \mathbb{S}^1 \times \mathbb{S}^1$ (configuration space of the vertical rolling disk)} and is parametrized by the (generalized) coordinates $(x,y,\varphi,\theta)$ where $(x,y)$ are the cartesian coordinates of the contact point $M_0$ in the $xy$-plane, $\varphi$ the rotation angle of the disk and $\theta$ the orientation of the disk.

%\begin{figure}[h]
%    \centering
%	 \includegraphics[scale=1]{RollingVerticalDisk.eps}
%    \caption{The vertical rolling disk}
%\end{figure} 

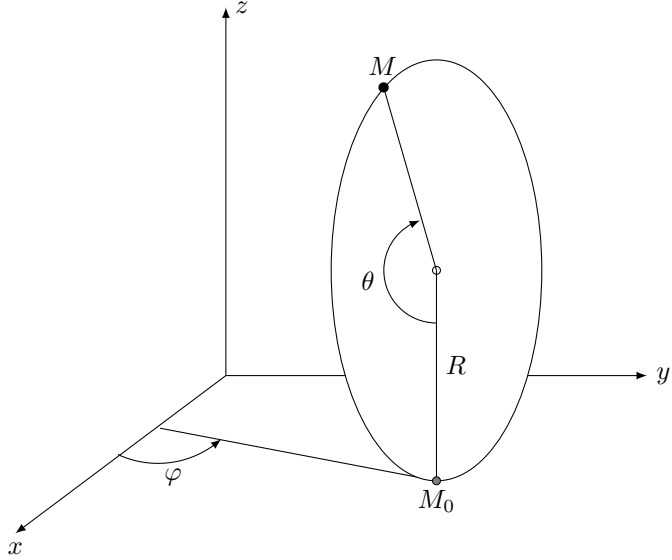
\begin{figure}[h]
    \centering
\begin{tikzpicture}[x=0.7cm,y=0.7cm,
declare function={
f(\t)=2*cos(\t);
g(\t)=4*sin(\t);}
]		
\draw[->,> = latex](-4,-2)--(-8,-5) node[below]{$x$};			\draw[->,> = latex](-4,-2)--(4,-2) node[right]{$y$};		
\draw[->,> = latex](-4,-2)--(-4,5) node[right]{$z$};
%\draw[domain=0:360,smooth,variable=\t]
%  plot ({f(\t)},{g(\t)});
\draw (-5.25,-3) -- (0,-4);
\draw[fill=white] (0:0) ellipse (2 and 4);
\draw (0,0) circle(1.5pt);
\draw ({f(120)},{g(120}) node{$\bullet$};
\draw(-1.02,3.53) node[above]{$M$};
\draw (0,0) -- ({f(120)},{g(120});
\draw (0,0) -- (0,-4);
\draw[fill=gray] (0,-4) circle(1.5pt);
\draw(0,-4) node[below]{$M_0$};
\draw(0,-1.8) node[right]{$R$};
\draw[->,> = latex](0,-1) arc(-90:-252:1);
\draw(-1,-0.2) node[left]{$\theta$};
\draw[->,> = latex](-6.05,-3.5) arc(-115:-48:1.8);
\draw(-5,-3.6) node[below]{$\varphi$};
\end{tikzpicture}
\caption{The vertical rolling disk}
\end{figure}

%\begin{center}
%\begin{pspicture}[linewidth=0.8pt](-2,-2)(5,4)
%\psline{->}(0,0)(-2,-2)
%\psline{->}(0,0)(5,0)
%\psline{->}(0,0)(0,4)
%\rput[tl](-2,-1.6){$x$}
%\rput[tc](4.8,0.22){$y$}
%\rput[tl](-0.3,3.9){$z$}
%\psline[linewidth=0.5pt](-0.5,-0.5)(2.45,-1)
%\psellipse[linewidth=1.2pt,fillcolor=white,fillstyle=solid](2.5,1)(1,2)
%\psline[linewidth=0.6pt]{-}(2.5,1)(2.5,-1)
%\psline[linewidth=0.6pt]{o-*}(2.5,1)(2,2.67)
%\psarc[linewidth=0.6pt]{->}(-0.5,-0.5){0.8}{227}{350}
%\psarcn[linewidth=0.6pt]{->}(2.5,1){0.4}{270}{110}
%\rput[cr](2.7,0){$R$}
%\rput[tl](-0.6,-1.5){$\varphi$}
%\rput[tl](1.72,1.05){$\theta$}
%\rput[tl](1.72,3.05){$M$}
%\rput[bc](2.5,-1.25){$M_0$}
%\end{pspicture}
%\end{center}

Let $(q,v)= \left( x,y,\theta,\varphi,v_x,v_y,v_{\theta},v_{\varphi} \right) $ be the local coordinates for the tangent bundle $TQ$.\\
The standard free Lagrangian without constraints is 
\[
L \left( x,y,\theta,\varphi,v_x,v_y,v_{\theta},v_{\varphi} \right) 
=
\dfrac{1}{2}m \left( v_x^2 + v_y^2 \right)
+ \dfrac{1}{2} I v_{\theta}^2
+ \dfrac{1}{2} J v_{\varphi}^2
\]
where $I$ is the moment of inertia of the disk about the axis perpendicular to the plane of the disk, and $J$ is the moment
of inertia about an axis in the plane of the disk (both axes passing through the center of the disk).\\
The \emph{non holonomic constraints}\index{non holonomic constraints}\footnote{See \cite{Blo03}, 1.3, for the notions of holonomic and non holonomic constraints.} 
due to the rolling contact without slipping on the plane are
\[
\left\{
\begin{array}[c]{rcl}
\dot{x} &=& R \cos \varphi \; \dot{\theta}  \\
\dot{y} &=& R \sin \varphi \; \dot{\theta}
\end{array}
\right.
\]
The constraint may be represented by the distribution $\Delta \subset TQ$ given for each $q \in Q$ by
\[
\Delta = \{ v_q \in TQ, <\omega^i,v_q>=0,\; i \in \{1,2\} \}
\]
where the $1$ forms $\omega^i$ are given by
\[
\left\{
\begin{array}[c]{rcl}
\omega^1 &=& dx - R \cos \varphi \; d\theta  \\
\omega^2 &=& dy - R \sin \varphi \; d\theta
\end{array}
\right.
\]
This distribution is locally spanned by
\[
\dfrac{\partial}{\partial \varphi}
\textrm{~~~and~~~}
R \cos \varphi \dfrac{\partial}{\partial x}
+ R \sin \varphi \dfrac{\partial}{\partial t}
+\dfrac{\partial}{\partial \theta}.
\]
On can define the distribution 
$\Delta_{T^\ast Q}^\ast = \Omega^\flat \left( \Delta_{T^\ast Q} \right) $  where  $\Delta_{T^\ast Q}$ is the lifted distribution of $\Delta$ on $T^\ast Q$.\\
According to \cite{YoMa06}, Proposition~5.3, the induced Dirac structure on $T^\ast Q$ is given, for each $(q,p) \in T^\ast Q$, by the set $\mathcal{D}_\Delta$ of elements  $ \left( v_{(q,p)},\alpha_{(q,p)} \right) $ 
in $ T_{(q,p)} T^\ast Q \times T_{(q,p)}^\ast T^\ast Q$ such that 
\[
\left\{
\begin{array}
[c]{c}
\alpha_{(q,p)} \in \Delta_{T^\ast Q}^\ast  \\
v_{(q,p)} - P^\sharp (q,p).\alpha_{(q,p)}
\in \left( \Delta_{T^\ast Q}^\ast \right) ^{\operatorname{o}} (q,p)
\end{array}
\right.
\]
where $P^\sharp$ is the associated morphism to the canonical Poisson structure \\
$P : T^\ast T^\ast Q \times T^\ast T^\ast Q \to \mathbb{R}$.\\
It is given locally, for $z \in T^\ast Q$, by
\[
\mathcal{D}_\Delta(z) =
\left\lbrace
\left( \left( q,p,\dot{q},\dot{p} \right) , 
\left( q,p,\alpha,\omega \right) \right) ,
\left\{
\begin{array}
[c]{c}
\dot{q} \in \Delta(q)  	\\
\omega = \dot{q}	\\
\alpha + \dot{p} \in \Delta^{\operatorname{o}}(q)
\end{array}
\right.
\right\rbrace .
\]
\end{example}

\begin{example}
\label{Ex_LCCircuit}
\textsf{LC circuit.}\\
Let us consider the \emph{LC circuit}\index{LC circuit}\footnote{This LC circuit is an example of \emph{degenerate Lagrangian system with constraints}.} below
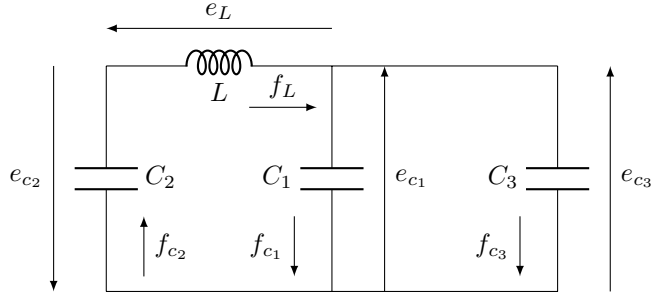
\begin{figure}[h]
    \centering  
\[
\begin{circuitikz}
\draw
(0,0) -- (3,0)
(0,0) to[C, l_=$C_2$] (0,3)
(0,3) to[L, l_=$L$] (3,3)
(3,0) to[C, l^=$C_1$] (3,3)
(6,0) to[C, l^=$C_3$] (6,3)
(3,0) -- (6,0)
(3,3) -- (6,3)
;
\draw[->,>=latex]  (3,3.5) -- (0,3.5)node[midway,above]{$e_L$};
\draw[->,>=latex]  (1.9,2.45) -- (2.8,2.45)node[midway,above]{$f_L$};
\draw[->,>=latex]  (-0.7,3) -- (-0.7,0)node[midway,left=0.5]{$e_{c_2}$};
\draw[->,>=latex]  (0.5,0.2) -- (0.5,1)node[midway,right=0.5]{$f_{c_2}$};
\draw[->,>=latex]  (2.5,1) -- (2.5,0.2)node[midway,left=0.5]{$f_{c_1}$};
\draw[->,>=latex]  (3.7,0) -- (3.7,3)node[midway,right=0.5]{$e_{c_1}$};
\draw[->,>=latex]  (5.5,1) -- (5.5,0.2)node[midway,left=0.5]{$f_{c_3}$};
\draw[->,>=latex]  (6.7,0) -- (6.7,3)node[midway,right=0.5]{$e_{c_3}$};
\end{circuitikz}
\]
\caption{LC circuit}
\end{figure} 
with an inductor $\ell$ and three capacitors $c_1$, $c_2$  and $c_3$ (cf. \cite{VdS98}, \cite{YoMa06}, and \cite{LeOh11}).\\
The configuration space\index{configuration space} is the $4$-dimensional vector space 
$Q = \{ \left( q^\ell,q^{c_1},q^{c_2},q^{c_3} \right) \}$ which corresponds to charges in the elements of the circuit.\\
An element $\left( f^\ell,f^{c_1},f^{c_2},f^{c_3} \right)$ of the tangent space $T_q Q$ corresponds to the currents in the different elements.\\
The Kirchhof current law imposes the constraints
\[
\left\{
\begin{array}
[c]{rcl}
-f^\ell + f^{c_2}	&=&	0  \\
-f^{c_1} + f^{c_2} - f^{c_3}	&=& 0
\end{array}
\right.
\]
This defines the constraint distribution $\Delta$ by
\[
\Delta = \{ f \in TQ: \; 
\forall i \in \{1,2\}, <\omega^i,f> = 0 
\}
\] 
where the constraint $1$-forms $\omega^1$ and $\omega^2$ are given by
\[
\left\{
\begin{array}
[c]{rcl}
\omega^1 	&=&	-dq^\ell + dq^{c_2}  \\
\omega^2	&=&	-dq^{c_1} + dq^{c_2} - dq^{c_3}	
\end{array}
\right.
\]
The annihilator distribution is
\[
\Delta^{\operatorname{o}}  
= \{ e \in T^\ast Q: \;
\forall f \in \Delta, <e,f> = 0 \}
=\operatorname{span} \{ \omega^1,\omega^2 \}.
\]
An element $e$ of $T^\ast Q$ denotes the voltage of the circuit.
In coordinates, an element of $T^\ast Q$ is written as $(q,p)$ with $p = \left( p_\ell,p_{c_1},p_{c_2},p_{c_3} \right) $ and the Liouville form is $\Omega = dq \wedge dp$.\\
The induced Dirac distribution is then
\[
\mathcal{D}_\Delta(q,p) 
=
\{
\left( \dot{q},\dot{p},\alpha_q,\alpha_p \right) \in T T^\ast Q \oplus T^\ast T^\ast Q: \;
\dot{q} \in \Delta, 
\dot{q} = \alpha_p,
\dot{p} + \alpha_q
\in \Delta^{\operatorname{o}}
\}
\]
\end{example}

\subsection{Integrable Dirac structures as Lie algebroids}
\label{___IntegrableDiracStructuresAsLieAlgebroids}

A \emph{Lie algebroid}\index{Lie algebroid} is a finite dimensional smooth vector bundle $\pi : E \to M$  over $M$ with a vector bundle homomorphism $ \rho : E \to TM$, called the \emph{anchor}\index{anchor}, and a Lie algebra bracket $[.,.] : \Gamma(E) \times \Gamma(E) \to \Gamma(E)$ satisfying:
\begin{description}
\item[\textbf{(LA1)}]
$\rho$ is a Lie algebra homomorphism;
\item[\textbf{(LA2)}]
for all $f \in C^\infty(M)$ and all $X$ and $Y$ in $\Gamma(E)$,
\[
[X,fY] = f[X,Y] \left( L_{\rho(X)}f \right) Y.
\]
\end{description}
For a Dirac structure $L$ on $M$, the vector bundle $\pi_L : L \to M$ inherits a Lie algebroid structure with bracket on $\Gamma(L)$ given by the restriction of the Courant bracket and  anchor  given by the restriction of the projection $\operatorname{pr}_T : {T}^{\mathfrak{p}}(M) \to TM$ to $L$ (cf. \cite{JoRa12}, 2.2 and \cite{Bur13}, 4.1). \\

\section{Partial  Dirac structures on a convenient Lie algebroid}
\label{__PartialAlmostDiracStructuresOnConvenientAlgebroid}

The purpose of this section is to show how  we can generalize the notion of Dirac structure from a finite dimension context to an infinite dimensional setting. As in the symplectic framework, we need to use an adapted  notion of partial (almost) Dirac structure. We begin by a linear convenient setting which are needed in the next sections.

\subsection{Linear partial Dirac structure}
\label{___LinearPartialDiracStructure}
 
Consider a  convenient space $\mathbb{E}$ and let $\mathbb{E}^\prime$ be its (bounded)  dual. Given any vector space $\mathbb{E}^\flat$ provided with its own convenient structure such that the inclusion $\mathbb{E}^\flat \to \mathbb{E}^\prime$ is bounded, we set $\mathbb{E}^{\mathfrak{p}}=\mathbb{E}\oplus \mathbb{E}^\flat$ which is called \emph{the partial Pontryagin space associated to $\mathbb{E}$}. As in finite dimension, the canonical pairing $<.,.>$ between $\mathbb{E}^\prime$ and $\mathbb{E}$ induces a pairing (also denoted $<.,.>$) between $\mathbb{E}^\flat$ and $\mathbb{E}$ which is left non degenerate. We define the symmetric pairing $<<.,.>>$ on  $\mathbb{E}^{\mathfrak{p}}$ by
\[
<<(u,\alpha),(v,\beta)>>=<\alpha, v>+<\beta,u>
\]
for all $(u,\alpha)$ and $(v,\beta)$ in $\mathbb{E}^{\mathfrak{p}}$.\\
If $\mathbb{D}$ is a closed vector subspace of $\mathbb{E}^{\mathfrak{p}}$, the \emph{orthogonal} $\mathbb{D}^\perp$\index{Dperp@$\mathbb{D}^\perp$ (orthogonal of $\mathbb{D}$)} of $\mathbb{D}$ is the set 
\[
\{(v,\beta)\in \mathbb{E}\oplus \mathbb{E}^{\flat}:\;
\forall (u,\alpha)\in \mathbb{D},\,
 <<(u,\alpha),(v,\beta)>>=0\; \}.
\]
A \emph{linear  partial  Dirac structure}\index{linear!partial Dirac structure} on $\mathbb{E}$ is a closed subspace $\mathbb{D}$ of $\mathbb{E}^{\mathfrak{p}}$ such that 
$\mathbb{D}=\mathbb{D}^\perp$. When $\mathbb{E}^\flat=\mathbb{E}^\prime$, a linear linear Dirac structure is simply called a \emph{linear Dirac structure}.

\begin{remark}
\label{R_ElementariesRemarks} 
Let $\mathbb{D}$  be  a partial linear Dirac structure  on $\mathbb{E}$.
\begin{enumerate}
\item[1.]
We always have:
\begin{center}
$ \forall (u,\alpha)\in \mathbb{D},\;
<\alpha,u>=0$
\end{center}
since $<\alpha,v>+<\beta, u>=0$ for all $(v,\beta)\in \mathbb{D}$ and, in particular,\\
for $(v,\beta)=(u,\alpha)$.
\item[2.] 
$\mathbb{D}$ is maximal in the following sense:\\ if $\mathbb{D}_1$ is  a partial linear Dirac structure  on $\mathbb{E}$ contained in $\mathbb{D}$, then $\mathbb{D}_1=\mathbb{D}$.\\ Indeed we always have
 $\mathbb{D}^\perp\subset \mathbb{D}_1^\perp=\mathbb{D}_1\subset \mathbb{D}$, the result follows from $\mathbb{D}^\perp=\mathbb{D}$.
\end{enumerate}
\end{remark}

Let $\mathbb{X}$ be any vector subspace of $\mathbb{E}$; the \emph{partial  annihilator}\index{partial!annihilator}\index{X0@$\mathbb{X}^0$ (partial annihilator)} of $\mathbb{X}$  is the subspace $ \mathbb{ X}^0$ of $\mathbb{E}^\flat$ defined by
\[
\mathbb{X}^0=\{\alpha\in \mathbb{E}^\flat:\,
\forall  u\in \mathbb{X}, <\alpha,u>=0 \}.
\]

\begin{conventionnotation}
\label{CN_EEflat}
From now on, when there is no ambiguity, we will consider  $\mathbb{E}$ and $\mathbb{E}^\flat$ as closed spaces of $\mathbb{E}^{\mathfrak{p}}$.
\end{conventionnotation}
 The classical results in finite dimension (cf. \cite{YoMa06}) have, in this context, the following adaptation:

\begin{lemma}
\label{L_PartialAnulatorProperties}${}$
\begin{enumerate}
\item 
If $\mathbb{X}^{\operatorname{ann}}$ is the annihilator   of $\mathbb{X} $  then we have 
\begin{center}
$
\mathbb{X}^0=\mathbb{X}^{\operatorname{ann}}\cap \mathbb{E}^\flat.
$
\end{center}
In particular,  $\mathbb{X}^0$ is a convenient subspace of $\mathbb{E}^\flat$.\\
\noindent  Moreover, we have $(\mathbb{X}^0)^{\operatorname{ann}}\cap \mathbb{E}\subset \mathbb{X}$.

\item Let $\mathbb{X}$ be  a closed subspace of $\mathbb{E}$ and $ Q:\mathbb{E}\to \mathbb{E}^\flat$ be a skew-symmetric bounded linear map (i.e. $<Qu,v>=-<Q v,u>$). If $Q(\mathbb{X})$ is closed in $\mathbb{E}^\flat$, then 
$$\mathbb{D}_Q^\mathbb{X}=\{(u,\alpha)\in \mathbb{X}\times\mathbb{E}^\flat\;:\; \alpha-Q(u)\in \mathbb{X}^0\}$$
is a partial Dirac structure on $\mathbb{E}$. In particular,  $\mathbb{D}_\mathbb{X}:=\mathbb{X}\oplus\mathbb{X}^0$ is a partial Dirac structure on $\mathbb{E}$ and $\mathbb{D}_Q:=\{(u,Q(u)), u\in \mathbb{E}\}$  is a partial Dirac structure on $\mathbb{E}$.

%\item 
%If $\mathbb{X}$ is closed, then $\mathbb{D}_\mathbb{X}:=\mathbb{X}\oplus\mathbb{X}^0$ is a partial Dirac structure on $\mathbb{E}$ if and only if  $(\mathbb{X}^0)^{\operatorname{ann}}\cap \mathbb{E}=\mathbb{X}$.\\
%In particular, this condition is satisfied if $(\mathbb{E}^\flat)^{\operatorname{ann}}\cap \mathbb{E}=\{0\}$.
\item
If $P: \mathbb{E}^\flat \to \mathbb{E}$ is a linear bounded map which is skew symmetric,\\
(i.e. $<\alpha, P \beta>=-<\beta,P \alpha>$), then the graph $\mathbb{D}_P$ of $P$, considered as a subset of $\mathbb{E}\oplus\mathbb{E}^\flat$ is  a partial linear  Dirac structure on $\mathbb{E}$ if  $(\mathbb{E}^\flat)^{\operatorname{ann}}\cap \mathbb{E}=\{0\}$ or if $P$ is surjective.

\item  For $i \in \{1,2\}$, let $\mathbb{D}_i$ be a partial linear Dirac structure relative to the Pontryagine space $\mathbb{E}_i\oplus \mathbb{E}^\flat_i$. If  $<.,.>_i$ denotes the canonical pairing between $\mathbb{E}_i^\prime$ and $\mathbb{E}_i$, the canonical pairing $<.,.>$ between $\mathbb{E}_1^\prime\oplus\mathbb{E}_2^\prime$ and $\mathbb{E}_1\oplus\mathbb{E}_2$ is $<.,.>_1+<.,.>_2$ and  then the set 

$\mathbb{D}=\{ (u_1+u_2,\alpha=p_1^\flat(\alpha_1)+p_2^\flat(\alpha_2)), \;:\; (u_1\alpha_1)\in \mathbb{D}_1,\; (u_2,\alpha_2)\in \mathbb{D}_2\}$

is a  a partial linear Dirac structure relative  to the Pontryagine space
\[
(\mathbb{E}_1\oplus\mathbb{E}_2)\oplus (\mathbb{E}^\flat_1\oplus\mathbb{E}^\flat_2).
\]
$\mathbb{D}$ is called \emph{the direct sum of $\mathbb{D}_1$} and is denoted $\mathbb{D}_1\oplus \mathbb{D}_2$.

\end{enumerate}
\end{lemma}

\begin{proof}\\

1. From its definition, we have   $\mathbb{X}^0=\mathbb{X}^{\operatorname{ann}}\cap \mathbb{E}^\flat$.  It is well known that $\mathbb{X}^{\operatorname{ann}}$ is 
closed in $\mathbb{E}^\prime$. The last property is a consequence of the fact that the inclusion of $\mathbb{E}^\flat$ in $\mathbb{E}^\prime$ is bounded and so 
continuous for the $c^\infty$-topology, which implies that (for the $c^\infty$-topology) the inverse image of a closed set is closed in $\mathbb{E}^\flat$.\\
On the other hand, we have 
\[
(\mathbb{X}^0)^{\operatorname{ann}}\cap \mathbb{E}
=
\{ x\in \mathbb{E}:\forall \alpha\in \mathbb{X}^0,\; <\alpha,x>=0 \}
\subset \left( \mathbb{X}^{\operatorname{ann}} \right) ^{\operatorname{ann}}\cap \mathbb{E}=\mathbb{X}
\]  
since $\mathbb{X}$ is closed.\\

2. At first, we have $\mathbb{D}_Q^\mathbb{X}=\mathbb{X}\oplus( Q(\mathbb{X} )+\mathbb{X}^0)$. If $Q(\mathbb{X})$ is closed in $\mathbb{E}^\flat$, according to Point 1,  so is  $Q(\mathbb{X} )+\mathbb{X}^0$. Therefore,  $\mathbb{D}_Q^\mathbb{X}$ is a closed subspace of $\mathbb{E}^\mathfrak{p}$.\\

Now,  we have
\begin{center}
$
(\mathbb{D}_Q^\mathbb{X})^\perp
=\{(w,\gamma)\in \mathbb{E}^{\mathfrak{p}}:\,
\forall u\in \mathbb{X},\; \forall \alpha-Q(u)\in\mathbb{X}^0:\;
<\alpha,w>+<\gamma,u>=0\}.
$
\end{center}
First we show that  $\mathbb{D}_Q^\mathbb{X}\subset ( \mathbb{D}_Q^\mathbb{X})^\perp$.  Let $(w,\gamma)\in\mathbb{D}_Q^\mathbb{X}$. Then we have:
$$<\alpha,w>+<\gamma,u>=<Q(u),w>+<Q(w),u>=0$$
 from the skew symmetry of $Q$. This proves the announced inclusion.

\noindent For the inverse  inclusion,  consider  $(w,\gamma)\in(\mathbb{D}_Q^\mathbb{X})^\perp$. By assumption we have
\begin{equation}
\label{alpha1}
<\alpha,w>+<\gamma,u>=0
\end{equation}
for all $(u,\alpha)\in \mathbb{D}_Q^\mathbb{X}$. But, for any $(u,\alpha)\in \mathbb{D}_Q^\mathbb{X}$, we have 
\begin{equation}
\label{alpha2}
<\alpha,w>=<Q(u),w>
\end{equation}
for all  $w\in \mathbb{E}.$\\
 
At first, choose $u=0$. Then since $Q(u)=0$ for all $\alpha\in \mathbb{X}^0$ we have   $(u,\alpha)$ belongs to $\mathbb{D}_Q^\mathbb{X}$.  From  relation (\ref{alpha1}), it follows that $\alpha(w)=0$ for all $\alpha\in \mathbb{X}^0$ and so $w\in( \mathbb{X}^0)^{\operatorname{ann}}\cap \mathbb{E}$. But,  since    $(\mathbb{X}^0)^{\operatorname{ann}}\cap \mathbb{E}
\subset \mathbb{X}$ then $w$ belongs to $\mathbb{X}$.  Now taking in account relation (\ref{alpha2}), since $w$ belongs to $\mathbb{X}$, relation (\ref{alpha1}) implies:

$<Q(u),w>+<\gamma,u>=0$ for all $u\in \mathbb{X}$. 
In other words 

$<\gamma,u>=<Q(w), u>$ from the skew symmetry of $Q$. 

\noindent Therefore $(w,\gamma)$ belongs to $\mathbb{D}_Q^\mathbb{X}$, which ends the proof  that $\mathbb{D}_Q^\mathbb{X}$ is a  partial linear Dirac structure on 
$\mathbb{E}$.

For $Q=0$, we have $\mathbb{D}_Q^\mathbb{X}=\mathbb{D}_\mathbb{X}$ and for $\mathbb{X}=\mathbb{E}$, we have $\mathbb{D}_Q^\mathbb{X}=\mathbb{D}_Q$, which ends the proof.\\

3. Assume that $P:\mathbb{E}^\flat \to \mathbb{E}$ is a  linear bounded map which is skew symmetric (i.e. $<\alpha, P(\beta)>=-<\beta, P(\alpha)>$). The graph $\mathbb{D}_P$ of $P$,  as a subset of    $\mathbb{E}\oplus\mathbb{E}^\flat$, is $\{(u, \alpha)\in \mathbb{E}\times 
\mathbb{E}^\flat\;:\; u=P(\alpha)\}$.\\ 
Now since 
\begin{center}
$\forall \alpha \in \mathbb{E}^\flat,\;
<\alpha,P(\beta)>+<\beta, P(\alpha)>=0$,
\end{center}
it follows that  $\mathbb{D}_P\subset  \mathbb{D}_P^\perp$. \\
Let $(w,\beta)\in   \mathbb{D}_P^\perp$.
  Then $<\beta,P(\alpha)>=-<\alpha,w>$ for all $ \alpha\in \mathbb{E}^\flat$. Thus $<\alpha,P(\beta)>=<\alpha,w>$ for all all $\alpha\in 
\mathbb{E}^\flat$, which implies that 
$P\beta-w$ belongs to $(\mathbb{E}^\flat)^{\operatorname{ann}}\cap \mathbb{E}$.\\
If $(\mathbb{E}^\flat)^{\operatorname{ann}}\cap \mathbb{E}=\{0\}$, this implies that  $\mathbb{E}^\flat$ separates points, and so it follows that that $P\beta=w$.\\
If $P$ is surjective, then there exists $\beta^\prime\in \mathbb{E}^\flat$ such that $P\beta^\prime=w$. It follows that  $<\alpha, P(\beta-\beta^\prime)>=0=<\beta-\beta', P\alpha>$ for all $\alpha\in \mathbb{E}^\flat$. Since $P$ is surjective and the bilinear form $<.,.>$  is left non degenerate, we obtain $\beta-\beta'=0$, which ends the proof of Point 3.\\

4. Let $p_i$ and $p_i^\flat$ (resp. $p$ and $p^\flat$) the canonical projection of $\mathbb{E}_i\oplus \mathbb{E}^\flat_i$ over $\mathbb{E}_i$ and $\mathbb{E}_i^\flat$ (resp. $(\mathbb{E}_1\oplus\mathbb{E}_2)\oplus (\mathbb{E}^\flat_1\oplus\mathbb{E}^\flat_2)$ over  $(\mathbb{E}_1\oplus\mathbb{E}_2)$ and $(\mathbb{E}^\flat_1\oplus\mathbb{E}^\flat_2)$ ) respectively. Then we have $p=p_1+p_2$ and $p^\flat=p_1^\flat+p_2^\flat$. Consider any  $(u_1+u_2, \alpha)$ and $(v_1+v_2,\beta)$ in $\mathbb{D}$ and assume that $\alpha=p^\flat(\alpha_1)+p_2^\flat(\alpha_2)$ and $\beta=p_1^\flat(\beta_1)+p_2^\flat(\beta_2)$ where $(u_i,\alpha_i)$ and $(v_i,\beta_i)$ belongs to $\mathbb{D}_i$ for $i \in \{1,2\}$. Then we have 

$\begin{matrix}
<<(u_1+u_2, \alpha),(v_1+v_2,\beta)>>&=\beta(u_1+u_2)+\alpha((v_1+v_2)\hfill{}\\
&=\beta_1(u_1)+\beta_2(u_2)+\alpha_1(v_1)+\alpha_2(v_2)\hfill{}\\
&=0\hfill{}\\
\end{matrix}
$

 \noindent since  $(u_i,\alpha_i)$ and $(v_i,\beta_i)$ belongs to $\mathbb{D}_i$ and $\mathbb{D}_i$ is a partial linear Dirac structure.\\
 
\noindent  Consider now $(w,\gamma)\in (\mathbb{E}_1\oplus\mathbb{E}_2)\oplus (\mathbb{E}^\flat_1\oplus\mathbb{E}^\flat_2)$. We have a decomposition
 
 $(w,\gamma)=(w_1,\gamma_1)+(w_2,\gamma_2)$ where $(w_i,\gamma_i)\in \mathbb{E}_i\oplus \mathbb{E}^\flat_i$. Assume that  
 
 $<<(w,\gamma), (u_1+u_2,\alpha)>>=0$ for all $(u_1+u_2,\alpha)\in \mathbb{D}$ 
 
 \noindent and, for such an $\alpha$, let  $\alpha_i$ be  such that $(u_i,\alpha_i)\in \mathbb{D}_i$ and   $\alpha=p^\flat(\alpha_1)+p^\flat(\alpha_2)$. 
 
 \noindent Thus we have
 
 $\gamma_1(u_1)+\gamma_2(u_2)+\alpha_1(v_1)+\alpha_2(v_2)=0$ for all $ (u_i,\alpha_i)\in \mathbb{D}_i.$ 
 
 \noindent In particular, for  $(u_2,\alpha_2)=(0,0)\in \mathbb{D}_2$, we obtain
 
  $\gamma_1(u_1)+\alpha_1(w_1)=0$ for all  $(v_1,\alpha_1)\in \mathbb{D}_1$.
  
  \noindent This implies that $(w_1,\gamma_1)$ belongs to $\mathbb{D}_1$. By the same argument, we get $(w_2,\gamma_2)$ belongs to $\mathbb{D}_2$ and so $\mathbb{d}$ is a partial linear Dirac structure on $\mathbb{E}_1\oplus\mathbb{E}_2$.
\end{proof}

\begin{remark}
\label{R_ContreExamples}${}$
\begin{enumerate}
\item 
If $\mathbb{E}^\flat =\mathbb{E}^\prime$, Lemma \ref{L_PartialAnulatorProperties}, 3. is always true since, in this case, 
$\mathbb{X}^0=\mathbb{X}^{\operatorname{ann}} $ and $(\mathbb{X}^{\operatorname{ann}})^{\operatorname{ann}}=\mathbb{X}$, which is a natural generalization of  classical results obtained in finite dimension (cf.
\cite{YoMa06}) and \cite{Cou90}). 
%Note that in this case, the conclusion of Lemma~\ref{L_PartialAnulatorProperties}, 4. is also true of course. 

Thus we obtain a natural generalization of  classical results in finite dimension (cf.
\cite{YoMa06}) and \cite{Cou90}).\\
However, when $\mathbb{E}^\flat \not=\mathbb{E}^\prime$, if the assumptions of Lemma~\ref{L_PartialAnulatorProperties}, 3.  are not satisfied, the result is not true in general as the elementary following examples show.\\
In  the Banach space $c_0$ of real sequences which converge to $0$, provided with its canonical basis 
$\{e_n, n\in \mathbb{N}\}$, consider the set $\mathbb{X}$  generated by $\{e_n, n>1\}$. Then  $\ell^1$ will be identified with  the dual of $c_0$ and provided with its canonical basis $\{e_n^*, n\in \mathbb{N}\}$.  We denote by  $c_0^\flat$ the subspace of $\ell^1$  generated by $\{e_n^*, n>1\}$. Therefore $(c_0^\flat)^{\operatorname{ann}}\cap c_0$  is the vector space generated by $e_1$. Take $P:c_0^\flat \to c_0$ defined by 
\begin{center}
$
\forall n>1,\,P(e_n^*)=-e_n.
$
\end{center}
Then  the space $ \mathbb{D}^\perp_P$ contains the pair $(e_1, e^*_2)$ which does not belong to  $ \mathbb{D}_P$.

\item  
The density of $\mathbb{E}^\flat$ in $\mathbb{E}^\prime$ implies the condition $(\mathbb{E}^\flat)^{\operatorname{ann}}\cap \mathbb{E}=\{0\}$ and so, in such a situation, 
each  conclusion of Lemma~ \ref{L_PartialAnulatorProperties} is  true. More generally, we have the same result  {\rm if $\mathbb{E}^\flat$ separates points in $\mathbb{E}$.
\item 
 We discus the relation between the condition $(\mathbb{E}^\flat)^{\operatorname{ann}}\cap \mathbb{E}=\{0\}$,\\
the density of $\mathbb{E}^\flat$ in $\mathbb{E}^\prime$ and the property "$\mathbb{E}^\flat$ separates points in $\mathbb{E}$".}
\begin{description}
\item (i) The condition $(\mathbb{E}^\flat)^{\operatorname{ann}}\cap \mathbb{E}=\{0\}$  does not imply the density of  $\mathbb{E}^\flat$ in $\mathbb{E}^\prime$ 
 as illustrated by the following example:\\
in $\mathbb{E}=\ell^1$ take $\mathbb{E}^\flat =c_0\subset \ell^\infty$. Then $\mathbb{E}^\flat\not=\mathbb{E}$ is closed  in $\mathbb{E}$ and therefore not dense (for the topology of 
the norm). 
\item (ii)The assumption  $(\mathbb{E}^\flat)^{\operatorname{ann}}\cap \mathbb{E}=\{0\} $ is equivalent to  the assumption that $\left((\mathbb{E}^\flat)^{\operatorname{ann}}\cap \mathbb{E}\right)^{\operatorname{ann}}=\mathbb{E}^\prime $ which (as a consequence of the definition of a convenient space)  implies that  $\mathbb{E}^\flat$ separates points of  $\mathbb{E}$.\\
The converse is true for a Fr\'echet vector space $\mathbb{E}$ since its convenient dual $\mathbb{E}^\prime $ is equal to its topological dual $\mathbb{E}^*$ and it is a consequence of the Hahn-Banach Theorem and the fact that for a subset $\mathbb{F}$ of $\mathbb{E}$, then  
$\left((\mathbb{F}^{\operatorname{ann}}\cap \mathbb{E}\right)^{\operatorname{ann}}$ is  the weak-* closure of $\mathbb{F}$ in $\mathbb{E}^*$.
Unfortunately,  the Hahn-Banach Theorem is not true, in general,  in the convenient dual of a convenient space (cf. \cite{KrMi97}) and then, if $\mathbb{E}^\flat$ separates points of $
\mathbb{E}$, this could  not imply that  $(\mathbb{E}^\flat)^{\operatorname{ann}}\cap \mathbb{E}=\{0\}$.
\end{description}
\end{enumerate}
\end{remark}

We denote by $p:\mathbb{E}^{\mathfrak{p}}\to \mathbb{E}$ and $p^\flat: \mathbb{E}^{\mathfrak{p}} \to \mathbb{E}^\flat$ the natural projections. Given a  closed subspace  $\mathbb{D}$ of  $\mathbb{E}^{\mathfrak{p}}$, 
 we have $\ker  p_{| \mathbb{D}}= \mathbb{D}\cap \mathbb{E}^\flat$  and $\ker  p_{| \mathbb{D}}^\flat= \mathbb{D}\cap \mathbb{E}$. \\

From now on,  we assume that  $\mathbb{D}$ is a partial linear Dirac structure and we will use systematically Conventions and Notations~\ref{CN_EEflat}.
%\noindent Then we have $p(\mathbb{D})=\mathbb{D}\cap\mathbb{E}$ and   $p^\flat(\mathbb{D})= \mathbb{D}\cap \mathbb{E}^\flat$. In particular 
%we have  
%%\end{equation}
% Since $p$ and $p^\flat$ are bonded linear map and $\mathbb{E}^\frak{p}$ is a convenient space,  $\ker p$ and $\ker p^\flat$   are closed subspaces of $\mathbb{E}^{\mathfrak{p}}$. But by assumption $\mathbb{D}$  a closed so are  $\ker p_{| \mathbb{D}}$ and  $\ker p_{| \mathbb{D}}^\flat$. It follows that $\mathbb{D}/(\mathbb{D}\cap \mathbb{E})$ and $\mathbb{D}/(\mathbb{D}\cap \mathbb{E}^\flat)$ can be provided with a  quotient convenient space structure so that $p$ and $p^\flat$ induces respectively  a convenient  ismorphism :\\
%${}\;\;\;\;\;\; \hat{p}_{| \mathbb{D}}: \mathbb{D}/(\mathbb{D}\cap \mathbb{E}^\flat)\to  p(\mathbb{D})$ and  $\hat{{p}}^\flat_{| \mathbb{D}}: \mathbb{D}/(\mathbb{D}\cap \mathbb{E})\to  p^\flat (\mathbb{D})$ respectively. \\

%\begin{remark}\label{R_QuotientIsomorphisms}${}$
%\begin{enumerate}
%\item[1.] 
%\emph{Note that $\hat{p}_{| \mathbb{D}}$ and $\hat{{p}}^\flat_{| \mathbb{D}}$  are convenient isomorphisms  but not necessary homeomorphism for the locally convex quotient structure.} %(cf. for example \cite{No05}).\\%Bornological quotients  Acad. Roy. BelgPublisher: Académie Royale de Belgique 2005}\\
%\item[2.] 

\begin{lemma}
\label{L_PDPflatDClosed}${}$
\begin{enumerate}
\item 
$p(\mathbb{D})$ and $p^\flat(\mathbb{D})$ are closed in $\mathbb{E}$ and $\mathbb{E}^\flat$ respectively. %and $\mathbb{D}=p(\mathbb{D})\times p^\flat(\mathbb{D})$
\item  
Let $\varpi:\mathbb{D}\to  \mathbb{D}/(\mathbb{D}\cap \mathbb{E}^\flat)$ and $\varpi^\flat:\mathbb{D}\to \mathbb{D}/(\mathbb{D}\cap \mathbb{E})$ be the projections on the quotient spaces.\\
Then $p$ and $p^\flat$ induces convenient  isomorphisms
\begin{center}
$ \hat{p}_{| \mathbb{D}}: \mathbb{D}/(\mathbb{D}\cap \mathbb{E}^\flat)\to  p(\mathbb{D})$ \text{~~and~~}  $\hat{{p}}^\flat_{| \mathbb{D}}: \mathbb{D}/(\mathbb{D}\cap \mathbb{E})\to  p^\flat (\mathbb{D})$ 
\end{center}
such that $p_{| \mathbb{D}}= \hat{p}_{| \mathbb{D}}\circ \varpi$ and $p_{| \mathbb{D}}^\flat= \hat{p}_{| \mathbb{D}}^\flat\circ \varpi^\flat$ respectively.
%\item $p(\mathbb{D})^0=p^\flat(\mathbb{D})=(\mathbb{E}\cap \mathbb{E})^0=\mathbb{D}\cap \mathbb{E}^\flat$
 \end{enumerate}
\end{lemma}

\begin{remark}
\label{R_QuotientIsomorphisms}${}$
Note that $\hat{p}_{| \mathbb{D}}$ and $\hat{{p}}^\flat_{| \mathbb{D}}$  are convenient isomorphisms  but not necessary homeomorphisms for the locally convex quotient structure. 
%(cf. for example \cite{No05}).\\%Bornological quotients  Acad. Roy. BelgPublisher: Académie Royale de Belgique 2005}\\
\end{remark} 

\begin{proof}${}$\\ 
1. 
Since $p$ and $p^\flat$ are canonical projections of $\mathbb{E}\times\mathbb{E}^\flat$ on $\mathbb{E}$ and $\mathbb{E}^\flat$ respectively,  and $\mathbb{D}$ is closed in $\mathbb{E}^{\mathfrak{p}}$, it follows that their ranges by these maps are also closed in $\mathbb{E}$ and $\mathbb{E}^\flat$ respectively.\\
2.  Since $p$ and $p^\flat$ are bounded linear maps and $\mathbb{E}^\frak{p}$ is a convenient space,  $\ker p$ and $\ker p^\flat$   are closed subspaces of $\mathbb{E}^{\mathfrak{p}}$. % But by assumption $\mathbb{D}$  is closed so  its intersection  $\ker p_{| \mathbb{D}}$ and  $\ker p_{| \mathbb{D}}^\flat$ and also $p(\mathbb{D}$ and $p^\flat(\mathbb{D}$ from Lemma \ref {L_DE0=DEflat}.\\
It follows that $\mathbb{D}/(\mathbb{D}\cap \mathbb{E})$ and $\mathbb{D}/(\mathbb{D}\cap \mathbb{E}^\flat)$ can be provided with  quotient convenient space  structures. Precisely, %if we denote by $\varpi$ (resp. $\varpi^\flat$) the canonical projection of  $\mathbb{D}$ onto $\mathbb{D}/(\mathbb{D}\cap\mathbb{E}^\flat)$ (resp. of  $
%\mathbb{D}$ onto $\mathbb{D}/(\mathbb{D}\cap\mathbb{E})$ ), 
for the quotient convenient structure on   $\mathbb{D}/(\mathbb{D}\cap\mathbb{E}^\flat)$ (resp. $\mathbb{D}/
(\mathbb{D}\cap\mathbb{E})$) the bounded sets are of type $\varpi(B)$ (resp. $\varpi^\flat(B)$) where $B$ is a bounded set in $\mathbb{D}$. \\
Now, from an algebraic point of view, $p$ and $p^\flat$ induce bijective linear maps
\begin{center}
$\hat{p}_{| \mathbb{D}}: \mathbb{D}/(\mathbb{D}\cap \mathbb{E}^\flat)\to  p(\mathbb{D})$ and  $\hat{{p}}^\flat_{| \mathbb{D}}: \mathbb{D}/(\mathbb{D}\cap \mathbb{E})\to  p^\flat (\mathbb{D})$ respectively
\end{center} 
such that 
\begin{equation}
\label{eq_quotientrelations}
\hat{p}_{| \mathbb{D}}\circ \varpi={p}_{| \mathbb{D}}\textrm{ and } \hat{p}_{| \mathbb{D}}^\flat \circ \varpi^\flat={p}_{| \mathbb{D}}^\flat.
\end{equation} 
But, since $\mathbb{D}$ is closed in  $\mathbb{E}\times\mathbb{E}^\flat$, $B$ is bounded in $\mathbb{D}$ if and only if $B$ is bounded in  $\mathbb{E}\times\mathbb{E}^\flat$. As ${p}
_{| \mathbb{D}}$ and ${p}_{| \mathbb{D}}^\flat$ are bounded maps, it follows easily from (\ref{eq_quotientrelations}) that  $\hat{p}_{| \mathbb{D}}$ and  $\hat{p}_{| \mathbb{D}^\flat}$ 
are bounded as well as their inverse maps.
\end{proof}
 
\begin{lemma}
\label{L_DE0=DEflat}${}$
\begin{enumerate}
\item 
$\mathbb{D}\cap \mathbb{E}=\ker p^\flat_{| \mathbb{D}}\subset p(\mathbb{D})$ 
\text{~~and~~} $\mathbb{D}\cap \mathbb{E}^\flat=\ker p_{|\mathbb{D}}\subset p^\flat(D)$.
\item  
$\mathbb{D}\cap \mathbb{E}^\flat =p(\mathbb{D})^0$.
\item  $(p^\flat(\mathbb{D}))^{\operatorname{ann}}\cap \mathbb{E}=(\mathbb{D}\cap \mathbb{E})$ and so $p^\flat(\mathbb{D})=(\mathbb{D}\cap \mathbb{E})^0$.
%\item Assume that $\mathbb{D}\cap \mathbb{E}$ is complemented in $\mathbb{E}$.  If $\mathbb{H}$ is a  such a complement of then $(\mathbb{D}\cap \mathbb{E})^\flat:=\mathbb{H}^0$ is isomorphic to the convenient space $\mathbb{E}^\flat/(\mathbb{D}\cap\mathbb{E}^0)$ and  then $p^\flat (\mathbb{D}$ is isomorphic to $\mathbb{E}^\flat/ ((\mathbb{H}^0)
\end{enumerate}
\end{lemma}

\begin{proof} ${}$\\
1. $v$ belongs to $ \mathbb{D}\cap \mathbb{E}$ if and only if  $(v,0)$ belongs to $\mathbb{D}$ and so $p(v,0)=v$ belongs to the range of $ p_{| \mathbb{D}}$. The same arguments work for $ p^\flat_{| \mathbb{D}}$.\\

\noindent 2. For Point 2,  we have the following equivalences:\\
\noindent(1)   $ u$ belongs to $p(\mathbb{D})$ if and only if there exists $\alpha\in \mathbb{E}^\flat$ such that $(u,\alpha)$ belongs to $\mathbb{D}$;

\noindent (2) $\beta$ belongs to $(\mathbb{D}\cap \mathbb{E}^\flat)$ if and only if $(0,\beta)\in \mathbb{D}$;\\
thus, from the assertion  (1), $\beta$ belongs to $p(\mathbb{D})^0$ 

if and only if $0=\beta(u)=<< (0,\beta), (u,\alpha)>> $ for all $(u,\alpha)\in \mathbb{D}$,
 
  if and only if  $(0,\beta)$ belongs to $\mathbb{D}^\perp=\mathbb{D}$ since $\mathbb{D}$ is a partial linear Dirac structure. 
  
  This last condition is equivalent to $\beta\in \mathbb{D}\cap \mathbb{E}^\flat$ from the assertion (2).\\
 
\noindent 3. For Point 3., we also have  the following equivalences:\\

\noindent (1') $\beta$ belongs to $p^\flat(\mathbb{D})$ if and only if there exists $v\in \mathbb{E}$ such that $(v,\beta)\in \mathbb{D}$;

%\noindent (2') $\beta$ belongs to $(\mathbb{D}\cap \mathbb{E})^0$ if and only if $\beta(u)=0$ for all $u\in \mathbb{D}\cap \mathbb{E}$;

\noindent  (2') $v$ belongs to  $(p^\flat(\mathbb{D}))^{\operatorname{ann}}\cap \mathbb{E}$ if and only if $\alpha(v)=0$ for all $\alpha\in  p^\flat(\mathbb{D})$;

\noindent(3')  $v$ belongs to $\mathbb{D}\cap\mathbb{E}$ if and only if $(v,0)$ belongs to $\mathbb{D}$;\\

Thus, in the one hand,  as in the previous proof, we have

 from assertion (2'), $v$ belongs to  $(p^\flat(\mathbb{D}))^{\operatorname{ann}}\cap \mathbb{E}$ 

if and only if $0=\alpha(v)=<<(v,0),(u,\alpha)>>$ for all $(u,\alpha)\in \mathbb{D}$, 

if and only if  $(v,0) $ belongs to $\mathbb{D}^\perp=\mathbb{D}$, 
 
 which is equivalent to $v$ belongs to $\mathbb{D}\cap\mathbb{E}$ from assertion (3').\\

On the other hand, from the  relation $(p^\flat(\mathbb{D}))^{\operatorname{ann}}\cap \mathbb{E}=\mathbb{D}\cap \mathbb{E}$, we obtain:
 
\[
  \begin{matrix}
  (\mathbb{D}\cap \mathbb{E})^0&=(\mathbb{D}\cap \mathbb{E})^{\operatorname{ann}}\cap \mathbb{E}^\flat\hfill{}\\
  &=\left(((p^\flat(\mathbb{D}))^{\operatorname{ann}}\cap \mathbb{E})\right)^{\operatorname{ann}}\cap \mathbb{E}^\flat\hfill{}\\
  &=\left((((p^\flat(\mathbb{D}))^{\operatorname{ann}})^{\operatorname{ann}}+ \mathbb{E}^{\operatorname{ann}}\right)\cap \mathbb{E}^\flat\hfill{}\\
  &=\left((((p^\flat(\mathbb{D}))^{\operatorname{ann}})^{\operatorname{ann}}\right)\cap \mathbb{E}^\flat\hfill{}
\end{matrix}
\]
Now, $(((p^\flat(\mathbb{D}))^{\operatorname{ann}})^{\operatorname{ann}}$ is the closure (for the $c^\infty$ topology) $\operatorname{Cl}(p^\flat(\mathbb{D}))$ of $p^\flat(\mathbb{D})$ in $\mathbb{E}^\prime$ (for the $c^\infty$ topology). But since $p^\flat(\mathbb{D})$ is closed in $\mathbb{E}^\flat$ and the inclusion of $\mathbb{E}^\flat$ in  $\mathbb{E}$ is bounded (and so continuous for the $c^\infty$ topology), it follows that 
\[
\left((((p^\flat(\mathbb{D}))^{\operatorname{ann}})^{\operatorname{ann}}\right)\cap \mathbb{E}^\flat=p^\flat(\mathbb{D})
\]
which ends the proof.
%2. If $(u,\alpha)$ (resp. $(v,\beta)$) belongs to $\mathbb{D}\cap \mathbb{E}$ (resp. $\mathbb{D}\cap \mathbb{E}^\flat$) then we have $\alpha=0$ (resp. $v=0$). Therefore  let $\beta\in \mathbb{D}\cap \mathbb{E}^\flat$.  Since $\mathbb{D}^\perp=\mathbb{D}$, we have 
%$$0=<<(u,0),(0,\beta)>>=\beta(u)$$ 
%for all $u\in \mathbb{D}\cap \mathbb{E}$  and so $\beta$ belongs to $(\mathbb{D}\cap \mathbb{E})^0$. Conversely, if $(v,\beta)$ belongs to $(\mathbb{D}\cap \mathbb{E})^0$ then $(v,\beta)$ must belongs to $\mathbb{E}^\flat$ and so $v=0$ and we have $\beta(u)=0$ for all $u\in \mathbb{D}\cap \mathbb{E}$. This implies that $(0,\beta)$ belongs to  $\mathbb{D}\cap \mathbb{E}^\flat$.\\
%From Point 1, $\mathbb{D}\cap \mathbb{E}\subset p(\mathbb{D})$, so $p(\mathbb{D})^0\subset (\mathbb{D}\cap \mathbb{E})^0$. On the other hand  $p(\mathbb{D})^0\subset p^\flat (\mathbb{D})$. But if $(\beta,v)$  (resp. $(\alpha,u)$) belongs to $p^\flat(\mathbb{D})$ (resp. $p(\mathbb{D})$) since we must have $v=0$  (resp. $\alpha=0$) we must have 
%$0=<\beta,0>+<0, u>=<\beta,u>$ which implies that $p^\flat(\mathbb{D})=p(\mathbb{D})^0$. By same arguments, $\mathbb{D}\cap \mathbb{E}^\flat\subset p(\mathbb{D}$ which ends the proof.
\end{proof}
  
\begin{lemma}
\label{L_QuotientPairing} 
The pairing $<.,.>$ between $\mathbb{E}^\flat$ and $\mathbb{E}$ induces a well defined pairing (again denoted by $<.,.>$) between $ p^\flat (\mathbb{D})$ and $ p (\mathbb{D})$ which is also left non degenerate, defined by 
\begin{equation}\label{eq_QuotientPairing} 
<p^\flat(v,\beta),p(u,\alpha)>=<\beta, u>=-<\alpha,v>.
\end{equation}
%Moreover, $p^\flat (\mathbb{D})((u,\alpha)_{| p(\mathbb{D}} $ (resp. $p (\mathbb{D})((v,\beta)_{| p^\flat(\mathbb{D}} $) is a well defined $1$-form on $p(\mathbb{D}$  (resp. $p^\flat(\mathbb{D}$) which depends only on the value of  $p^\flat (\mathbb{D})((u,\alpha)$ (resp. $p (\mathbb{D})((v,\beta)$.\\ 
 \end{lemma}
\begin{proof}
Consider  $(u,\alpha), (u',\alpha'), (v,\beta), (v',\beta') \in \mathbb{D}$ such that $ p_{| \mathbb{D}} (u,\alpha)= p_{| \mathbb{D}} (u',\alpha')$  and $ p^\flat_{| \mathbb{D}}(v,\beta)= p^\flat_{| \mathbb{D}} (v',\beta')$. 
Thus we have   
$u-u'\in \mathbb{D}\cap \mathbb{E}$, $\alpha=\alpha'$ and $v=v'$, $\beta-\beta'\in \mathbb{D}\cap \mathbb{E}^\flat$. Since $\mathbb{D}=\mathbb{D}^\perp$, we have 
$<\beta,u>+<\alpha,v>=0$ and $<\beta,u'>+<\alpha,v>=0$ and so $<\beta, u>=<\beta', u>=-<\alpha, v>$. Using the same argument with $(u,\alpha)$  with $(v,\beta)$ and $(v,\beta')$, we obtain  $<\alpha,v>=<\alpha, v' >=-<\beta,u>$.\\
 So, this relation implies that the value   $<p^\flat_{| \mathbb{D}}(v,\beta),p_{| \mathbb{D}}(u,\alpha) >$ does not depend on the choice of $(u,\alpha)$,  $(u',\alpha)$ and $(v,\beta)$.
\end{proof}

Denote by  $\mathbb{L}=p(\mathbb{D})\subset \mathbb{E}$ where and $\mathbb{L}^\flat =p^\flat (\mathbb{D})\subset \mathbb{E}^\flat$.\\
By analogy with the finite dimension, (cf. \cite{Cou90}), we can define:
\begin{enumerate}
	\item[$\bullet$]
	$P_{\mathbb{L}}:\mathbb{L}\to \mathbb{L}^\prime$ by  $P_{\mathbb{L}}\circ p(\delta)=p^\flat(\delta)_{| \mathbb{L}}$ 
	where $\mathbb{L}^\prime$ is the dual of $\mathbb{L}$;  
	\item[$\bullet$]
	$P^\flat _{\mathbb{L}}:\mathbb{L}^\flat \to (\mathbb{L}^\flat)^\prime$ by $P_{\mathbb{L}}^\flat
	\circ p^\flat(\delta)=p(\delta)_{| \mathbb{L}^\flat}$ for all $\delta\in \mathbb{D}$  where$(\mathbb{L}^\flat)^\prime$ is the dual of $\mathbb{L}^\flat$ and where $p(\delta)$ is considered as an element of $\mathbb{E}^{\prime\prime}$ and   so $ p(\delta)_{| \mathbb{L}^\flat}$  means the restriction of such elements of $\mathbb{E}^{\prime\prime}$ to $\mathbb{L}^\flat\subset \mathbb{E}^\prime$.
\end{enumerate} 

\begin{lemma}
	\label{L_LbinterL'Linter Lb'}
	$\mathbb{L}^\flat \cap \mathbb{L}^\prime$ (resp.  $(\mathbb{L}^\flat)^\prime \cap \mathbb{L}$) is a convenient space which is equal to $\mathbb{L}^\flat/(\mathbb{D}\cap \mathbb{E}^\flat)$ (resp. $\mathbb{L}/(\mathbb{D}\cap \mathbb{E})$).
\end{lemma}

\begin{proof}
%At first note that $\mathbb{L}^\prime\cap \mathbb{L}^\flat=\{\alpha_{| \mathbb{L}};\; \alpha\in \mathbb{L}^\flat\}$. 
For $\alpha$ and $\alpha'$ in $\mathbb{L}^\flat$, we have $\alpha(u)=\alpha'(u)$ for each $u\in \mathbb{L}$ if and only if $\alpha-\alpha'=0$ on $\mathbb{L}$, which is equivalent to  $\alpha-\alpha'$ belongs to $\mathbb{L}^0=\mathbb{D}\cap\mathbb{E}^\flat$ (cf. Lemma~\ref{L_DE0=DEflat}) and so $\mathbb{L}^\flat \cap \mathbb{L}^\prime=\mathbb{L}^\flat/(\mathbb{D}\cap \mathbb{E}^\flat)$. But from the same Lemma, $\mathbb{D}\cap \mathbb{E}^\flat$ is closed in $ \mathbb{L}^\flat$. Therefore, $\mathbb{L}^\flat \cap \mathbb{L}^\prime$ has a structure of convenient space.\\
Using analog arguments, one can prove that $(\mathbb{L}^\flat)^\prime \cap \mathbb{L}$ is a convenient space which is equal to $\mathbb{L}/(\mathbb{D}\cap \mathbb{E})$.
\end{proof}

Now we have in situation to prove the following theorem:
\begin{theorem}
%\label{P_DiracPoisson}
\label{T_DiracPoisson} 
Let $\mathbb{D}$ be a partial linear Dirac structure on $\mathbb{E}$. 
\begin{enumerate}
\item[(1)] 
There exists  a linear surjective  bounded skew symmetric map $P_{\mathbb{L}}: \mathbb{L}\to  \mathbb{L}^\flat \cap \mathbb{L}^\prime$  whose kernel is  $\mathbb{D}\cap\mathbb{E}$. In particular, the range of $P_\mathbb{L}$ is also equal to $\mathbb{L}^\flat/(\mathbb{D}\cap \mathbb{E}^\flat)$. On $\mathbb{L}$,  we have a 
bounded skew symmetric  form $\Omega\in \bigwedge^2\mathbb{L}^\prime$ defined by
$\Omega	(u,v)=-<P_\mathbb{L}(v), u>$ whose kernel is $\mathbb{D}\cap\mathbb{E}$. Moreover, $\mathbb{D}$ is the graph of $\Omega$. 
% In particular, if $\mathbb{D}\cap\mathbb{E}=0$, then $\Omega_L$ is a weak symplectic form and in the other hand, if $\mathbb{D}\cap \mathbb{E}^\flat=\{0\}$
\item[(2)]  
The map $P_\mathbb{L}^\flat$ is a surjective  linear   bounded skew symmetric map $P_{\mathbb{L}}^\flat: \mathbb{L}^\flat \to \left(\mathbb{L}^\flat\right)^\prime \cap \mathbb{L}$ whose kernel is  $\mathbb{D}\cap\mathbb{E}^\flat $. In particular,  the range of $P^\flat_\mathbb{L}$ is also equal to   $ \mathbb{L}/(\mathbb{D}\cap \mathbb{E})$. Moreover $P^\flat _L$ is a partial linear Poisson structure on $\mathbb{E}$ and  $\mathbb{D}$ is the graph of $P^\flat_\mathbb{L}$.
 \item[(3)] 
 The map $P_\mathbb{L}$  (respectively $P^\flat_\mathbb{L}$) induces an isomorphism $\widehat{P}_\mathbb{L}:\mathbb{L}/(\mathbb{D}\cap \mathbb{E})= (\mathbb{L}^\flat)^\prime\cap \mathbb{L}\to \mathbb{L}^\flat/(\mathbb{D}\cap \mathbb{E}^\flat)= \mathbb{L}^\prime\cap\mathbb{L}^\flat$ (resp. $\widehat{P}_\mathbb{L}^\flat:\mathbb{L}^\flat/(\mathbb{D}\cap \mathbb{E}^\flat)=\mathbb{L}^\prime\cap\mathbb{L}^\flat\to \mathbb{L}/(\mathbb{D}\cap \mathbb{E})=(\mathbb{L}^\flat)^\prime\cap \mathbb{L} $) and $\widehat{P}_\mathbb{L}^\flat=(\widehat{P}_\mathbb{L})^{-1}$.
\item[(4)] 
Consider the following assertions
\begin{enumerate}
\item $\mathbb{D}\cap \mathbb{E}=\{0\}$;
\item $\mathbb{L}^\flat=\mathbb{E}^\flat$;
\item $p:\mathbb{D}\to \mathbb{E}$ is an isomorphism;
\item $\mathbb{D}\cap \mathbb{E}^\flat=\{0\}$;
\item $\mathbb{L}=\mathbb{E}$;
\item $p^\flat:\mathbb{D}\to \mathbb{E}^\flat$ is an isomorphism;
\item $P_\mathbb{L}$ is an isomorphism from $\mathbb{E}$ to $\mathbb{E}^\flat$;
\item $P^\flat_\mathbb{L}$ is an isomorphism from $\mathbb{E}^\flat$ to $ \mathbb{E}$.
\end{enumerate}
Then we have the following equivalences with the associated situations:
\begin{description}
\item[(i)]
$(a)\Longleftrightarrow (b) \Longleftrightarrow(c)$; in this case, $\mathbb{L}^\flat =\mathbb{E}^\flat $ is isomorphic to $\mathbb{D}$,  $\Omega$ is a weak symplectic form and $P_\mathbb{L}$ is injective with  range $\mathbb{E}^\flat/(\mathbb{D}\cap\mathbb{E}^\flat)=\mathbb{E}^\flat \cap \mathbb{L}^\prime$.
\item[(ii)]
$(d)\Longleftrightarrow (e) \Longleftrightarrow(f)$; in this case, $\mathbb{L}=\mathbb{E}$ is isomorphic to $\mathbb{D}$ , $\Omega$ is strong pre-symplectic and $P^\flat_\mathbb{L}$ is injective with range $\mathbb{E}/(\mathbb{D}\cap \mathbb{E})=(\mathbb{L}^\flat)^\prime\cap \mathbb{E} $.
\item[(iii)] 
$(a)+(d)\Longleftrightarrow (g) \Longleftrightarrow(h)$; in this case, $\mathbb{E}$ is isomorphic to $\mathbb{E}^\flat $ and to $\mathbb{D}$, $\Omega$ is weak symplectic and $\Omega^\flat $ is an isomorphism from $\mathbb{E}$ to $\mathbb{E}^\flat$.
\end{description}
\end{enumerate}
\end{theorem}
 
\begin{proof}${}$\\ %The proof of Point 1 and Point 2  are similar (as  for  Point 4 and Point 5 in lemma \ref{L_PartialAnulatorProperties}). Thus we only prove Point 1.\\
\noindent 1.  The value of $P_{\mathbb{L}}$ does not depend on the choice  of $(u,\alpha)\in \mathbb{D}$. Indeed, as we have already seen, if $p(u,\alpha)=p(u',\alpha')$ then $u=u'$ and 
 $\alpha-\alpha'$ belongs to $\ker p_{| \mathbb{D}}=\mathbb{D}\cap \mathbb{E}^\flat$.  But from Lemma~\ref{L_DE0=DEflat}, $\mathbb{L}^0= \mathbb{D}\cap \mathbb{E}^\flat$. Therefore, we have $p^\flat(u,\alpha)_{| \mathbb{L}}=p^\flat(u,\alpha')_{| \mathbb{L}}$. In particular,  if  $\mathbb{L}^\prime$  is the dual  of  $\mathbb{L}$, then 
 $P_{\mathbb{L}}$ takes values in  $\mathbb{L}^\prime\cap \mathbb{L}^\flat$. Note that if $i_{\mathbb{L}}$ is the inclusion of $\mathbb{L}$ in $\mathbb{E}$  then 
\[
P_\mathbb{L}(p(u,\alpha))=i^*_\mathbb{L} \left( \iota(p^\flat(u, \alpha) \right)
\]
where $\iota$ is the inclusion of $\mathbb{E}^\flat$ in $\mathbb{E}^\prime$.\\
The restrictions of  $p$  (resp. $p^\flat$)  to $D$ are  bounded  surjective linear maps  onto $\mathbb{L}$ (resp. $\mathbb{L}^\flat$) (cf. proof of Lemma~\ref{L_PDPflatDClosed}). Thus each bounded set  in $\mathbb{L}$ (resp. in $\mathbb{L}^\flat$) is the image of a bounded set in $\mathbb{D}$ by $p$ (resp. by $p^\flat$).  Since  $i_\mathbb{L}$ and $\iota$  are bounded linear maps,  it follows that  $P_{\mathbb{L}}$ is a bounded linear map.
 
On the other hand, $P_\mathbb{L}$ is skew symmetric according to the  pairing between $ p^\flat (\mathbb{D})$ and $ p (\mathbb{D})$  defined by (\ref{eq_QuotientPairing}).
Indeed:
\[
\begin{array}{rcl}
<P_\mathbb{L} \left( p(u,\alpha) \right),p(v,\beta)>
	&=&<p^\flat(u,\alpha), p(v,\beta)>\\
	&=&<\alpha,v> = -<\beta,u>\\
	&=&-<P_\mathbb{L} \left( p(v,\beta) \right),p(u,\alpha)>.
\end{array}
\]
Moreover, according to  its definition,  the kernel of $P_\mathbb{L}$ is the  closed subspace $\mathbb{D}\cap \mathbb{E}$ of $\mathbb{L}$ (cf. Lemma \ref{L_DE0=DEflat}). Now from its definition, the range of $P_\mathbb{L}$ is precisely the convenient quotient space  $\mathbb{L}^\flat/(\mathbb{D}\cap \mathbb{E}^\flat)$.  Thus, by Lemma \ref{L_LbinterL'Linter Lb'}, $P_\mathbb{L}$ is surjective.% Conversely, for any $\alpha\in \mathbb{L}^\flat$ denote by $[\alpha]$ its class in $\mathbb{L}^\flat/(\mathbb{D}\cap \mathbb{E}^\flat)$. Then from the definition of $P_\mathbb{L}$ it follows easily that that for any $\alpha'\in [\alpha]$   if $(u,\alpha)$ belongs to $\mathbb{D}$ then $(u,\alpha')$ belongs to $(\mathbb{D})$ and so $\alpha_{|\mathbb{L}}=\alpha'_{|\mathbb{L}}$ which show that $P_\mathbb{L}$ is surjective which implies  The result (cf proof of Lemma \ref{L_PDPflatDClosed}). \\
Therefore, we have a  well defined bilinear bounded $2$-form  $\Omega$ on $\mathbb{L}$ given by
\[
\Omega(u,v)=-<P_\mathbb{L}(v), u>
\]
whose kernel is $\mathbb{D} \cap \mathbb{E}$.\\
By definition of $p$, $p^\flat$ and $P_\mathbb{L}$, we have 
\begin{equation}
\label{eq_GraphPL}
\mathbb{D}=\{(u,\alpha)\in \mathbb{E}\times\mathbb{E}^\flat:\; u\in \mathbb{L}, \alpha\in \mathbb{L}^\flat, P_\mathbb{L}(u)=\alpha_{| \mathbb{L}}\}
\end{equation}
which is exactly the graph of $\Omega$.\\
2. We must prove that the definition of $P^\flat_\mathbb{L}$  does not depend on the choice of $\delta=(u,\alpha)\in \mathbb{D}$. Assume that $p(u,\alpha)=p(u',\alpha')$. Then it 
follows that $\alpha=\alpha'$ and $u-u'$ belongs to  $\ker p^\flat= \mathbb{D}\cap \mathbb{E}$. This means that for $u-u'$ considered as an element of $\mathbb{E}\subset \mathbb{E}
^{\prime\prime}$, we have $p^\flat(u,\alpha)_{| \mathbb{L}^\flat}= p^\flat(u',\alpha)_{| \mathbb{L}^\flat}$ and so $P^\flat_\mathbb{L}(\delta)$ belongs to $(\mathbb{L}
^\flat)^\prime\cap\mathbb{L}$.  In fact, if $ \delta=(u,\alpha)\in \mathbb{D}$, then  $p^\flat(\delta)=u$ and so belongs to $p(\mathbb{D})\subset \mathbb{E}\subset \mathbb{E}
^{\prime\prime} $. More precisely, the adjoint  $\iota^*$ of $\iota$ is a bounded map from $\mathbb{E}^{\prime\prime}$ to $(\mathbb{E}^\flat)^\prime$ and so its restriction to $
\mathbb{E}$ is a bounded map $\iota^*_\mathbb{E}$ from $\mathbb{E}$ to $(\mathbb{E}^\flat)^\prime\cap\mathbb{E}$, and if  $i_{\mathbb{L}^\flat}$ the inclusion of $\mathbb{L}
^\flat$ in $\mathbb{E}^\flat$ we have 
\[
P^\flat_\mathbb{L}(p^\flat(u,\alpha))
=i_{\mathbb{L}^\flat}^* 
 \left( \iota_\mathbb{E}(p(u,\alpha)) \right) .
\]
By  same arguments as for $P_\mathbb{L}$, it follows that $P^\flat_\mathbb{L}$ is a bounded map.\\
As for $P_\mathbb{L}$, the range of $P^\flat_\mathbb{L}$ is $ \mathbb{L}/(\mathbb{D}\cap \mathbb{E})$ and so $P^\flat_\mathbb{L}$ is surjective by Lemma~\ref{L_LbinterL'Linter Lb'}. 

The proof of the other properties of $P^\flat_\mathbb{L}$ announced is analogue as for $P_\mathbb{L}$.\\

On the one hand,  $\mathbb{L}^\flat\subset \mathbb{E}^\flat\subset \mathbb{E}^\prime$ and the inclusion of   $\mathbb{L}^\flat$ in $ \mathbb{E}^\prime$ is bounded  since $\mathbb{L}^\flat$ is closed in $\mathbb{E}^\flat$ and the inclusion of $\mathbb{E}^\flat $ in $\mathbb{E}^\prime$ is bounded.\\
On the other hand,  the canonical 
pairing in restriction to $(\mathbb{L}^\flat)^\prime \cap \mathbb{L}\times \mathbb{L}^\flat$ is bounded and  so the inclusion of  $(\mathbb{L}^\flat)^\prime \cap \mathbb{L}$ 
into $\mathbb{L}$ is $ <u,.>\mapsto u$ which is bounded as the evaluation map of the bilinear bounded map $<\;,\;>$. It follows that the inclusion of $(\mathbb{L}^\flat)^\prime 
\cap \mathbb{L}$ into $\mathbb{E}$ is bounded. This implies that $P^\flat_\mathbb{L}$ is a partial Poisson linear anchor on $\mathbb{E}$.\\ 

Finally as in the previous case, we have 
\begin{equation}
\label{eq_GraphPflatL}
\mathbb{D}=\{(u,\alpha)\in \mathbb{E}\times\mathbb{E}^\flat:\; u\in \mathbb{L}, \alpha\in \mathbb{L}^\flat, P^\flat_\mathbb{L}(\alpha)=u_{| \mathbb{L}}\}.
\end{equation}
\smallskip

\noindent 3. From Point 1., $P_\mathbb{L}$ is a bounded map from the convenient space $\mathbb{L}$ to the convenient space $\mathbb{L}^\flat/(\mathbb{D}\cap \mathbb{E}^\flat)$. Thus, as in the proof of Lemma \ref{L_PDPflatDClosed}, we obtain a quotient map $\widehat{P}_\mathbb{L}: \mathbb{L}/(\mathbb{D}\cap \mathbb{E})\to \mathbb{L}^\flat/(\mathbb{D}\cap \mathbb{E}^\flat)$  which is a convenient isomorphism. The same type of  argument works for the quotient map  $\widehat{P}_\mathbb{L}^\flat:\mathbb{L}^\flat/(\mathbb{D}\cap \mathbb{E}^\flat)\to \mathbb{L}/(\mathbb{D}\cap \mathbb{E}) $. The last property is a direct consequence of the definition of each map  $P_\mathbb{L}$  and  $P_\mathbb{L}^\flat$.\\

 If $\mathbb{D}\cap\mathbb{E}=0$, then  $\Omega_\mathbb{L}$ is (weak) symplectic and also  $\mathbb{L}^\flat =(\mathbb{D}\cap\mathbb{E})^0 =\mathbb{E}^\flat$
 \footnote{Cf. Lemma \ref{L_DE0=DEflat}.}.  
 From Lemma \ref{L_LbinterL'Linter Lb'}, it follows that  we have $\mathbb{L}^\flat \cap \mathbb{L}^\prime=\mathbb{E}^\flat\cap \mathbb{E}^\prime=\mathbb{E}^\flat$.  Thus from the previous results, $P_\mathbb{L}=\widehat{P}_\mathbb{L}$ and $P^\flat_\mathbb{L}=\widehat{P}_\mathbb{L}^\flat$ and 
  therefore  $P_\mathbb{L}^\flat $ is an isomorphism from $\mathbb{E}^\flat$ to $\mathbb{L}=\mathbb{E}$, which ends the proof of Point 3.\\

\noindent  4.${}$ \\
\noindent (i)  $ (a)\Longleftrightarrow (b)$: $\mathbb{D}\cap \mathbb{E}=0$ implies  $( \mathbb{D}\cap \mathbb{E})^0=\mathbb{E}^\prime\cap\mathbb{E}^\flat=\mathbb{E}^\flat$; the converse is trivial.  Thus  the equivalence is a consequence of  Lemma~\ref{L_DE0=DEflat}.

$(a)+(b)\Longrightarrow (c)$ is a consequence of Lemma~\ref{L_PDPflatDClosed}.

$(c)\Longrightarrow (a)$ is obvious. \\
 The associated situation described is a direct consequences of Point 2 and Point 3.\\

\noindent (ii)  $ (d)\Longleftrightarrow (e)$: from  Lemma~\ref{L_DE0=DEflat}, we have $\mathbb{D}\cap \mathbb{E}^\flat=\{0\}$ equivalent to $\mathbb{L}^0=\{0\}$ implies $\mathbb{L}=\mathbb{E}$
 since $(\mathbb{L}^0)^{\operatorname{ann}}$ is closed in $\mathbb{E}^{\prime\prime}$ and so is $\mathbb{E}$, therefore,  $(\mathbb{L}^0)^{\operatorname{ann}}\cap\mathbb{E}$ is the closure of the set $\{u\in \mathbb{E}:\, \forall \alpha\in \mathbb{L}^0, \alpha(u)=0 \}$ which is  the closed set $\mathbb{L}$.  Thus as  $\mathbb{E}^0=\{0\}$, we obtain $\mathbb{L}=\mathbb{E}$. The converse is trivial.
 
  $(d)+(e)\Longrightarrow (f)$ is a consequence of Lemma \ref{L_PDPflatDClosed}.
  
  $(f)\Longrightarrow (d)$  is trivial.\\
The associated situation described is a direct consequences of Point 2 and Point 3.\\
  
(iii) $(a)+(c)\Longrightarrow (g)$ and  $(a)+(c) \Longrightarrow (h)$:  if $\mathbb{D}\cap\mathbb{E}=0$, then  $\Omega_\mathbb{L}$ is (weak) symplectic and also  $\mathbb{L}^\flat =(\mathbb{D}\cap\mathbb{E})^0 =\mathbb{E}^\flat$\footnote{Cf. Lemma \ref{L_DE0=DEflat}.}.  From Lemma \ref{L_LbinterL'Linter Lb'}, it follows that  we have $\mathbb{L}^\flat \cap \mathbb{L}^\prime=\mathbb{E}^\flat\cap \mathbb{E}^\prime=\mathbb{E}^\flat$.  Thus from Point 3, $P_\mathbb{L}=\widehat{P}_\mathbb{L}$ and $P^\flat_\mathbb{L}=\widehat{P}_\mathbb{L}^\flat$ and 
  therefore  $P_\mathbb{L}^\flat $ is an isomorphism from $\mathbb{E}^\flat$ to $\mathbb{L}=\mathbb{E}$.

  $(g)\Longrightarrow (a)+ (c)$ and $(h)\Longrightarrow (a) + (c) $ results from Point 3.\\
The associated situation described is a direct consequence of Point 2 and Point 3.
\end{proof}

\begin{remark} ${}$
\begin{enumerate} 
\item[1.] 
According to Lemma \ref{L_PartialAnulatorProperties}, 2., the result of Theorem \ref{T_DiracPoisson}, 1. implies that the graph $\mathbb{D}_{P_\mathbb{L}}$ is a linear partial Dirac structure on $\mathbb{L}$. Moreover, according to (\ref{eq_GraphPL}), the associated $2$-form is also defined by:  $\Omega(u,v)=\alpha(v)=-\beta(u)$ if $p^\flat (u,\alpha)=\alpha$  and $p(v,\beta)=v$.
\item[2.] In Theorem \ref{T_DiracPoisson}, 2., if $\mathbb{E}^\flat=\mathbb{E}^\prime$,  we have 
\[
\mathbb{L}^0=(\mathbb{D}\cap \mathbb{E})^{\operatorname{ann}}=\left(\mathbb{E}/(\mathbb{D}\cap \mathbb{E})\right)^\prime
\] 
and so $P^\flat_\mathbb{L}$ is a skew symmetric bounded map from $\left(\mathbb{E}/(\mathbb{D}\cap \mathbb{E})\right)^\prime$ to $\mathbb{E}/(\mathbb{D}\cap \mathbb{E})$ as in 
finite dimension (cf. Proposition 1.1.4 in \cite{Cou90}). In our general context, we have $\mathbb{L}^0=(\mathbb{D}\cap \mathbb{E})^0= (\mathbb{D}\cap 
\mathbb{E})^{\operatorname{ann}}\cap\mathbb{E}^\flat$,  so it seems natural to set 
$\left(\mathbb{E}/(\mathbb{D}\cap \mathbb{E})\right)^\flat=\mathbb{L}^0.$
Therefore we can consider that 
$P^\flat_\mathbb{L}$ is a skew symmetric bounded map from $\left(\mathbb{E}/(\mathbb{D}\cap \mathbb{E})\right)^\flat$ to $\mathbb{E}/(\mathbb{D}\cap \mathbb{E})$.\\
By the way, in this sense,  Theorem \ref{T_DiracPoisson}, 1. and 2.  can be considered as a kind of generalization of Proposition 1.1.4 in \cite{Cou90}.
\end{enumerate}
\end{remark}

\subsection{Partial almost Dirac structure} \label{___PartialAlmostDirac}

From now until the end of this document (except for a specific context),  $M$ is a convenient manifold modelled on a convenient space $\mathbb{M}$.
We consider its kinematic tangent bundle $p_{M}:TM\to M$ and its kinematic cotangent bundle $p_{M}^{\prime}:T^{\prime}M\to M$.\\

Let $\pi:E\to M$ be a convenient vector bundle and $\rho: E\to TM$ an anchor. If $\pi^*:E^\prime\to M$ is the dual bundle  of $E$, we consider a weak subbundle $\pi^\flat:E^\flat \to M$ of the dual bundle. As in the previous linear context, we can associate the \emph{Pontryagine bundle} $E^{\mathfrak{p}}=E\oplus E^\flat$ over $M$\index{Ep@$E^{\mathfrak{p}}=E\oplus E^\flat$ (Pontryagine bundle)}. The  canonical pairing between $E^\prime $ and $E$ is denoted $<.,.>$ and it induces a bounded bilinear map on $E^\flat\times E$ which is left non degenerate and,  if there is no  ambiguity, will be also denoted by the same symbol. 

\begin{conventionnotation}
\label{C_EEflat} ${}$
\begin{enumerate}
\item 
From now, if  there is no ambiguity,  $E^{\mathfrak{p}}$ will be considered indifferently as $E\oplus  E^\flat$ or   $E\times E^\flat$ and  the bundles $E$ and $E^\flat$ will be considered as weak subbundles of $E^{\mathfrak{p}}$.
\item 
If $ F$ is a closed weak subbundle of $E$, we set $F^{\operatorname{ann}}=\displaystyle\bigcup_{x\in M}F^{\operatorname{ann}}_x$ and $F^0=F^{\operatorname{ann}}\cap E^\flat$.
Note that  $F^{\operatorname{ann}}$ is always a closed subbundle of $E^\prime$ but, in general, $F^0$ is not a subbundle of $E^\flat $ although $F_x^0$ is always closed in $E_x^\flat$ for any $x\in M$.
\item 
We say that $E^\flat$ is \emph{separating}\index{separating} if $E_x^\flat$ separates points of $E_x $ for any $x\in M$.
\end{enumerate}
\end{conventionnotation}

On $E^{\mathfrak{p}}$, we consider the bilinear bounded form $<<.,.>>$ defined by:
\[
<<(u,\alpha), (v,\beta)>>=<\alpha,v>+<\beta, u>.
\]
Let $D$ be a weak  closed subbundle of $E^{\mathfrak{p}}$. According to the previous subsection, we set $D^\perp=\displaystyle\bigcup_{x\in M} D_x^\perp$. 

\begin{definition}
\label{D_AlmostDirac}
We say that $D$ is a \emph{partial almost Dirac structure}\index{partial!almost Dirac structure} on $E$, if we have $D_x^\perp=D_x$ for all $x\in M$.\\
If $T^\mathfrak{p}M=TM\oplus T^\prime M$, a partial almost structure is simply called an \emph{almost Dirac structure}\index{almost Dirac structure}.
\end{definition}

\begin{example}
\label{ex_EEflat} 
According to Convention~\ref{C_EEflat}, the bundle $E$ (resp. $E^\flat$) is a partial almost Dirac structure on $E$.\\
\end{example}

\begin{example}
\label{ex_EE0}  
Let $F$ be a closed subbundle of $E$ and assume that $F^0$ is a subbundle of  $E^\flat$. Thus $D=F\oplus F^0$ is a closed subbundle of 
$E^{\mathfrak{p}}$. Then, according to Lemma~\ref{L_PartialAnulatorProperties}, 2.,    $F\oplus F^0$ is a partial almost Dirac structure on $E$ .
\end{example}

\begin{example}
\label{ex_SkewsymmetricMorphism1} 
Let $P:E^\flat \to E$ be a  convenient skew symmetric morphism (i.e. $P_x$ is skew symmetric for any $x\in M$). We denote by $D_P$ the graph of $P$, considered as a subset of $E^{\mathfrak{p}}$. If  $P(E^\flat)$ is a closed  subbundle of $E$ then $D_P$ is a closed subbundle of $E^{\mathfrak{p}}$. If, moreover,  $(E_x^\flat)^{\operatorname{ann}}\cap E_x=\{0\}$ for all $x\in M$, then $D_P$ is a partial almost Dirac structure on $E$ (cf. Lemma~\ref{L_PartialAnulatorProperties}, 3.). This result  is always true if  $E^\flat$ is separating and, in particular, when $E^\flat=E^\prime$ as in finite dimension (cf. Remark \ref{R_ContreExamples}, 1. and 4.). Also, if $P$ is surjective, then $D_P$ is a closed subbundle of $E^{\mathfrak{p}}$ and an almost partial Dirac structure on $E$ from Lemma~\ref{L_PartialAnulatorProperties}, 3.
\end{example}

\begin{example}
\label{ex_SkewsymmetricMorphism2} 
Let $\Omega:E \to E^\flat$ be a skew symmetric bundle morphism.\\
We denote by $D_\Omega$ the graph of $\Omega$. If  $\Omega(E)$ is a closed  subbundle of $E^\flat$ then $D_\Omega$  is a closed subbundle of $E^{\mathfrak{p}}$ and a partial almost Dirac structure on $E$ from Lemma ~\ref{L_PartialAnulatorProperties}, 2.
\end{example}

%Given a partial almost Dirac  structure $D$ on $E$, we denote by $p:E^{\mathfrak{p}}\to E$ and $p^\flat :E^{\mathfrak{p}}\to E^\flat$. According to the end of the previous subsection, for any $x\in M$, we have $\ker p_{| D_x}=D_x\cap E^\flat_x=(D_x\cap E_x)^0$ and $p(D_x)=D_x\cap E_x:=L_x$. We set $L=\bigcup_{x\in M}L_x $. In general $L$ is not a subbundle of $E$
\subsection {Courant bracket and partial  Dirac structure}
\label{___CourantBracketDirac}

\emph{This section is nothing but an adaptation of the results of \cite{Vul14} to our situation of convenient partial almost Dirac structures.}\\

For any convenient bundle $F$ over $M$, we denote by $F_U$ the restriction of $F$ over an open set $U$ of $M$ and  by $\Gamma(F_U)$ the set of smooth sections of $F$ over $U$. Under the general previous  context, we first have:
 
\begin{lemma}
\label{L_LieDerivativeSectionEflat} 
Assume that the vector bundle $E$ has a convenient Lie algebroid structure $(E,M,\rho, [.,.]_E)$.
 For any $s \in \Gamma(E_U)$ and $\alpha\in \Gamma(E^\flat_U)$, the Lie derivative $L_{\rho(s)}\alpha$ is a section of $E^\flat_U$.
\end{lemma}
 
\begin{proof}
Let $\mathbb{M}$ be the convenient space on which $M$ is modelled and $\mathbb{E}$ (resp. $\mathbb{E}^\flat$) the typical model of the fibre of $E$ (resp. $E^\flat$) and $\mathbb{E}^\flat$ which is a vector subspace of $\mathbb{E}^\prime$ whose inclusion is bounded since the inclusion  of $E^\flat$ in $E$ is an injective morphism. As it is a local problem, we can assume that  $U$ is a $c^\infty$-open set of $\mathbb{M}$ which is simply connected and so that we have:
 
 $TM=U\times\mathbb{M}$, $\;E_U=U\times\mathbb{E}$, $\;E^\prime_U=U\times\mathbb{E}^\prime$, and $E^\flat_U=U\times \mathbb{E}^\flat$.

\noindent Now, in this context  for the Lie algebroid structure on $E$,   each local sections  of $E_U $ can be identified with a  smooth map from $U$ to $\mathbb{E}$. By the way,  the Lie bracket 
$[.,.]_E$ has the following expression (\cite {CaPe23}):
\begin{equation}
\label{eq_EBracket}
[s,v]_E(x)=d_xv(\rho(s))-d_xs(\rho(v))-C_x(s,v)\footnote{In this proof, the differential $d$ means the classical differential of maps in  the convenient setting.}
\end{equation}
where  $x\mapsto C_x$ is a smooth field of skew symmetric bilinear on $\mathbb{E}$ with values in $\mathbb{E}$.  Then,  the Lie derivative $L_{\rho(s)}\alpha$ can be written:
\begin{equation}
\label{eq_LieDerivative}
\begin{matrix}
L_{\rho(s)}\alpha(v)
	&=d(<\alpha, v>)(\rho(s))-\alpha([s,v]_E\hfill{}\\
	&=d(\alpha(\rho(s))(v)+\alpha\circ dv(\rho(s))-\alpha\circ  dv(\rho(s))+\alpha\circ ds(\rho(v))+\alpha\circ C_x(s,v)\hfill{}\\
	&=d\alpha(\rho(s))(v)-\alpha\circ ds\circ \rho(v)+\alpha\circ C(s,v)\hfill{}
\end{matrix}
\end{equation}

In the previous local  context, any section $\alpha$ of $E^\flat_U$  can be considered as a smooth map  from $U$ to $\mathbb{E}^\prime$ which takes value in $\mathbb{E}^\flat$ and any section $s\in \Gamma(E_U)$ is  a smooth map $s:U\to \mathbb{E}$. It follows that,  for each $x\in U$,  the differential $d_x\alpha$ of the map $\alpha$ belongs to the space $L \left( \mathbb{M},\mathbb{E}^\flat \right) $ of bounded linear maps from $\mathbb{M}$ to $\mathbb{E}^\flat$ and so $d_x\alpha(\rho(s))$ belongs to $\mathbb{E}^\flat$.\\
On the other hand,  $d_x s\circ \rho$ belongs to $L \left( \mathbb{E},\mathbb{E} \right)$. Thus, $\alpha\circ d_xs\circ \rho$ belongs to $\mathbb{E}^\flat$ since $\alpha$ belongs to $\mathbb{E}^\flat$. By the same argument $\alpha\circ C_x(s,.)$ belongs to $\mathbb{E}^\flat$. \\
Taking into account the previous consideration, the last member of  (\ref{eq_LieDerivative}) implies that  $L_{\rho(s)}\alpha$ belongs to $E^\flat_U$.
\end{proof}

\smallskip 
Let $U$ be an open set in $M$. As a generalization of the Courant bracket on a Banach Lie algebroid  (cf. \cite{Vul14}), we introduce:

\begin{definition}
\label{D_CourantBracket}
The \emph{Courant bracket}\index{Courant bracket}\index{bracket!Courant} on sections of $E^{\mathfrak{p}}$ is defined by
\begin{equation}
\label{eq_Courant}
[(r,\alpha),(s,\beta)]_C
=\left([r,s]_E,L_{\rho(r)}\beta-L_{\rho(s)}\alpha-\displaystyle\frac{1}{2} d(i_s \alpha-i_r\beta)\right)
\end{equation}
for all sections $(r,\alpha), (s,\beta)\in \Gamma(E^{\mathfrak{p}}_U) $.
\end{definition}

\begin{proposition}
\label{P_PropertyCourantBracket}
The Courant bracket $[.,.]_C$ takes values in $\Gamma(E^{\mathfrak{p}})$, depends only on the bracket $[\;,\;]_E$ and satisfies the Leibniz property.\\
\end{proposition}

\begin{proof} 
First of all, by definition, the bracket $[.,.]_E$ takes values in $E$ and satisfies the Leibniz property.\\
On the other hand, from Lemma \ref{L_LieDerivativeSectionEflat},  $L_{\rho(r)}
\beta-L_{\rho(s)}\alpha$ takes values in $E^\flat$. Finally, as in the proof of this Lemma, in the same local context, for all $v\in \mathbb{M}$, we have $d<\alpha,s> (v)=d\alpha(s(v))+
\alpha\circ d s(v)$ which belongs to $E^\flat_U$ from argument used in the proof of Lemma~\ref{L_LieDerivativeSectionEflat}. Thus $[.,.]_C$ takes values in $E^{\mathfrak{p}}$. The 
Leibniz property of the second component of this bracket is a consequence of classical properties of Lie derivative and the fact that
\[
d<\alpha,fs>(v)
=f d\alpha(s,v) + df(v)<\alpha,s>
\] 
and also 
\[
d<f\beta,r>(v)=df(v)<\beta,r>+ f\beta\circ ds(v)
\]
which ends the proof.
\end{proof}

\begin{remark}
%\label{otherexpresionC}
\label{R_OtherExpressionOfC}
As in finite dimension (cf. \cite{Marc16} for instance),  according to the classical definition of Lie derivative and exterior differential of forms and relation between Lie derivative, inner product and exterior derivative on a convenient Lie algebroid (cf. \cite{CaPe23}, Chap. 3, 18.) the Courant bracket has also the following expressions:
\begin{equation}
\label{eq_CExpresion}
\begin{matrix}
[(r,\alpha),(s,\beta)]_C&=\left([r,s]_E,L_{\rho(r)}\beta-L_{\rho(s)}\alpha-\displaystyle\frac{1}{2} d(i_s \alpha-i_r\beta)\right)\hfill{}\\
				&=\left([r,s]_E,L_{\rho(r)}\beta-i_s d\alpha)\right)\hfill{}\\
				&= \left([r,s]_E,i_rd\beta-i_s d\alpha+d(i_r\beta\right)\hfill{}\\

\end{matrix}
\end{equation}
\end{remark}
Unfortunately, in general, it is not a Lie bracket on $\Gamma(E^{\mathfrak{p}})$.\\
Thus, we can consider the Jacobiator $J= \left( J_1,J_2 \right)$ (cf. \cite{Vul14})
\[
\begin{array}{rcl}
J \left((r,\alpha),(s,\beta),(t,\gamma)\right)	&=& [[(r,\alpha),(s,\beta)]_C,(t,\gamma)]_C\\
	&&+[[(s,\beta),(t,\gamma)]_C,(r,\alpha)]_C\\	
	&&+[[(t,\gamma)(r,\alpha)]_C, (s,\beta)]_C
\end{array}
\]
for all $(r,\alpha),(s,\beta),(t,\gamma)\in \Gamma(E^{\mathfrak{p}}_U)$.\\
Note that since $E$ has a Lie algebroid structure, then the first component $J_1$ satisfies $J_1=0$. \\

\emph{We are interested in conditions under which the restriction of the  Courant bracket on  sections of $D_U$ is a Lie bracket.} \\

For this purpose, we introduce the following map on $\Gamma(D_U)$:
\begin{equation}
\label{eq_TD}
\mathbf{T}_D\left((r,\alpha),(s,\beta),(t,\gamma)\right)=<<[(r,\alpha), (s,\beta)]_C,(t,\gamma)>>
\end{equation}
for all $(r,\alpha),(s,\beta),(t,\gamma)\in \Gamma(D_U)$.\\

Then using analog arguments as in \cite{Vul14}, 3., we have:

\begin{lemma}
\label{L_PropertiesJD} 
In restriction to sections of $D_U$,  we have the following properties, for all $(r,\alpha),(s,\beta),(t,\gamma)\in \Gamma(D_U)$:
\begin{enumerate}
\item  
$\mathbf{T}_D \left( (r,\alpha),(s,\beta),(t,\gamma) \right) 
= L_{\rho(r)}\beta(t)+ L_{\rho(s)}\gamma(r)+ L_{\rho(t)}\alpha(s)$.
\item 
The second component $J_2$ of the Jacobiator $J$ can be written:
\[
J_2\left((r,\alpha),(s,\beta),(t,\gamma)\right)=\displaystyle\frac{1}{2} d_E\left( \mathbf{T}_D((r,\alpha),(s,\beta),(t,\gamma))\right)
\]
where $d_E$ is the exterior differential associated to the Lie algebroid structure $(E,M,\rho, [.,.]_E)$.
\end{enumerate}
\end{lemma}

\begin{proposition}
\label{P_CNSCourant bracketD} 
Assume that $E$ is a convenient Lie algebroid and that $D$  is an almost partial Dirac structure on $E$. Then the restriction of the Courant bracket $[.,.]_C$ to $\Gamma(D_U)$ is a Lie bracket if and only if $\mathbf{T}_D\equiv 0$.
\end{proposition}

\begin{proof} 
If $[.,.]_C$ induces a Lie bracket on $\Gamma(D_U)$, from (\ref{eq_TD}), since $D^\perp=D$, it follows that $\mathbf{T}_D\equiv 0$.\\
Conversely, if $\mathbf{T}_D\equiv 0$, from (\ref{eq_TD}), $[(r,\alpha),(s,\beta)]_C$ must belong to $\Gamma(D_U^\perp)$ for all $(r,\alpha),(s,\beta)\in \Gamma(D_U)$ and so the restriction of the Courant bracket $[.,.]_C$ to $\Gamma(D_U)$ is a Lie Bracket.
\end{proof}

\begin{definition}
\label{D_DiracStructure} 
A  (partial) almost Dirac structure $D$ on a convenient Lie algebroid $E$ is called \emph{involutive}\index{involutive} if, for any open set $U$ in $M$, $\Gamma(D_U)$ is stable under the restriction of the Courant bracket $[.,.]_C$ to sections of $\Gamma(D_U)$.
\end{definition}

We denote by $p:E^\mathfrak{p}\to E$ and define $\rho^{\mathfrak{p}}:E^{\mathfrak{p}}\to TM$ by
\begin{equation}
\label{eq_rhop}
\rho^{\mathfrak{p}}((r,\alpha))=\rho\circ p((r,\alpha))
\end{equation}
where $\rho:E \to TM$ is the anchor and so $\rho^{\mathfrak{p}}: E^{\mathfrak{p}}\to TM$ is also an anchor.\\
On the other hand, since $D_U=D_U^\perp$, the restriction of $[.,.]_C$ to $\Gamma(D_U)$ can be written:
\begin{equation}
\label{eq_CourantrestrictD}
[(r,\alpha),(s,\beta)]_C
=\left([r,s]_E,L_{\rho(r)}\beta-L_{\rho(s)}\alpha+d(i_s \alpha)\right)
\end{equation}
for all sections $(r,\alpha), (s,\beta)\in \Gamma(D_U) $.

\begin{theorem}
\label{T_DLieAlgebroid}
A (partial) almost Dirac structure $D$ on a convenient Lie algebroid $(E,M,\rho, [.,.]_E)$ is a (partial) Dirac structure if and only if the restriction $\rho^D$ of $\rho^{\mathfrak{p}}$ to $D$ is an anchor such that $(D,M,\rho^D, [.,.]_{ D})$ has a convenient Lie algebroid structure on $D$.
\end{theorem}

\begin{proof} 
At first, according to the expression of the Lie derivative in local coordinates  (cf. (\ref{eq_LieDerivative})) and also $d(i_s \alpha)$, it follows that the second member of 
(\ref{eq_CourantrestrictD}) only depends on the $1$-jet of each section $(r,\alpha), (s,\beta)\in \Gamma(D_U) $. Since  the property is satisfied for $[.,.]_C$, the same is true for $[.,.]_D$ (relatively to the anchor $\sigma$).\\

Assume that $D$ is a partial Dirac structure on $E$ and fix any open set $U$ in $M$. From Proposition~\ref{P_CNSCourant bracketD}, it follows that $\mathbf{T}_D\equiv 0$  and from Lemma \ref{L_PropertiesJD}, 2., $J_2$ in restriction to $\Gamma(D_U)$ vanishes. Since $J_1\equiv 0$, it follows that the bracket $[.,.]_D$ satisfies the Jacobi identity. Therefore, from Proposition~3.19 in \cite{CaPe23}, we have
 $\rho^D([(r,\alpha),(s,\beta)]_D)=[\rho^D(r,\alpha),\rho^D(s,\beta)]_D$ for all $(r,\alpha), (s,\beta)\in \Gamma(D_U) $.  Since these properties are satisfied over any open set $U$ in $M$, it follows that $(D,M,\rho^D, [.,.]_{ D})$ is a Lie algebroid.\\
 
Conversely, if $(D,M,\rho^D, [.,.]_{ D})$ is a Lie algebroid, the bracket $[.,.]_D$ satisfies the Jacoby identity and so $J_2$ in restriction to $\Gamma(D_U)$ vanishes over any open set $U$ in $M$. From Lemma~\ref{L_PropertiesJD}, 2., the value of  $\mathbf{T}_D$ on $\Gamma(D_U)$ is then constant and since $\mathbf{T}_D(0,0,0)=0$, this implies $\mathbf{T}_D\equiv 0$. But it is true over any open set $U\in M$; so, from Proposition~\ref{P_CNSCourant bracketD}, this implies that $D$ is a partial Dirac structure on $E$.
\end{proof}

\begin{example}
\label{Ex_DE} 
The bundle $E$ is an almost partial Dirac structure on $E$ and if $E$ is also a Lie algebroid, then $E$ is a partial Dirac structure on $E$.
\end{example}

The following Proposition  and Corollary give  a large class of examples which are   a generalization of classical results in finite dimension:\\

\begin{proposition}\label{P_DSkewsymmetricMorphism2} 
Let $F$ be a closed subbundle of  $E$ and $Q:E \to E^\flat$ be a skew symmetric convenient  bundle morphism. We set $F^0= \displaystyle\bigcup_{x\in M}F_x^0$ and 
\[
(D_Q^F)=\bigcup_{x\in M} \{ (u,\alpha)\in E^{\mathfrak{p}}_x\;:\; u\in F_x,\; \alpha-Q_x(u)\in F_x^0 \}.
\]
Assume that $Q(F)+ F^0$ is a closed subbundle of $E^\flat$. Then $D^F_Q$ is a partial almost Dirac structure. In this case, 
 $D_Q^F$ is  then a partial Dirac structure if and only if $F$ is involutive (according to the Lie bracket $[.,.]_E$) and  the restriction of the  $2$ form $\Omega_x(u,v)=<Q(u),v>$ to $F$ is closed.
% In particular if $Q=0$ then $D_F:=D_Q^F=F\oplus F^0$ is involutive if and only $F$ is involutive and if $F=E$ then $D_Q:=D_Q^E$  ( {\it i. e.} the graph of $Q$) is  involutive if and only if $\Omega$ is closed on $E$.
\end{proposition}
  
By taking $Q=0$ in the one hand and $F=E$ on the other hand we obtain:
  
\begin{corollary}
\label{C_XX0Omega}${}$
  \begin{enumerate}
  \item  
  Let $F$ a closed subbundle of $E$ such that $F^0$ is a subbundle of $E^\flat$. Then $F\oplus F^0$ is a  partial Dirac structure if and onliy if $F$ is involutive.
 \item 
 Let  $Q:E \to E^\flat$ be a skew symmetric convenient  bundle morphism such that $Q(E)$ is a closed subbundle of $E^\flat$. Then the graph of  $Q$ is a partial Dirac structure if and only the $2$-form  $\Omega$ associated to $Q$ is closed.
  \end{enumerate}
  \end{corollary}
  
\begin{remark}
\label{R_F0} 
When $E^\flat =E^\prime$ then $F^0$ is always a closed subbundle of $E^\flat$. Indeed, if $\mathbb{E}$ (resp. $\mathbb{F}$) is the typical fibre of $E$ (resp. $F$), since $F$ is closed in $E$ and  if $M$ is connected, we may assume that $\mathbb{F}$ is a closed convenient subspace of $\mathbb{E}$. Therefore, we can choose an atlas $\{
  \left( U_\alpha,\Phi_\alpha \right) \}_{\alpha\in \mathbf{A}}$ of local trivializations $\Phi_\alpha: E_{| U}\to U\times \mathbb{E}$ such that the restriction of each $\Phi_\alpha$ to $F_{| U}$ is a  trivialization onto $U\times\mathbb{F}$.  Now $\{ \left( U_\alpha, (\Phi_\alpha^{*})^{-1} \right) \}_{\alpha\in \mathbf{A}}$ is an atlas of trivialization for $E^\prime$.  It follows that the restriction of $\Phi_\alpha^*$ to $U\times\mathbb{F}^0$ is a convenient isomorphism onto $F^0_{| U}$, which implies the result. Unfortunately, when $E^\flat\not=E^\prime$, since $E^\flat$ is only a weak subbundle of $E^\prime$, in general,  the restriction of $\Phi_\alpha^*$ to $U\times \mathbb{E}^\flat$ is not a convenient isomorphism onto $E^\flat_{| U}$.\\
\end{remark}
   
\begin{proof}[Proof of Proposition \ref{P_DSkewsymmetricMorphism2}] 
According to the assumption, we have $p(D_Q^F)=F$ and $p^\flat(D_Q^F)=Q(F)\oplus F^0$ closed in $E$ and $E^\flat$ respectively, and so $D_Q^F$ is a closed subbundle of $E^{\mathfrak{p}}$. According to Lemma \ref{L_PartialAnulatorProperties}, 2., it follows that it is an almost partial Dirac structure.

\begin{lemma}
\label{L_FormOnF}  
If $F$ is involutive, then, for any section $\gamma\in \Gamma(E^\flat_U)$ such that  $\gamma_{| F}=0$, $(L_{\rho(u)}\gamma)_{| F}=0$ and $(i_u d\gamma)_{| F}=0$ for any $u\in \Gamma(F_{U})$.
\end{lemma}

\begin{proof}[Proof of Lemma \ref{L_FormOnF}]
 Since the inclusion of $E^\flat$ in the dual bundle $E^\prime$ is an injective smooth morphism, any $\gamma\in \Gamma(E^\flat_U)$
 can be considered as a smooth $1$-form on $E_U$. Therefore,  for any $u$ and $v \in \Gamma(F_U)$, we have  $L_{\rho(u})\gamma(v)=\rho(\gamma(v))-\gamma([u,v]_E)$ and $i_ud\gamma(v)=\rho(u)(\gamma(v))-\rho(v)(\gamma(u))-\gamma([u,v]_E)$.\\
 Since $\gamma_{|F}\equiv 0$ and $[u,v]_E\in \Gamma(F_U)$, each term in the second member in each previous expression is identically null.
\end{proof}
 
Note that since $Q$ is a bundle convenient morphism, this implies that $\Omega $ (as defined in the Proposition) is a well defined   smooth $2$-form on $E$.\\ 
Over  any open set  $U$ and  $(r,\alpha)$ and any $(s,\beta)$ in $ \Gamma((D_Q^F)_U)$ we have (cf.  Remark \ref{R_OtherExpressionOfC}): 
\begin{equation}
\label{DDF1}
[(r,\alpha),(s,\beta)]_C=([r,s]_E,L_{\rho(r)}\beta- i_sd\alpha)).
\end{equation}

Thus  $D_Q^F$ is involutive  if and only if $[r,s]_E\in \Gamma(F_U)$ and the restriction of $L_{\rho(r)}\beta-i_s d\alpha-i_{[r,s]_E}\Omega$  to $F$ is identically null, for any $(r,\alpha)$ and $(s,\beta)$ in $\Gamma( (D_Q^F)_U)$. In particular $D_Q^F$ is involutive implies   $F$  involutive.\\

Assume that  $F$ is involutive. In particular, $F$ provided with  the anchor $\rho_{| F}$ and the bracket  $[.,.]_F$  which is the restriction of $[.,.]_E$ to $F$ $\rho_{| F}$ has a structure of convenient Lie algebroid. Then, for any $(r,\alpha)$ and $(s,\beta)$ in $\Gamma( (D_Q^F)_U)$, the restriction of
$\alpha-i_r\Omega$ and $\beta- i_s\Omega$ to $F$ are  identically null. Therefore, from Lemma \ref{L_FormOnF}, the restriction of  $L_{\rho(r)}\beta-L_{\rho(r)}(i_s\Omega)$ and of $i_sd\alpha-i_sd(i_r\Omega)$ to $F$ are identically null. The same is true for $i_sd\alpha-i_sd(i_r\Omega)$.\\

On the other hand, from  the classical definition of Lie derivative and exterior differential of forms and relation between Lie derivative, inner product and exterior derivative on a convenient Lie algebroid (cf. \cite{CaPe23}, Chap 3., 18.  and also  \cite{Marc16}, 2.11) we have:
\[
\begin{matrix}
L_{\rho(r)}i_s\omega+i_sd(i_r\Omega))&=\Omega \left( [s,r]_E \right) +L_{\rho(r)}i_s\Omega-i_s L_{\rho(r)} \Omega+i_s i_r d \Omega\hfill{}\\
&=i_{[r,s]_E} \Omega+i_s i_r d \Omega\hfill{}\\
\end{matrix}
\]

By the way obtain:
$$L_{\rho(r)}\beta-L_{\rho(r)}(i_s\Omega)-i_sd\alpha+i_sd(i_r\Omega)=L_{\rho(r)}\beta-i_sd\alpha- i_{[r,s]_E} \Omega+i_s i_r d \Omega.$$
 Since the restriction to the first member  is identically zero for any $r$ and $s$  in $\Gamma(F_U)$, it follows that $L_{\rho(r)}\beta-i_sd\alpha- i_{[r,s]_E} \Omega$ is identically null for any $r$ and $s$ in $\Gamma(F_U)$, if and only if the restriction of $\Omega$ to $F$ is closed.\\
This implies that $D_Q^F$ is involutive if and only if $F$ is involutive and $\Omega_{| F}$ is closed. 
\end{proof}

\subsection{Integrability of a partial almost Dirac structure}
\label{__IntegrabilityAlmostDIracAlgebroid}
In finite dimension, any  almost Dirac structure  $D$  on a Lie algebroid $E$ over $M$ which is involutive is integrable, that is the distribution $\rho^D(D)$ defines a (singular) foliation on $M$. This is a direct consequence that a finite dimensional Lie algebroid is integrable. Unfortunately, this result is not true in general in the convenient setting under more assumption even in the Banach setting. Thus we introduce 
\begin{definition}
\label{D_integrability} 
Let $D$ be an almost partial Dirac structure on a convenient Lie algebroid $(E,M,\rho,[.,.])$.  We say that $D$ is \emph{integrable} if the distribution $\rho^D(D)$ associated to the anchored bundle $(D,M,\rho^D)$ defines a foliation on $M$.
\end{definition}
  
In the Banach setting, if $D$ is involutive,  we have the following general result as Corollary of Theorem 8.4 in \cite{CaPe23}.
  
\begin{theorem}
\label{T_integrableD}  
Consider a partial almost involutive  Dirac structure $D$ on a Banach Lie algebroid  $(E,M,\rho,[.,.])$. If for each $x\in M$, the kernel of the anchor $\rho^D$ is supplemented in the fibre $D_x$ and $\rho^D_x(D_x))$ is closed in $T_xM$, then  $D$ is integrable.
\\ %In particular, if $D$ is an Hilbert bundle, $ D$ is always integrable.
\end{theorem}
  
In fact, we have a more general result of integrability of  Banach  partial almost Dirac structure. For this purpose, we introduce as in \cite{CaPe23}:

\begin{definition}
\label{D_PartialConvenientLieAlgebroid}
Let $ \left( E,\pi,M,\rho \right) $ be a convenient anchored bundle. Given a sheaf $\mathfrak{E}_M$ of subalgebras of the sheaf $C^\infty _M$ of smooth functions on $M$ and let $\mathfrak{P}_M$  be a sheaf of $\mathfrak{E}_M$ modules of sections of $E$.  Assume that  $\mathfrak{P}_M$ can be provided with a structure of Lie algebras sheaf  which satisfies, for any open set $U$ in $M$:
\begin{description}
\item[\textbf{(CPLA1)}]
the Lie bracket $[.,.]_{\mathfrak{P}(U)}$ on $\mathfrak{P}(U)$ only depends on the $1$-jets of sections of ${\mathfrak{P}(U)}$;
\item[\textbf{(CPLA2)}]
for any $(\mathfrak{a},\mathfrak{a}')\in \left( \mathfrak{P}(U) \right) ^2$ and any $f\in \mathfrak{E}(U)$, we have the compatibility conditions
\begin{eqnarray}
\label{eq_rhoCompatibilitySheaf}
[\mathfrak{a},f\mathfrak{a}']_{\mathfrak{P}(U)}=df(\rho(\mathfrak{a}))\mathfrak{a}'+f[\mathfrak{a},\mathfrak{a}']_{\mathfrak{P}(U)}
\end{eqnarray}
\item[\textbf{(CPLA3)}]
$\rho$ induces  a Lie algebras morphism from  $\mathfrak{P}(U)$ to $\mathfrak{X}(U)$, for any open set $U$ in $M$.
\end{description}
Then $ \left( E,\pi, M,\rho,\mathfrak{P}_M \right)$ is called a convenient partial Lie algebroid\index{convenient!partial Lie algebroid}. The family $\{ [.,.]_{\mathfrak{P}(U)},\; U \textrm{ open set in }M \}$ is called a sheaf bracket\index{sheaf bracket} and is denoted $[.,.]_E$. \\
A partial convenient Lie algebroid  $\left( E,\pi, M,\rho,\mathfrak{P}_M \right) $  is called strong\index{Lie algebroid!strong partial} \index{strong partial Lie algebroid} 
if for any $x\in M$, the stalk\index{stalk}
\[
\mathfrak{P}_x=\underrightarrow{\lim}\{ \mathfrak{P}(U),\;\; \varrho^U_V,\;\; U \textrm{ open neighbourhood of } x \}
\]
is equal to $\pi^{-1}(x)$ for any $x\in M$.
\end{definition}

We the have the following theorem (cf. \cite{CaPe23}, Chap. 8) 
\begin{theorem}\label{T_IntegrabilityPartialAlgebroid}
Let $ \left( E,\pi, M,\rho,\mathfrak{P}_M \right)$ be a partial Banach-Lie algebroid.  If for each $x\in M$, $\ker \rho_x$ is split in the fibre $ E_x$ and  $\rho(E_x)$ is closed in $T_xM$,  then $ \left( E,\pi, M,\rho,\mathfrak{P}_M \right)$  is integrable.
\end{theorem}
  
\begin{corollary}
\label{C_integrablePartialDirac} 
Let $D$ be a partial almost Dirac structure on a Banach Lie algebroid $ \left( E,\pi, M.,\rho,[.,.]_E \right) $.\\
Assume that $ \left( D,M,\rho^D,[\;,\;]_D \right) $ is a  partial Lie algebroid
and, for each $x\in M$,  $\ker\rho^D_x$ is split in $\rho^D(D_x)$, closed in $T_xM$. Then  $\mathcal{D}$ is integrable.
\end{corollary}
  
Note that there exists a partial almost  Banach Dirac structure on a Banach  anchored bundle which is not  involutive  but which is integrable as shown in the following example:
  
\begin{example}
\label{Ex_ DirectSumDirac} 
\textsf{Direct sum of partial Dirac structures on a product of Hilbert Lie algebroids}.\\
For $i \in \{1,2\}$, let $ \left( E_i,M_i,\rho_i,[.,.]_i \right) $ be a convenient Lie algebroid.  We consider the  anchored product vector bundle $\left( E=E_1\times E_2,M= M_1\times M_2, \rho=\rho_1\times \rho_2 \right) $. In finite dimension, on this  product  $(E, M, \rho)$, there exists a unique structure of Lie algebroid  "compatible" with the structure on each $E_i$ (cf. for instance  \cite{Mei17} for a direct simple proof).  The essential argument for this result in finite dimension,  is that the  module of local sections of the bundle product  $E$ is generated by pullbacks of local sections of $E_1$ and $E_2$. Therefore, it follows that, in finite dimension, the direct sum of Dirac structures is again a Dirac structure (cf. \cite{YoJa10} and \cite{YJM10}).
{\bf Unfortunately, this not any more true in infinite dimension} as we will see.\\

Let $U$ be any open set in $M$,  then  $U_i=\pi_i(U)$ is an open set of $M_i$ and $U$ is contained in $U_1\times U_2$.  A local section $\mathfrak{a}$ of 
 $E=E_1\times E_2$  defined on $U$ is called \emph{projectable}\index{projectable section} if there exists a section $\mathfrak{a}_i$ of  $E_i$ over $U_i$, for $ i\in\{1,2\}$ such that 
 $\mathfrak{a}(x_1,x_2)=(\mathfrak{a}
 _1(x_1),\mathfrak{a}_2(x_2))$ 
 for all $(x_1,x_2)\in U$. We denote by $\mathfrak{S}_U$ the set of projectable sections on $U$,  by  $\mathfrak{E}_U$ the algebra of smooth maps on 
 $U$ and $\mathfrak{P}_U$ the $\mathfrak{E}_U$-module generated by $\mathfrak{S}_U$.  Our purpose is to show that we can provide  $\mathfrak{P}_U$ with a Lie algebra structure by using \cite{CaPe23}, Chap.~3, 3.18.4.\\
 
\begin{observation}
 \label{ob_Strong} 
 Note that for any $z= \left( z_1,z_2 \right) \in M$, over any  simply connected chart domain $V=V_1\times V_2 $  where $V_i$ is a chart domain around $z_i$, the triviality of $E$ over such domains implies that, for any $ \left( u_1,u_2 \right) \in E_z$,  there exists  a projectable section $ \left( \mathfrak{a}_1,\mathfrak{a}_2 \right) $ such that $\mathfrak{a}_i(z_i)=u_i$ $i \in \{1,2\}$; in particular, this is true around any $z\in U$.
\end{observation}
According to this previous Theorem,  without loss of generality, we may assume that $U$ is of type $U_1\times U_2$ where $U_i$ is a simply connected chart domain in $M_i$. 
  
On $\mathfrak{S}_U$, we can define a Lie bracket $[.,.]_{\mathfrak{S}_U}$ by
\begin{equation}
\label{eq_LieBracketProjectable}
[ (\mathfrak{a}_1,\mathfrak{a}_2)), (\mathfrak{b}_1,\mathfrak{b}_2)]_{\mathfrak{S}_U} (x_1,x_2)
=
 \left( [\mathfrak{a}_1,\mathfrak{b}_2]_{E_1}(x_1),[\mathfrak{a}_2,\mathfrak{b}_2]_{E_2}(x_2) \right)
\end{equation}
 for $(x_1,x_2)\in U$. Then $[ ., .]_{\mathfrak{S}_U}$  takes values in $\mathfrak{S}_U$ and satisfies  the following properties:
\begin{enumerate}
\item[--]  
it is skew symmetric;
\item[--] 
it depends on the $1$-jet of sections of  $\mathfrak{S}_U$;
\item[--]
it satisfies the Jacobi identity;  
\item[--]   
$\rho=(\rho_1,\rho_2)$ is compatible with the bracket.
\end{enumerate} 
 
Therefore from \cite{CaPe23}, Chap.~3, Theorem~3.7, we can extend the bracket $[.,.]_{\mathfrak{S}_U}$ to a Lie bracket  $[., .]_{\mathfrak{P}_U}$ such that $\mathfrak{P}_U$ 
 has a Lie algebra structure, and  the restriction of $\rho$ to  $\mathfrak{P}_U$  is a Lie algebra morphism into the Lie algebra $\mathcal{X}(U)$ of vector fields on $U$. By the way, we obtain a 
 sheaf of convenient Lie algebras $\mathfrak{P}_{M}$ on $M$ and so $ \left( E,\pi,M,\rho, \mathfrak{P}_M \right) $ is a partial convenient Lie algebroid.  Moreover, Observation~\ref{ob_Strong} implies that  it is a strong partial convenient Lie algebroid.    
Unfortunately,  if $U$ is any open set in $M$, the Lie bracket $[.,.]_{\mathfrak{P}_U}$ cannot be extended  to sections in  $\Gamma(E_{| U})$ in general.

However, we can note that if $E_i$ is a closed subbunlde of $TM_i$,  the classical Lie bracket on  $T(M_1\times M_2) $ will extend the Lie bracket of projectable sections of $E_1\times E_2$. \\

{\bf Note that  $E_1\times E_2$ being isomorphic to $E_1\oplus E_2$ over $M$,  all the previous results are also true for the direct sum $E_1\oplus E_2$.}\\

Now, assume that, for $i\in\{1,2\}$,  $D_i$ is a Dirac structure relative to the Pontryagine bundle $E_i\oplus E_i^\flat$, and we set  
\[
D:=D_1\oplus D_2\;
=\left\lbrace 
((x_1,x_2), u_1+ u_2, \alpha_1+\alpha_2),\;(x_i,u_i,\alpha_i)\in D_i,\textrm{ for } i\in\{1,2\}
\right\rbrace. 
\]
According to Lemma~\ref{L_PartialAnulatorProperties}, 4.,  $D$ is a partial 
 almost Dirac structure on the anchored bundle  $E_1\oplus E_2$.  Then $D$ can be provided with an anchored bundle \textit{via} the anchor $\rho^D=\rho_1\circ{p_1}_{| D_1}\oplus \rho_2\circ  {p_1}_{| D_1}$. 
 But, by assumption, for $i\in\{1,2\}$, each $D_i$ is a partial Dirac structure, so it  has a convenient Lie algebroid structure (cf. Theorem~\ref{T_DLieAlgebroid}). Therefore, from
 the previous result on  the direct sum of Lie algebroids, it follows that  $D$ can be provided with a strong  partial convenient Lie algebroid. Now, if we assume that $E_i$ and $E_i^\flat$ are Hilbert 
 bundles, then  $\ker \rho^D_x$ is supplemented in each fibre $D_x$ and $\rho_i^{D_i}(( E_i)_{x_i})$ is closed in $T_{x_i}M$, for all $x=(x_1,x_2)\in M$ by Corollary~\ref{C_integrablePartialDirac}, $D$ is a partial almost Dirac structure which is integrable, but not involutive since  there is no Lie algebroid structure on $E=E_1\oplus E_2$, in general,  and so the Courant bracket cannot be defined on $E$. \\
 
However, note  that   if $E_i$ is a closed   subbundle of $TM_i$, then $D_i=E_i\oplus E_i^0$ is a partial Dirac structure  relative to any Pontryagine bundle $TM_i\oplus T^\flat M_i$ (cf. Corollary~\ref{C_XX0Omega}) which is involutive if and only if $E_i$ is involutive  as a subbundle of $TM_i$. Moreover, in the Banach setting,  $E_i$ is involutive, then it is integrable (cf.\cite{CaPe23}) and from Frobenius theorem, the converse is true  when $E_i$ is supplemented in $TM_i$.  As  $(E_1\oplus E_2)^0=E_1^0\oplus E_2^0$, this implies that  $D_1\oplus D_2= (E_1\oplus E_2)\oplus (E_1\oplus E_2)^0$ is a  partial Dirac structure on  the Lie algebroid  $TM_1\times TM_2$, which, in the Banach setting, implies the integrability of the almost partial Dirac structure $D_1\oplus D_2$.
\end{example} 
 
\section{Partial Dirac structures on a convenient manifold}
\label{_PartialDiracStructuresOnAConvenientManifold}

\emph{In finite dimension, to a Dirac structure is associated a singular foliation and each leaf is provided with a presymplectic structure (cf. \cite{Bur13}).\\
In general, in the convenient setting (even in the Banach one), such results  are not true for  partial Dirac structures.\\
In this section, we will show that, under some  assumptions (which are satisfied in many classical examples), we recover analog  results. }

\subsection{Pre-symplectic convenient manifold} \label{___Pre-Symplectic}

\begin{definition} 
\label{D_PreSymplecticForm} 
A closed $2$-form $\Omega$ on $M$ is called
\begin{enumerate}
\item[(i)]   
\emph{pre-symplectic}\index{pre-symplectic!$2$-form} if 
\begin{description}
\item[\textbf{(PS1)}]
the kernel of $\Omega^\flat$ has a convenient bundle structure;
\end{description}
\item[(ii)] \emph{strong pre-symplectic} if 
\begin{description}
\item[\textbf{(PS2)}]
$\ker \Omega^\flat$ is supplemented;
% and  $\Omega^\flat$ induces a convenient isomorphism  from  $TM/(\ker\Omega^\flat)$ to $\Omega^\flat(TM)$;
\item[\textbf{(PS3)}]
the inclusion of $\Omega^\flat(TM)$ in $T^\prime M$ is an injective convenient  bundle morphism.
\end{description}
\end{enumerate}
\end{definition}
 
\begin{remark}
\label{R_widehatOmegasharp}${}$
\begin{enumerate}
\item[1.] 
Definition \ref{D_PreSymplectic}  is compatible with the classical definition of a   pre-symplectic form on a finite dimensional manifold.
%\item[2.] In the Banach setting  for the condition (PS2) if  $\ker\Omega^\flat$ is supplemented  we always have  that $\Omega^\flat$ induces a convenient isomorphism  from  $TM/(\ker\Omega^\flat)$ to $\Omega^\flat(TM)$,but  in  general, in convenient setting, this  is not always true.   
\item[2.] 
If  the condition \textbf{(PS1)} and \textbf{(PS2)} are  satisfied, the condition \textbf{(PS3)} is not satisfied in general, even in the Banach setting  (cf. Marsden example in \cite{Mars72}).
\item[3.] 
If conditions \textbf{(PS1)} and \textbf{(PS3)} are satisfied, this does not imply that  condition \textbf{(PS2)} is satisfied (cf. Example~\ref{Ex_Skewsymmetricmorphism}).
\end{enumerate}
\end{remark} 
As in finite dimension, $\ker\Omega^\flat$ is called the \emph{kernel} $\Omega$ and will be denoted $\ker \Omega$. We then have the following result:
\begin{lemma}
\label{P_PSProperties} 
If $\Omega$ is a pre-symplectic form on a manifold, then $\ker\Omega$ is an involutive bundle.
\end{lemma}
The proof of  this lemma uses the same arguments as in finite dimension  and is left to the reader as an exercise.

\begin{definition} 
\label{D_Foliation}
A \emph{foliation}\index{foliation} $\mathcal{F}$ on a manifold $M$ is a partition of $M$ into weak closed submanifolds called \emph{leaves}\index{leaf} of $\mathcal{F}$. The foliation $ \mathcal{F}$ is called \emph{regular}\index{regular!foliation}, if there exists an involutive closed subbundle of $TM$ whose maximal integral manifolds are the leaves of $\mathcal{F}$, otherwise the foliation is called \emph{singular}\index{singular!foliation}.
\end{definition}
 
Since $\ker\Omega$ is a closed involutive subbundle of $TM$, from \cite{CaPe23},  Theorem~8.4, we obtain:
\begin{corollary}
\label{C_IntegrableKernel}
If $\Omega$ is a pre-symplectic form on a Banach manifold $M$,  then $\ker \Omega$ defines a regular foliation\footnote{Here, we do not impose that $\ker\Omega$ is split contrary to Darboux's Theorem assumptions.}.
\end{corollary}

Let $\Omega$ be a pre-symplectic form on $M$.  A function $f$ defined on an open set $U\subset M$ is called \emph{admissible}\index{admissible!function} if there exists a vector field $X$ on $U$ such that $\Omega^\flat(X)=df$, but such  a vector field is not unique; it is defined modulo a vector field in $\ker \Omega$.  The set of such admissible functions is denoted by $\mathfrak{A}_\Omega(U)$. In fact, for a strong pre-symplectic form, we have:

\begin{proposition}
\label{P_AlgebraAdmissible} 
Let $\Omega$ be  a strong pre-symplectic form on a manifold $M$ which satisfies assumption~\emph{\textbf{(PS3)}} and denote $T^\flat M=\Omega^\flat(TM)$. Then we have:
\begin{enumerate}
\item[1.]
$\mathfrak{A}_\Omega(U)$ is  the algebra of smooth functions $f\in C^\infty(U)$ such that $df$ induces a section of $T^\flat M$. In fact, $\mathfrak{A}_\Omega(U)$ is the set of functions $f$ such that each iterated derivative $d^{k}f(x)\in L_{\operatorname{sym}}^{k}(T_{x}M,\mathbb{R})$ ($k \in \mathbb{N^\ast}$) at any $x \in M$ and satisfies\footnote {Cf. \cite{CaPe23}, Chap.~3.}:
\begin{equation}
\label{eq_dkfPartialDiracStructures}
\forall (u_2,\dots,u_k) \in (T_xM)^{k-1},\;
d^{k}_xf(.,u_{2},\dots,u_{k}) \in T_{x}^{\flat}M.
\end{equation}
\item[2.]
For any open set $U\in M$, and any $f\in\mathfrak{A}_\Omega(U)$, there exists a vector field $X_f$ on $U$ such that $\Omega^\flat(X_f)=df$. \\
The bracket
\[
\{f,g\}_\Omega:=df(X_g)
\]
is well defined and $ \left( \mathfrak{A}_\Omega(U),\{\;,\;\}_\Omega \right) $ has a structure of Lie-Poisson algebra.
\end{enumerate}
\end{proposition}
 
\begin{proof} 

1. According to properties \textbf{(PS2)} and \textbf{(PS3)}, consider a convenient subbundle $E$ such that $TM=\ker \Omega\oplus E$.\\
Then we have a convenient quotient bundle 
$TM/(\ker\Omega)$ and $\Omega^\flat$ induces a convenient quotient isomorphism bundle $\widehat{\Omega^\flat}:TM/(\ker\Omega)\to T^\flat M$. Moreover, the restriction $Q$  to $E$  of the projection of $TM$ onto  $TM/(\ker\Omega)$  gives rise to a convenient bundle isomorphism $Q:E\to TM/(\ker\Omega)$.\\
Now, if  $f: U\to \mathbb{R}$ is such that $df$ induces a section of $T^\flat M$, from the previous consideration, there exists a vector field $X_f$ on $U$ such that $\Omega^\flat(X_f)
=df$: take $X_f=(\widehat{\Omega^\flat} \circ J)^{-1}(df)$ where $J$ is the injection $E \to TM$. It follows that $\mathfrak{A}_\Omega(U)$ is  an algebra. \\

At first, it is clear that if $f:U\to \mathbb{R}$ is a smooth function which satisfies the relations  (\ref{eq_dkfPartialDiracStructures}), this implies that $df$ induces a section of $T^\flat M$ over $U$.\\

Since the converse  is a local problem, we may assume that $x=0$,  $U$ is a $c^\infty$-open set in $\mathbb{M}$. By the way,   $TM=U\times \mathbb{M}$, $\; T^\prime M=U\times \mathbb{M}^\prime$ and $\Omega^\flat$ is a smooth map from $U$ to the convenient space of skew symmetric linear maps  from $\mathbb{M}$ to $\mathbb{M}^\prime$. If $
\mathbb{M}^\flat$ (resp. $\mathbb{E}$, $\mathbb{K}$) is the typical fibre of $T^\flat M$ (resp. $E$, $\ker\Omega$), we may assume that the inclusion of $\mathbb{M}^\flat$ in $
\mathbb{M}^\prime$ is bounded, $\mathbb{E}$ and $\mathbb{K}$ are closed convenient subspaces of $\mathbb{M}$ with $\mathbb{M}=\mathbb{K}\oplus \mathbb{E}$. Therefore we have:

$T^\flat M=U\times\mathbb{M}^\flat$, 
$\;\ker\Omega^\flat=U\times \mathbb{K}$,
$\;\mathbb{M}=\mathbb{K}\times \mathbb{E}$\\
and so $TM=U\times \mathbb{K}\times \mathbb{E}$.

Under this context, we can choose $U$ as a product $U_1\times U_2\subset \mathbb{K}\times \mathbb{E}$ and each vector field $X$ on  $U$ can be written as a pair  $ \left( X_1,X_2 \right) $ such that $X_1 \left( y_1,y_2 \right) $ (resp. $Y_2 \left( y_1,y_2 \right) $) belongs to $T_{(y_1,y_2)}U_1\times\{y_2\}=\{(y_1,y_2)\}\times \mathbb{K}$ (resp. $T_{(y_1,y_2)}U_1\times\{y_2\}=\{(y_1,y_2)\}\times\mathbb{E})$. Thus, from the property of $f$, it follows that, for any fixed  $y_2\in U_2$,  $Y(f)(y_1,y_2)=0$ for all $y_1\in U_1$ and for any vector field $Y$  tangent  on  $U_1\times\{y_2\}$ to this open set. This implies that such a function $f$ is independent of the first variable. Since  $f$  only depends on the second variable, then, for all $(u_2,\dots,u_k) \in (T_yM)^{k-1}=\{(y_1,y_2)\}\times \mathbb{M}^{k-1}$, 
\[
d^{k}_yf(.,u_{2},\dots,u_{k})=d^{k}_yf(.,\bar{u}_{2},\dots,\bar{u}_{k}).
\]
wwhere, for $i \in \{2,\dots,k\}$, $\bar{u}_i$ is the projection of $u_i$ on $E_y=\{(y_1,y_2)\}\times \mathbb{E}$.\\
So, by induction on $k$, we can show that
\begin{center}
$d^{k}_yf(.,\bar{u}_{2},\dots,\bar{u}_{k})$ belongs to  $ T_y^\flat M=\{(y_1,y_2)\}\times \mathbb{M}^\flat.$
\end{center}
 
2. We have already seen that $X_f$ exists on $U$. Then for $f,g\in  \mathfrak{A}_\Omega(U)$ we set 
\[
\{f,g\}_\Omega=X_g(df)=\Omega(X_f,X_g).
\]
From the last member, it follows that this  bracket is well defined, skew symmetric  and takes values in $\mathfrak{A}_\Omega(U)$ according to  \cite{CaPe23}, Proposition~7.4.\\ 
Since $d(fg)=fdg+gdf$, the Leibniz property is satisfied. Now, as in finite dimension, using classical properties of the Lie derivative  (cf. \cite{KrMi97}), we have:
\[
\begin{matrix}
d\left(\{X_f,X_g\}_\Omega\right)&=d \Omega (X_f,X_g)\hfill{}\\
&=d \left( i_{X_f}i_{X_g} \Omega \right) \hfill{}\\
&=L_{X_g}i_{X_f}\Omega-i_{X_g}d i_{X_f}\Omega\hfill{}\\
 &=L_{X_g}i_{X_f}\Omega-i_{X_g}L_{X_f}\Omega \hfill{}\\
 &=-i_{[X_f,X_g]}\Omega\hfill{}
 \end{matrix}
\]
So we get
\[
d\left(\{f,\{g,h\}_\Omega\}_\Omega\right)=-i_{[X_f,[X_g,X_h]]}\Omega
\]
 which implies the Jacobi identity.
\end{proof}
  
\bigskip
As a particular case of pre-symplectic form, we have:
\begin{definition}
\label{D_ContactCase}  
A $1$-form $\omega$ on a convenient  manifold $M$ is called
\begin{enumerate}
\item[(i)] a \emph{contact form}\index{contact form}\index{form!contact} if
\begin{description}
\item[\textbf{(C1)}]
if $\Omega=-d\omega$, then $ \ker\Omega$ is a line sub-bundle of $TM$;
\end{description}
\item[(ii)] a \emph{strong contact form} if 
\begin{description}
\item[\textbf{(C2)}]
$TM=\ker\omega\oplus \ker \Omega$;
\item[\textbf{(C3)}]
$\Omega^\flat(TM)$ is a convenient bundle such that its inclusion in $T^\prime M$\\
is an injective convenient morphism.
\end{description}
\end{enumerate}
\end{definition}
The following result is direct consequence of Definition \ref{D_ContactCase}:
\begin{lemma}
\label{L_ContactProperty}
If $\omega$ is a (strong) contact form on $M$, we have the following properties
\begin{enumerate}
 \item[1.] $\Omega=d\omega$ is a (strong) pre-symplectic form on $M$;
 \item[2.] there exists a unique vector field $X$ on $M$ such that $i_X\omega=1$ and $i_X \Omega=0$\footnote{Such a vector field is called the Reeb\index{Reeb vector field} or Cartan  vector field.}.
\end{enumerate}
\end{lemma}
 
We now look for results on  projective limits (resp. direct limits)  of sequences of pre-symplectic forms on projective limits (resp direct limits) of sequence of Banach manifolds.\\

\begin{definition}
%\label{D_ProjectiveDirectLimit}
\label{D_ProjectiveLimitOfPresymplecticForms}
Let $\left( M_n,\lambda_n\right)_{n\in\mathbb{N}}$ be a projective sequence of Banach manifolds such that $\lambda_{n}:M_{n+1}\to M_n$ is an injective immersion.
A sequence $\left(  \Omega_{n}\right)_{n\in\mathbb{N}}$ of pre-symplectic forms $\Omega_{n}$ on  $M_n$ is called  \emph{coherent}\index{coherent!sequence},  if for each $x=\underleftarrow{\lim}x_n$,
 we have:
\[
\forall n \in \mathbb{N},\; (\Omega_{n+1})_{x_{n+1}}=\lambda_n^*  \left( (\Omega_n)_{x_n} \right) .
\]
\end{definition}
\smallskip
 
In a parallel way we introduce:
\begin{definition}
\label{D_AscendingSequencePreSymplectic}
Let $ \left(  M_{n},\varepsilon_n,\right) _{n\in\mathbb{N}}$ be an ascending sequence of paracompact\footnote{The assumption of paracompactness implies that the direct limit of the sequence $ \left(  M_{n},\varepsilon_n,\right) _{n\in\mathbb{N}}$ is a  Hausdorff convenient manifold convenient  (cf. \cite{CaPe23}, Chap. 3). Note that without this assumption all the results in this section are also true for direct limits.} Banach manifolds, i.e.  $\varepsilon _n:M_{n}\to M_{n+1}$ is the inclusion of the split submanifold $M_n$ in $M_{n+1}$.\\
A sequence $\left(  \Omega_{n}\right)  _{n\in\mathbb{N}}$ of pre-symplectic  forms $\Omega_{n}$ on  $M_n$ is called  \emph{coherent} if, for each $x=\underrightarrow{\lim}x_n$, we have 
\[
\forall n \in \mathbb{N}, (\Omega_{n})_{x_n} = \varepsilon _n^* ((\Omega_{n+1})_{x_{n+1}} ).
\]
\end{definition}
\smallskip
We then  have  the following results:
\begin{proposition}\label{P_ProjectiveDirectLimitPS}${}$
\begin{enumerate}
\item[1.]
Consider a sequence $\left(  {\Omega}_{n}\right)_{n\in\mathbb{N}}$ of coherent  (strong) pre-symplectic forms
$\Omega_{n}$  of $M_n$ on a projective sequence of Banach manifolds  $\left(  M_{n},\lambda_{n}\right) _{n\in\mathbb{N}}$.\\
Then $\Omega=\underleftarrow{\lim}(\Omega_n)$ is well defined and is a (strong) pre-symplectic  form on  the Fr\'echet manifold  $M=\underleftarrow{\lim}{M_n}$.
\item[2.]
Consider a sequence $\left(  {\Omega}_{n}\right)  _{n\in\mathbb{N}}$ of coherent  (strong) pre-symplectic forms
$\Omega_{n}$  of $M_n$ on an ascending  sequence of paracompact  Banach manifolds  $\left(  M_{n},\varepsilon_n \right)$.\\
Then $\Omega=\underrightarrow{\lim}(\Omega_n)$ is well defined and is a (strong) pre-symplectic on the  convenient manifold $M=\underrightarrow{\lim}{M_n}$.
\end{enumerate}
\end{proposition}

The proof of these properties is an  exercise   using results  on projective and direct limits  of Banach bundles  and  arguments on tensors of type (2,0) given in \cite{CaPe23}
and is left to the reader.\\
 
We end this section with some examples of convenient  pre-symplectic manifolds.

\begin{example}
\label{Ex_Skewsymmetricmorphism}  
Let $\mathbb{M}$ be a Banach space provided with a Riemannian metric $g$. We denote by $\mathbb{M}^\flat$ the range of the linear map $g^\flat:\mathbb{M}\to \mathbb{M}^\prime$ 
defined by $g^\flat(u)=g(u,.)$.   
Since $g^\flat $ is an injective continuous map, we provide  $\mathbb{M}^\flat$ with the norm  such that $g^\flat $ is an isomorphism and so the 
inclusion $\mathbb{M}^\flat \to \mathbb{M}'$ is continuous. Consider any linear partial  Poisson structure $P:\mathbb{M}^\flat \to \mathbb{M}$; it follows that $Q:\mathbb{M}\to \mathbb{M}^\flat$ defined by $Q=g^\flat\circ P\circ g^\flat$ is a continuous skew symmetric linear map.\\
Let $U$ be any open set in $\mathbb{M}$ and  denote  by  $g^U$ (resp. $P^U$) the weak 
Riemannian metric  (resp. anchor Poisson)  on $U$ defined by $(g^U)_x(u,v)=g(u,v)$ (resp. $P^U_x(\alpha)=P(\alpha)$)  for all $x\in U$ and $u,v$ in $T_xU=\{x\} \times \mathbb{M}$ 
(resp. $\alpha \in T_x^\flat U=\{x\}\times \mathbb{M}^\flat$). Then  if $Q^U=(g^U)^\flat\circ P^U\circ (g^U)^\flat$, it is easy to see that the $2$-form  $\Omega$ on $U$, defined by 
$\Omega(u, v)=<Q^U(v), u>$, is a closed $2$-form. Its kernel is $((g^U)^\flat)^{-1}(\ker P^U)$ and  it is involutive and so \textbf{(PS1)} is satisfied\footnote{In general, $\ker\Omega$ is not split.}. 
Note that if $\mathbb{K}=(g^\flat)^{-1}(\ker P)$, then $\ker \Omega_x=\{x\}\times K$. If $N=U\cap \mathbb{K}$, then the leaf through $x\in U$  defined by $\ker\Omega$  is  $\{ \{x\}\times N, x\in U \}$.\\
We have $\Omega^\flat=Q^U$. Now since  $Q(\mathbb{M})=g^\flat(P(\mathbb{M}^\flat))$ and  $\Omega^\flat(TU)=U\times Q(\mathbb{M})$, this implies that the second property of the assumption \textbf{(PS3)} is satisfied using the facts that  the inclusion $P(\mathbb{M}^\flat) \to  \mathbb{M}$ is continuous and $g^\flat$ is an isometry.\\
Assume now that $\ker P$ is supplemented in $\mathbb{M}^\flat$ by  $g^\flat(P(\mathbb{M}^\flat))$. Since $g^\flat$ is an isometry,  $\mathbb{K}$ is also supplemented  in $\mathbb{M}$. Therefore,  we have a canonical isomorphism $\widehat{Q}$ from $\mathbb{M}/\mathbb{K}$.  From  the triviality of $TU$, $\ker\Omega$ and $\Omega^\flat(TU)$, it follows that  \textbf{(PS2)} is satisfied, which ends the proof of the fact that  $\Omega$ is strong  pre-symplectic. Note that if $\ker P$ is  not supplemented in $\mathbb{M}^\flat$, $\Omega$ is only pre-symplectic but not strong  pre-symplectic.
\end{example}

\begin{example}
\label{Ex_BanachContactFormOnLoop} 
Let  $M$  be  an $n$-dimensional manifold provided with a Riemannian metric $g$. We denote by $\mathcal{L}(M)$ the Fr\'echet manifold of smooth loops $\gamma: \mathbb{S}^1\to M$.  The tangent space $T_\gamma\mathcal{L}(M)$ can be identified  with the set $\Gamma(\mathbb{S}^1,\gamma^* TM)$ of smooth sections of $\gamma^* TM$ over $\mathbb{S}^1$. For any $x\in M$,  we denote by 
\begin{center}
$\mathcal{L}_x(M)=\{\gamma\in \mathcal{L}(M): \gamma(0)=x\}$.
\end{center}
It is a finite co-dimensionnal manifold of $\mathcal{L}(M)$ whose tangent space  $T_\gamma \mathcal{L}_x(M)$ can be identified with 
$\left\lbrace X\in \Gamma(\mathbb{S}^1,\gamma^* TM): X(0)=0 \right\rbrace $.
 Finally, we denote by  $\mathcal{L}_{x,1}(M)$ the hypersurface in $\mathcal{L}_x(M)$ which is the set of loops $\gamma\in 
\mathcal{L}_x(M)$ whose $g$-length is $1$.\\
To $g$ is associated a strong Riemannian metric $g_\mathcal{L}$ defined by
\[
(g_\mathcal{L})_\gamma(X,Y)
=\displaystyle\int_0^1 g_{\gamma(t)} \left( X(t),Y(t) \right) dt.
\]
We consider the $1$-form $ \mu$ on $\mathcal{L}(M)$  given by:
\[
\mu_\gamma(X)=\displaystyle\frac{1}{2}\int_0^1g_{\gamma(t)}(X(t),\dot{\gamma}(t))dt.
\]
Then the restriction $\Omega$ of $d\mu$ to each manifold $ \mathcal{L}_x(M)$ is a symplectic form
 %V. Muñoz, F. Presas, Geometric structures on loop and path spaces ,  Published 23 January 2008 Mathematics Proceedings - Mathematical Sciences
  (cf. \cite{MuPr08}, 2.1). Following this paper, a vector field on $M$ is called \emph{gradient like}\index{gradient like} if $g^\flat(X)$ is a closed $1$-form $\mu$. Assume that there exists on $M$ a vector field $X$ which is  gradient like such that $L_Xg=g$. Then the restriction of $\mu$ to any  hypersurface  $\mathcal{L}_{x,1}(M)$ is a contact  form  (cf \cite{MuPr08}, Proposition~3.2). In particular, the restriction of $d\mu$ to  $\mathcal{L}_{x,1}(M)$ is pre-symplectic. In fact, it is a strong pre-symplectic form.\\
Indeed, according to \cite{Wur95}, 
%Tilmann Wurzbacher, 
%Symplectic geometry of the loop space of a Riemannian manifold. Journal of Geometry and Physics Volume 16, Issue 4, July 1995, Pages 345--384 
let $H^{k,2}(M)$ be the  Hilbert  manifold of loops $\gamma: \mathbb{S}^1\to M$ of Sobolev class $H^{k,2}$. As previously,  for $x\in M$ fixed, we can define the Hilbert submanifolds $H^{k,2}_x(M)$ and $H^{k,2}_{x,1}(M)$  and the $1$-form $\mu$ on  $H^{k,2}(M)$ is also well defined. The arguments above  (exposed in  \cite{MuPr08}) still work in this context; in particular, the restriction of $\alpha$ to $H^{k,2}_{x,1}(M)$ is a contact form and, in particular, the restriction of $d\mu$ to $H^{k,2}_{x,1}(M)$ is pre-symplectic and its line kernel is split. Thus,  from Remark~\ref{R_widehatOmegasharp}, 2., condition \textbf{(PS2)} is satisfied. On the other hand, it easy to see that the restriction  $\Omega$ of  $d\mu$ to the Hilbert manifold $H^{k,2}_x(M)$ is a strong symplectic form, so $\Omega^\flat(T(H^{k,2}_x(M))=T^\prime(H^{k,2}_x(M))$. This implies that $\Omega^\flat(T(H^{k,2}_{x,1}(M))$ is a $1$-codimensional subbundle of $T^\prime(H^{k,2}_{x,1}(M))$ and so \textbf{(PS3)} is satisfied. Using the fact that $\mathcal{L}(M)$  is the projective limit of the sequence of Hilbert manifolds $ \left( H^{k,2} \right) _{k\geq 1}$ (cf.\cite{Pel21}), the announced result comes from  Proposition~\ref{P_ProjectiveDirectLimitPS}, 1. 
\end{example}
 
\begin{example}
\label{Ex_DirectLimitLoop} 
In continuation with the end of the previous example, consider  an ascending sequence of finite dimensional Riemannian manifolds $ \left( M_n,g_n \right) _{n\in \mathbb{N}}$ such that for each $n$, the metric $g_n$ is the restriction of $g_{n+1}$ to $M_n$. Then  each Hilbert  manifold $H^{k,2}(M_n)$ is paracompact\footnote{See \cite{Pel18}.} and  is provided with a Riemannian metric $\hat{g}_n$ and a $1$-form $\mu_n$ as in the previous example. Moreover  $H^{k,2}(M_n)$ is a Hilbert submanifold of  $H^{k,2}(M_{n+1})$ and $\hat{g}_n$ is nothing but the restriction of $\hat{g}_{n+1}$ to  $H^{k,2}(M_n)$. Thus if $M=\underrightarrow{\lim}
 M_n=\displaystyle\bigcup_{n\in \mathbb{N}} M_n$, we have:
 $$H^{k,2}(M):=\underrightarrow{\lim} H^{k,2}(M_n)=\displaystyle\bigcup_{n\in \mathbb{N}}H^{k,2}(M_n).$$ 
 Moreover, if $x=\underrightarrow{\lim}x_n\in M$, it is clear that we also have
 \[  
H^{k,2}_x(M):=\underrightarrow{\lim} H^{k,2}_{x_n}(M_n)=\displaystyle\bigcup_{n\in \mathbb{N}}H^{k,2}_{x_n}(M_n)
\] 
and also
\[
H^{k,2}_{x,1}(M):=\underrightarrow{\lim} H^{k,2}_{x_n,1}(M_n)=\displaystyle\bigcup_{n\in \mathbb{N}}H^{k,2}_{x_n,1}(M_n).
\]
If $\Omega_n$ is the restriction of $d\mu_n$ to $H^{k,2}_{x_n}(M_n)$, it is easy to see that the restrictions of $\Omega_n$ give rise to a coherent sequence of pre-symplectic  forms on the ascending sequence $ \left( H^{k,2}_{x_n,1}(M_n) \right) $ of Hilbert manifolds and so,  from Proposition~\ref{P_ProjectiveDirectLimitPS}, 2., $\Omega=\underrightarrow{\lim} \Omega_n$ in restriction to $H^{k,2}_{x,1}(M)$ is a strong pre-symplectic form.
\end{example}

\begin{example}
\label{Ex_ContactStructure}
On the kinematic cotangent bundle $p_M^\prime:T^\prime M \to M$ of a convenient manifold, we have a canonical $1$-form, \emph{the Liouville form}\index{Liouville form}\index{form!Liouville} $\theta$ characterized by 
$$\theta(X)=<p_{T^\prime M}(X), Tp_M^\prime(X)>.$$
where $p_{T^\prime M}:T(T^\prime M)\to T^\prime M$.\\
Then $\omega=-d\theta$ is a weak symplectic form which is strong if and only if $M$ is modelled on a reflexive convenient space $\mathbb{M}$. In this case,  the set of $1$-jets  of 
smooth local functions on $M$ is isomorphic to the bundle $q_{M}^\prime:T^\prime M\times\mathbb{R}\to M$. If an element  of $T^\prime M\times\mathbb{R}$ is written $(x,\alpha,t)
$, then $\xi=dt-\theta$ is a strong contact $1$-form  on $T^\prime M\times\mathbb{R}$ whose Reeb vector field\index{Reeb vector field} is $\displaystyle \frac{\partial}{\partial t}$,  which  means that $
\Omega=d\xi$ is a  strong pre-symplectic form on $T^\prime M\times\mathbb{R}$ whose kernel is generated by $\displaystyle \frac{\partial}{\partial t}$.
\end{example}

\subsection{Partial Dirac Manifold}

As in the framework of partial Poisson structures (cf. \cite{CaPe23}), let $p^{\flat}:T^{\flat} M\to M$ be a weak sub-bundle of the kinematic cotangent bundle $p_{M}^{\prime}:T^{\prime}M \to M$.
\\
The \emph{partial kinematic Pontryagine bundle}\index{partial Pontryagine bundle} is the bundle 
$T^{\mathfrak{p}}M = TM \oplus T^\flat M$.\\
As in the previous section,  for $E=TM$,  this bundle is equipped with:
\begin{enumerate}
\item[$\bullet$]
a pairing $<<.,.>>$ defined as follows:\\
for any $x \in M$, any pair $ \left( X_x,Y_x \right) $ of $T_xM$ and any pair $ \left( \alpha_x,\beta_x \right) $ of $T_x^\flat M$,
\begin{eqnarray}
\label{eq_pairringTM}
%\label{eq_DirectSymmetricBilinearFormAlmostPartialDiracStructures}
<<(X_x,\alpha_x),(Y_x,\beta_x)>>
= 
\beta_x(X_x) + \alpha_x(Y_x)
\end{eqnarray}
\item[$\bullet$]
From Proposition \ref{P_PropertyCourantBracket}, a \emph{Courant bracket}\index{Courant bracket}\index{bracket!Courant}
\begin{equation}
\label{eq_CourantTM}
[(X,\alpha),(Y,\beta)]_C=\left([X,Y],L_X\beta-L_Y\alpha-\displaystyle\frac{1}{2}d(i_Y\alpha -i_X\beta)\right)
\end{equation}
for all sections $(X,\alpha), (Y,\beta)\in \Gamma(T^{\mathfrak{p}}M_U) $, where $[\;,\;]$ is the usual Lie bracket of vector fields.
\end{enumerate}

A \emph{partial almost  Dirac structure on $M$}\index{partial almost Dirac structure} is then a  closed subbundle  $D$ of $T^{\mathfrak{p}}M$.\\
When $T^\flat M=T^\prime M$ a partial almost  Dirac structure on $M$ is simply called \emph{an almost Dirac structure}\index{almost Dirac structure}.\\
The restriction of the previous Courant bracket to a partial almost Dirac (resp. almost Dirac) structure $ D$ is again denoted by $[\;,\;]_D$ and $D$ is called \emph{a partial Dirac (resp. Dirac) structure}  if the module  $\Gamma(D_U)$  of sections of $D$ on any open set $U$ of $D$ is involutive according to $[\;,\;]_D$; in other words, this means that  the partial almost Dirac structure $D$ on the Lie algebroid $TM$ is involutive. \\
%It is a\emph{ partial Dirac structure} if $D$ is  involutive. 
A smooth function $f$ defined on an open set $U$  is called \emph{admissible}\index{admissible!function} for the partial (almost) Dirac structure $D$ if there exists a vector field $X$ on $U$ such that $(X,df)$ is a section of $D$ over $U$.

\begin{example}
\label{Ex_LinearDirac} 
Let $\mathbb{E}$ be a convenient space, $\mathbb{E}^\flat$ a convenient  subspace of $\mathbb{E}^\prime$ and $\mathbb{D}$ a partial linear Dirac structure on $\mathbb{E}$. Then $D=\mathbb{D}\times \mathbb{E}$ is a partial Dirac structure on $\mathbb{E}$.
\end{example}

\begin{example}
\label{ex_Presymplectic}
Let $\Omega$ be a $2$-form on a convenient manifold. Assume that the range $ T^\flat M=\Omega^\flat(TM)$    of the associate morphism $
\Omega^\flat:TM\to T^\prime M$ is a weak sub-bundle of $T^\prime M$. Then the graph $D_\Omega$ of $\Omega$ is the set 
\[
\{(x,v,\alpha), x\in M, v\in T_xM, \alpha=\Omega^\flat(u)\in T_x^\flat M\}
\] 
which is a partial almost Dirac structure (cf. Example \ref{ex_SkewsymmetricMorphism2}). 
If $\Omega$ is closed, from Corollary \ref{C_XX0Omega}, 2., it follows that $D_\Omega$ is a partial Dirac structure on $M$. 
Assume that $\Omega$ is a strong pre-sympectic form. To such a Dirac structure $D_\Omega$, on any 
open set $U$ in $M$, we can associate an algebra $\mathfrak{A}_\Omega(U)$ of admissible functions (cf. Proposition~\ref{P_AlgebraAdmissible}). The restriction $[\;,\;]_{D_\Omega}$ of the Courant bracket to $D_\Omega$ will permit to define a Poisson bracket on each algebra $\mathfrak{A}_\Omega(U)$ \textit{via} the projection $p^\flat$ which is exactly the Poisson 
bracket defined in Proposition  \ref{P_AlgebraAdmissible}. 
For any $f\in \mathfrak{A}_\Omega(U)$, there exists a Hamiltonian vector field $X_f$  characterized by $\Omega^\flat(X_f)=df$. Therefore, any function $f$  defined on an open set $U$ in $M$ is an admissible function for the Dirac structure  $D_\Omega$ if and only $f\in  \mathfrak{A}_P(U)$ and so $ (X_f,df)$ is a section of $D_\Omega$ over $U$.
\end{example}

\begin{example}
\label{ex_PartialPoisson} 
Recall that a partial convenient Poisson structure\index{partial Poisson structure} on a convenient manifold $M$ corresponds to the  following data (cf. \cite{CaPe23}):
\begin{description}
\item[$\bullet$] 
a convenient bundle which is a vector subbundle $T^\flat M$ of $T^\prime M$ and such that the inclusion is a convenient (injective) bundle morphism;
\item[$\bullet$] an anchor $P:T^\flat M\to TM$ which is skew symmetric relatively to the natural pairing between $T^\flat M$ and $TM$ called an almost Poisson anchor;
\item[$\bullet$]
the Schouten bracket\footnote{As defined in \cite{CaPe23}.} $[P,P]$ vanishes.
\end{description}
When the typical fibre $\mathbb{M}^\flat\subset \mathbb{M}^\prime$ is separating on $\mathbb{M}$  or if $P$ is surjective, then $D_P$ is a partial almost Dirac structure (cf. example~\ref{ex_SkewsymmetricMorphism1}). Note that when $T^\flat M=T^\prime M$, we have $\mathbb{M}^\flat =\mathbb{M}^\prime$ and so {\bf  we recover the classical result in finite dimension  that a Poisson manifold is a Dirac structure on $M$}.\\
From now on, assume that $D_P$ is a partial almost Dirac structure. \\
For any open set $U$  in $M$, let $\mathfrak{A}_P(U)$ be the set of smooth functions as defined in  Proposition \ref{P_AlgebraAdmissible}. To any $f\in \mathfrak{A}_P(U)$, we can associate a Hamiltonnian vector field $X_f=P(df)$ and so we have a Lie Poisson algebra structure whose Poisson bracket is given by $\{f,g\}_P=df(X_g)$ which satisfies the Jacobi identity.\\
We want to show  that $D_P$ is stable for the bracket  (\ref{eq_CourantTM}). For this purpose, it is sufficient to show that $T_{D_P}\equiv 0$ (cf. Proposition \ref{P_CNSCourant bracketD}). But for all functions $f,g,h\in  \mathfrak{A}_P(U)$, we have (cf. \cite{Marc16}, 2.15): 
%https://www.math.ru.nl/~imarcut/index_files/Dirac.pdf (2016)
$$\begin{matrix}
{\bf T}_{D_P}\left((X_f,df), (X_g,dg),(X_h,dh)\right)
&=<<[(X_f,df),(X_g,dg)]_C, (X_h,dh)>>\hfill{}\\
&=<<([X_f,X_g]d(i_{X_f}dg)), (X_h,dh)>>\hfill{}\\
&=[X_f,X_g](h)+X_h(X_f(g))\hfill{}
\end{matrix}
$$
where one has:\\
$[X_f,X_g](h)+X_h(X_f(g))=\{f,\{g,h\}_P\}_P-\{g,\{f,h\}_P\}_P+\{h,\{f,g\}_P\}_P=0$.\\
Since the Jacobi property is satisfied on $ \mathfrak{A}_P(U)$, for all open sets $U$ in $M$, 
this implies that $D_P$ is a partial Dirac  structure. Of course, $f$  defined on the open set $U$ is an admissible function for this Dirac structure  if and only if $f\in  \mathfrak{A}_P(U)$ and $(X_f,df)$ is a section of $D_P$ over $U$.
\end{example}
 
According to the canonical Lie algebroid structure on $TM$ whose anchor is the identity, if $p:T^{\mathfrak{p}}M\to TM$ is the natural projection  and $D$, as in the previous section,  is a partial Dirac structure, we define $\rho^D: D\to TM$ by
$\rho^D(X,\alpha)=p(X,\alpha)$. From Theorem~\ref{T_DLieAlgebroid}~, we have:

\begin{theorem}
\label{T_DLieAlgebroidTM}
 A partial almost Dirac structure $D$ on a convenient manifold $M$  is a partial Dirac structure if and only if  $ \left( D,M,p_{| D}, [.,.]_{ D} \right) $ is a convenient Lie algebroid.
\end{theorem}

Thus, if $D$ is a partial Dirac structure on $M$, the distribution $\mathcal{D}=\rho^D(D)$ is called \emph{the characteristic distribution}\index{characteristic distribution} of $D$. In this case, the partial Dirac structure is called \emph{regular}\index{regular} if the characteristic distribution is a subbundle of $TM$, otherwise it is called \emph{singular}\index{singular}.\\
In finite dimension, the characteristic distribution of a Dirac structure is always integrable and defines a singular foliation. Unfortunately, in general, in the convenient setting,  the characteristic distribution of a partial Dirac structure $D$ is not integrable. The same is true even  in the Banach setting without  more assumptions. For instance, according to Lemma~\ref{L_PDPflatDClosed} and   \cite{CaPe23}, Theorem~8.4, we have:

\begin{theorem}
\label{T_FoliationD} 
Let  $D$ be a partial Dirac structure on a Banach manifold $M$.\\
If $\ker p_{| D}$ is supplemented in each fibre, then the characteristic distribution $\mathcal{D}$ is integrable.
\end{theorem}

\subsection{Integrability and pre-symplectic  structures on leaves} 

Even in finite dimension (cf. Example~\ref{Ex_TwistedPoisson}), for some almost partial Dirac structures $D$ which are not involutive, the characteristic distribution is integrable. We will see that, in the convenient setting, such situations also exist (cf. $\S$~\ref{_ProjectiveAndDirectLimitsOfBanachDiracStructures}). 

\begin{example}
\label{Ex_TwistedPoisson}
(cf. \cite{GrXu12}, Example 2.4.4) 
%  Melchior Grutzmann, Xiaomeng Xu, Cohomology for almost Lie algebroids arxiv.org/abs/1206.5466
Let $(M,\pi)$ be a Poisson manifold and consider a $3$-form $H$ on $M$. Then we define the Schouten bracket 
\[
[\pi,\pi] (df,dg,dh)=H(\pi(df),\pi(dg),\pi(dh))
\]
 for all $f,g,h\in C^\infty (M)$ where $\pi^\flat:T^*M\to TM$ is the canonical morphism associated to $\pi$. Then the graph of $\pi$ is an almost Dirac structure on $M$ but not a Dirac structure: it is only  a twisted Dirac structure according to the twisted Courant bracket 
$$[(X,\alpha),(Y,\beta)]_C=\left([X,Y],L_X\beta-L_Y\alpha-\displaystyle\frac{1}{2}d(i_Y\alpha -i_X\beta)\right)+H(X,Y,.)$$
However the characteristic distribution is integrable (cf. \cite{GrXu12}, Theorem \ref{T_IntegrableDPreAlgebroid}).
\end{example}

As in $\S$~\ref{__IntegrabilityAlmostDIracAlgebroid}, we also introduce:

\begin{definition} 
\label{D_PartialIntegrableDirac}  
Let $D\subset T^\mathfrak{p}M$ be a partial almost Dirac structure on a  convenient manifold $M$. We say that $D$ is a partial integrable Dirac structure on $M$ if the characteristic distribution $\mathcal{D}$ associated to the anchored bundle $(D,M,p_{| D})$ is integrable.
\end{definition}

As an application of Theorem \ref{T_integrableD}  in this  context of partial almost Dirac structure on a Banach manifold, we obtain:
 
\begin{theorem}
\label{T_IntegrableDPreAlgebroid} 
Let  $D$ be a partial almost  Dirac structure on a Banach manifold $M$.\\
Assume that the anchored bundle  $ \left( D,M,p_{| D}\right) $  has a structure of strong partial Lie algebroid  and  $\ker p_{| D}$ is split in each fibre, then the characteristic distribution $\mathcal{D}$ is integrable.\\
In particular, in the Hilbert setting, if $ \left( D,M,p_{| D}\right) $  has a structure of strong partial Lie algebroid, then the characteristic distribution distribution is integrable.
\end{theorem}

Note that Example~\ref{Ex_TwistedPoisson} can be generalized easily to the Banach setting and so gives an illustration of Theorem~\ref{D_PartialIntegrableDirac}.\\

For any partial integrable Dirac structure on a convenient manifold,  under  some additional  assumptions,  we will show that an almost pre-symplectic structure  is defined  on each  leaf of the characteristic  foliation, which generalizes a classical result in finite dimension (cf. \cite{Cou90},  \cite{Bur13} or \cite{Marc16}).\\

More precisely, if we  denote by $p^\flat$ the natural projection of $T^{\mathfrak{p}}M$ on $T^\flat M$, we have: 

\begin{theorem}
\label{T_GeometricStructuresOnLeaves} 
Let $D$ be a partial almost  integrable Dirac structure on a convenient manifold $M$. Assume that $\ker p_{| D}$  and $\ker p^\flat_{| D}$ are supplemented  in each fibre. Consider a leaf $L$ of the characteristic foliation. Then we have:
\begin{enumerate}
\item[1.] 
$T^\flat L:=\displaystyle\bigcup_{x\in L}p^\flat(D_x)$ is a closed weak convenient subbundle of $T^\prime M_{| L}$. 
\item[2.] 
There exists a  skew symmetric convenient bundles morphism\\ $P_L:TL\to T^\prime L$  whose kernel is $(D\cap TM)_{| L}\subset TL$ and whose range is  the  weak subbundle $T^\flat L\cap T^\prime L$.\\
In particular, $\Omega_L(u,v)=<P_L(v), u>$ is a  $2$-form on $L$ whose kernel is  $(D\cap TM)_{| L}$ and  we have $\Omega^\flat =P_L$. Moreover, if $D$ is a partial Dirac structure, $\Omega_L$ is closed and so is   strong pre-symplectic.
\item[3.] 
There exists a  skew symmetric convenient bundle morphism $P_L^\flat$ from $T^\flat L$ to its dual bundle $T ^{\flat\prime} L$  whose kernel is $(D\cap T^\flat M)_{| L}\subset T^\flat L$ and whose range is  the  weak subbundle $T^{\flat\prime} L\cap TL$ of $T^{\flat\prime} L$. Moreover, if $D$ is a partial Dirac structure,  $P^\flat_L: T^\flat L \to (T^{\flat\prime} L\cap TL)\subset TL$ is a partial Poisson structure on $L$.
\item[4.]  
$P_L$ (resp. $P^\flat_L$) induces a convenient isomorphism $\widehat{P}_L$ (resp. $\widehat{P}^\flat_L$) from $T^\flat L\cap T^\prime L$ to $TL\cap T^{\flat\prime}$ (resp. 
$TL\cap T^{\flat\prime}$ to $T^\flat L\cap T^\prime L$) and $P_L^{-1}=P^\flat_L$.\\
Moreover, $D_L\cap TM_{L}=\{0\}$ if and only if $D_L\cap T^\flat M=\{0\}$ and then $P_L$ is an 
isomorphism from $TL$ to $T^\flat L$ and, if $D$ is a partial Dirac structure,   $\Omega_L$ is weak symplectic  form on $L$.
\end{enumerate}
\end{theorem}

\begin{remark}
\label{R_AssumptionsSatisfied}
In finite dimension, all the previous  assumptions are satisfied and so Theorem~\ref{T_GeometricStructuresOnLeaves} is a generalization of classical results on Dirac manifolds (cf. \cite{Cou90}, \cite{YoMa06}, \cite{Bur13} and \cite{Marc16}). \\
In the Hilbert setting, the assumptions of Theorem~\ref{T_GeometricStructuresOnLeaves} are always satisfied.\\
In the Banach setting, under the assumption of %Theorems~\ref{T_IntegrableD} or
Theorem~\ref{T_IntegrableDPreAlgebroid}, if $D\cap TM$ is split in each fibre, then  Theorem~\ref{T_GeometricStructuresOnLeaves} can be applied.
\end{remark}

\begin{remark}
\label{R_ResultAndExample} 
In this Remark, we assume that $D$ is a partial Dirac structure.
\begin{enumerate}
\item[1.]
According to Theorem ~\ref{T_DiracPoisson}, 4., situation (i),  Example~\ref{ex_PartialPoisson}  corresponds to the case where $D\cap T^\flat M=\{0\}$  and  $p^\flat: D\to T^\flat M$ is an isomorphism,  $\Omega_L$ is a symplectic form on each leaf $L$ of the characteristic distribution, $TL$ is isomorphic to the quotient bundle $D_{| L}/(D_{| L}\cap T^\flat M_{| L})$. In fact, the associated partial Poisson anchor  $P:T^\flat M\to TM$ is defined in the following way:
 
for any $x\in M$, $x$ belongs to one and only one leaf $L$ and we set\\ $P_x(\alpha)=P^\flat_L\circ (p_x^\flat)^{-1}(\alpha)$ which belongs to $T_xL\subset T_xM$.
\item[2.] 
According to Theorem~\ref{T_DiracPoisson}, 4., situation (ii),  Example~\ref{ex_Presymplectic} corresponds to the case where $D\cap T^\flat M =\{0\}$, 
  the partial Dirac structure is regular and 
 $p^\flat: D\to T^\flat M$ is an isomorphism.
\item[3.] The common situation of Point 1 and Point 2 corresponds to  Theorem  \ref{T_DiracPoisson}, 4., situation (iii). In this case  $p:D\to TM$ and $p^\flat D\to T^\flat  M$ are isomorphisms and  $\Omega$ is a  symplectic form on $M$.
\end{enumerate}
\end{remark}

\begin{proof}[Proof of Theorem \ref{T_GeometricStructuresOnLeaves}]${}$\\
\noindent 1.  
Since by assumption, the characteristic distribution $D$ is integrable,  if $L$ is a leaf, it follows that $p(D_{|L})=TL$. Since $D$ is a partial  almost Dirac structure, $ \left( D,M,p_{|D} \right) $ is a convenient Lie anchored bundle and so its restriction $(D_{|L}, L, p_{| D_{|L}})$ is also a convenient  anchored bundle, and since $D$ is a closed subbundle of $T^\mathfrak{p} M$, it follows that 
$TL$ is a closed convenient subbundle of $TM_{| L}$ from Lemma~\ref{L_PDPflatDClosed}.

On the other hand, by assumption, the kernel of $p^\flat_{| D_{|L}}$ is supplemented in each fibre and it follows that $\ker p^\flat_{| D_{|L}}$ is a subbundle of $D_L$\footnote{Same proof as in \cite{BGJP19},  Proposition~6.5.}.  Using the adapted arguments in the convenient setting of those used in the Banach framework in \cite{Lan95}, Chap.~III, $\S$3, we can show that $D_{| L}/\ker p^\flat_{D_{| L}}$ is a convenient bundle over $L$.

Using Lemma~\ref{L_PDPflatDClosed}, it follows that  for each $x\in L$, we have  an isomorphism $(\hat{p}^\flat_ {D_{| L}})_x)$ form $((D_{| L})_x)/(\ker (p_{D_{| L}})_x)$   
to $T_x^\flat L=p^\flat(D_x)\subset T_xM_{| L}$ where $T_x^\flat L$ is closed in $T_xM_{| L}$. Now this isomorphism is a bounded injective linear map  $(\hat{p}^\flat_ 
{D_{| L}})_x$ from  $(\hat{p}^\flat_ {D_{| L}})_x$ to $T_x^\flat M_{| L}$ for each $x\in L$. Since  $D_{| L}/\ker p_{D_{| L}}$ is a convenient bundle over $L$ from the 
smooth uniform boundedness principle (cf. \cite{KrMi97}), it follows that  we get an injective bounded bundle morphism $\hat{p}^\flat_ {D_{| L}}$ from  $D_{| L}/\ker p^\flat_{D_{| L}}$ to $T^\flat M_{| L}$ and so its range   $T^\flat L=\displaystyle\bigcup_{x\in L} T_x^\flat L$ is a convenient closed  subbundle of  $T^\flat M_{| L}$.\\

\noindent 2. We define 

$P _L: TL\to T^\flat L$ by $(P_L)_x(u)=\alpha_{| T_xL}=P_{T_xL}(u)$ (cf. Theorem~\ref{T_DiracPoisson})\\
\noindent for any $\alpha\in T_x^\flat M$ such that $(u,\alpha) $ belongs to  $D_x$ and any $x\in L$.  
Note that  $P_L$ takes values in $T^\prime L$ and since $P_L$ is a bounded linear map in restriction to each fibre, again according to the smooth uniform boundedness principle, it follows that $P_L$ is a convenient morphism from $TL$ to $T^\prime L$. .\\

 From  Theorem \ref{T_DiracPoisson}, 1., $(P_L)_x : T_x L\to T_x^\flat L\cap T_x^\prime L$ is well defined and is a skew symmetric surjective  bounded linear map  on $T_x^\flat L\cap T_x^\prime L$ whose kernel is $D_x\cap T_xM=\ker p^\flat_{| D_x}\subset T_x L$. This implies that $P_L:TL\to T^\prime L$ is a convenient skew symmetric  morphism whose range is $T^\flat L\cap T_x^\prime L$. Therefore $\Omega_L(u,v)=<P_L(v), u>$ is a  $2$-form on $L$ whose kernel is  $(D\cap TM)_{| L}$ and  we have $\Omega^\flat =P_L$.\\
 On the other hand,   by assumption, we have $D_x=(D_x\cap T_xM)\oplus F_x$ for some convenient closed subspace $F_x$ of $D_x$. But, $T_xL=p(D_x)=(D_x\cap T_xM)+ p(F_x)$.
But since $(D_x\cap T_xM)\cap F_x=\{0\}$, this implies that $\ker (P_L)_x$ is supplemented in $T_xL$. As previously for $T^\flat L$,  it follows that $\ker P_L$ is a subbundle of $D_L$ and we obtain   that $TL/\ker P_L$ is a convenient bundle over $L$. By the way (again as for $T^\flat  L$), we obtain also  an injective convenient  bundle morphism $\widehat{P}_L: TL/\ker P_L \to T^\prime L$ and so its range is  $T^\flat L\cap T^\prime L$. In particular, since the injection of $T_x^\flat L\cap T_x^\prime L$ in $T^\prime_x L$ is a convenient injective bounded linear map (cf. Proof of Theorem~\ref{T_DiracPoisson}, 1.),  it follows that  $T^\flat L\cap T^\prime L$ is a weak subbundle of $T^\prime L$.
 
If $D_L$ is a partial Dirac structure on $M$, then $\left( D, M, p_{| D} [\;,\;]_D \right) $ is  a convenient Lie algebroid. Therefore its restriction $(D_{|L}, L, p_{| D_{|L}} [\;,\;]_D)$ to $L$  is  a convenient Lie algebroid and so    $D_L$ is a partial Dirac structure on $L$  (cf. Theorem~\ref{T_DLieAlgebroidTM}).  Now, from (\ref{eq_GraphPL}), for $x\in L$,   $D_x$ is the graph of $(\Omega_L)_x$, it follows that $D_L$ is the graph of $P_L$. Since   $\Omega_L^\flat=P_L$ and so $\Omega_L^\flat(TL)=T^\flat L\cap T^\prime L$ is a weak subbundle of $T^\prime L$ (as we have seen previously) and moreover, $\ker \Omega_L= D\cap TM$ is supplemented in $TL$. 

Since   $D_L$ is a partial Dirac structure on $L$ and $D_L$ is the graph of $P_L$,from Corollary~\ref{C_XX0Omega}, 2., it follows that $\Omega_L$ is closed and so $\Omega_L$ is a strong pre-symplectic form on $L$.\\

\noindent 3.  In a "dual way",  from Theorem \ref{T_DiracPoisson}, 2.,  by analog arguments we can show that  $P^\flat _L$ defined by
$P^\flat_L(\alpha)=u_{| TL}$,
for all $\alpha\in T_x^\flat L$ and $u\in T_xL$ so that $(u,\alpha)$ belongs to $D_x$ for all $x\in L$.\\
By analog arguments  as in the proof of 2., we can show that  $P^\flat _L$ is a convenient bundle morphism from $TL$ to $T^{\flat\prime}L $ which is skew symmetric relative to the canonical pairing between $T^\flat L$  and $TL$. Moreover, the kernel of $P_L^\flat$ is the subbundle of $T^\flat L$ defined by $(D\cap T^\flat M)_{| L}$ and it range is the quotient bundle $TL/(D\cap TM)_{| L}$ which is isomorphic to $T^{\flat\prime}L\cap TL$ which is a weak subbundle of $T^{\flat\prime}L$ (cf. the argument at the end of the proof of Theorem \ref{T_DiracPoisson}, 2. and use the smooth uniform boundness principle).\\
If   $D$ is a partial Dirac structure,  we have already seen that  $D_{| L}$ is a partial Dirac structure on $L$. From (\ref{eq_GraphPflatL}) it follows that $D_L$ is the graph of $P^\flat _L$.  Now, from example \ref{ex_PartialPoisson}, it follows that $P^\flat_L$ is a partial Poisson structure on $L$.\\
 
 \noindent 4. According to the properties of $P_L$ and $P^\flat_L$ proved in Point 2  and Point 3, the assertions in  Point 4  of Theorem~\ref{T_GeometricStructuresOnLeaves}  are   a direct consequence of Theorem~\ref{T_DiracPoisson}, 3. and 4.
\end{proof}

\begin{example}
\label{Ex_RegularDiracmanifolds}
\textsf{Regular partial almost Hilbert  Dirac structure.}
  
Let $D$ be a partial almost Dirac structure on  a Hilbert  manifold $M$ relative to a the Hilbert Pontryagin bundle $TM\oplus T^\flat M$, and assume that $p(D)$ is a closed subbundle of $TM$\footnote{$D$ is then called a {\bf regular partial almost structure}.}.  Then the assumption of Theorem \ref {T_GeometricStructuresOnLeaves} are satisfied.  Assume that $D$ is involutive and let $L$  be a leaf of $D$.   Consider  any $1$-form $\alpha$ on $M$ which is a section of $T^\flat M$ on  some open set $U$ and such that $\alpha_{| L}=0$. From Lemma \ref{L_FormOnF}, we have $L_X\alpha_{| L}=0$ and $i_Xd\alpha_{| L}=0$ for any section $X$ of $L_{|U}$. In particular, this is true for any  section $\alpha$ of $L^\flat=P^\flat (D)$ over $U$. Then from (\ref{DDF1}), $D$ will be a partial Dirac structure on $M$ if $L$ is involutive. In this case,  on each leaf of $L$, we have a structure of strong pre-symplectic Hilbert  manifold.\\
  In particular consider Hilbert spaces $\mathbb{E}$ and $\mathbb{E}^\flat$ and a partial linear Dirac structure $\mathbb{D}$ on $\mathbb{E}$. Then $T\mathbb{E}=\mathbb{E}\times \mathbb{E}$  and  $D=\mathbb{D}\times \mathbb{E}$ is a partial almost Dirac structure on the manifold $\mathbb{E}$ and $p(D)=\mathbb{L}\times \mathbb{E}$ if $\mathbb{L}=p(\mathbb{D})$.  Thus $L$ is an integrable subbundle of $T\mathbb{E}$ and from the previous considerations, it follows that $D$ is a partial Dirac structure on the manifold  $\mathbb{E}$. The characteristic leaf of $D$ through $x\in \mathbb{E}$ is $\{x\}\times \mathbb{L}$ and the strong  pre-symplectic form is $\Omega_{\mathbb{L}}$, according to notations of Theorem \ref{T_DiracPoisson}, 2.
\end{example}

\begin{example}
\label{Ex_SumBanachDiracstructures}
\textsf{Sum of partial Dirac structures on a product of Hilbert manifolds.}\\
For $i\in\{1,2\}$, let $M_i$ be a Banach manifold and $D_i$ be a partial Dirac structure on $M_i$ relative to the Pontryagine bundle $TM_i\oplus T^\flat M_i$. From Example~\ref{Ex_ DirectSumDirac}, $D_1\oplus D_2$ is a partial almost Dirac structure  which is integrable if the kernel of  ${p_i}_{| D_I}$ is supplemented in each fibre of $TM_i$ which is assumed from now on. Note that this assumption  is always true in the Hilbert setting.  From now on, we assume that $M_i$ is a Hilbert manifold. Since $D_i$ is a partial Dirac structure on $M_i$, all the assumptions of Theorem~\ref{T_GeometricStructuresOnLeaves}  are satisfied. Therefore, the characteristic distribution of $D_i$ is integrable, and if $L_i$ is a leaf, we have a presymplectic form $\Omega_{L_i}$ on $L_i$ and a partial Poisson structure $P^\flat _{L_i}:T^\flat L_i\to TL_i$.\\ 
Moreover, $D:=D_1\oplus D_2$ is a partial almost Dirac structure on $M:=M_1\times M_2$ which is integrable. 
The leaf of the characteristic distribution of $D$ through a point $z=(z_1,z_2)$ is the product $L_1\times L_2$ where $L_i$ is the characteristic leaf of $D_i$ through $z_i$. From Theorem~\ref{T_GeometricStructuresOnLeaves}, we have a skew symmetric bundle morphism $P_L:TL\to T^\prime L$ defined by $(P_L)_x(u)=\alpha_{| T_xL}=P_{T_xL}(u)$
\noindent for any $\alpha\in T_x^\flat M$ such that $(u,\alpha) $ belongs to  $D_x$ and any $x\in L$. Therefore we have $P_L=P_{L_1}\oplus P_{L_2}$ and so the kernel of $P_L$
is $(D_1\cap TM_1)_{| L_1}\oplus (D_2\cap TM_2)_{| L_2}$.
So the $2$-form $\Omega_L$ on $L$ is $\Omega_{L_1}\oplus \omega_{L_2}$.  It follows that $\Omega_L$ is also a strong pre-symplectic form on $L$.\\
 By analogue arguments, we also have a partial Poisson structure $P^\flat _{L}:T^\flat L\to TL$ which is $P^\flat_{L_1}\oplus P^\flat_{L_2}$.\\
\end{example}

\section{Projective and direct limits of Banach Dirac structures}
\label{_ProjectiveAndDirectLimitsOfBanachDiracStructures}
%\emph{ In this section we always have $\mathbb{E}^\flat=\mathbb{E}^\prime$. So we will always speak of (almost ) Dirac structure.}\\
\subsection{Case of linear Dirac structures on Banach spaces}
\label{__LinearDiracSequences}
For $i\in\{1,2\}$, we consider a  partial Pontryagin space  $\mathbb{E}_i^\mathfrak{p}=\mathbb{E}_i\oplus \mathbb{E}^\flat_i$ associated  to a {\bf Banach} space $\mathbb{E}_i$  provided with the pairing $<<\;,\;>>_i$. We denote by $p_i$ (resp. $p_i^\flat$) the canonical projections of  $\mathbb{E}_i^\mathfrak{p}$ to $\mathbb{E}_i$ (resp. $\mathbb{E}_i^\flat$). At first, we have:
\begin{lemmadefinition}
\label{L_PulbackPushward} 
Let $\phi : \mathbb{E}_1 \to \mathbb{E}_2$ be a linear map such that $\phi(\mathbb{E}_1)$ is closed in $\mathbb{E}_2$.\\
Assume that $\mathbb{E}_i^\flat$ is closed in $\mathbb{E}_i^*$  for $i \in \{1,2\}$ and  $\phi^*(\mathbb{E}_2^\flat )\subset  \mathbb{E}_1^\flat$.\\
Let $\mathbb{D}_i$ be a linear Dirac structure on $\mathbb{E}_i$, for $i \in \{1,2\}$. 
\begin{enumerate}
\item 
Assume that $\phi$ is injective and $\phi^*(\mathbb{E}_2^\flat)=\mathbb{E}_1^\flat $ and
\[
\phi^!(\mathbb{D}_2):=\{(u_1,\phi^*(\alpha_2))\in \mathbb{E}_1^\mathfrak{p}\;: (\phi(u_1), \alpha_2)\in \mathbb{D}_2\}.
\]
If $\mathbb{D}_1 \subset \phi^!({\mathbb{D}}_2)$ then $\mathbb{D}_1=\phi^!({\mathbb{D}}_2)$  and so $\phi^!({\mathbb{D}}_2)$ is a linear Dirac structure on $\mathbb{E}_1$.
In this situation,
$\phi^!({\mathbb{D}}_2)$ is called the \emph{pullback}\index{pullback} of $\mathbb{D}_2$.
\item 
Assume that $\phi$ is surjective and  set 
\[
\phi_!(\mathbb{D}_1):=\{(\phi(u_1),\alpha_2)\in \mathbb{E}_2^\mathfrak{p}\;: (u_1, \phi^*(\alpha_2))\in \mathbb{D}_1\}
\]
 If  $\mathbb{D}_2\subset\phi_!({\mathbb{D}}_1)$ then  $\mathbb{D}_2=\phi_!({\mathbb{D}}_1)$  and so $\phi_!({\mathbb{D}}_1)$ is a linear Dirac structure on $\mathbb{E}_2$.

In this situation, $\phi_!(\mathbb{D}_1)$
is called the \emph{pushforward}\index{pushforward} of $\mathbb{D}_1$.
\end{enumerate}
\end{lemmadefinition}

\begin{proof}
Note that  the adjoint map $\phi^*:\mathbb{E}_2^\prime\to \mathbb{E}_1^\prime$ is continuous. Since  $\mathbb{E}_i^\flat$ is closed in $\mathbb{E}_i^*$  for $i\in\{1,2\}$ and $\phi^*(\mathbb{E}_2^\flat ) \subset  \mathbb{E}_1^\flat$, it follows that $\phi^*_{| \mathbb{E}_2^\flat }$ is continuous.\\
\noindent 1. 
We have
\[
\phi^!(\mathbb{D}_2)
= \phi^{-1} \left( p_2(\mathbb{D}_2)\cap \phi(\mathbb{E}_1) \right) 
\times 
 \left( \phi^* \left( p_2^\flat(\mathbb{D}_2) \right) \right).
\]
Since $\phi(\mathbb{E}_1)$ is closed in $\mathbb{E}_2$, and $\mathbb{D}_2$ is closed in $\mathbb{E}_2^\mathfrak{p}$, it follows that 
 $p_2 \left( \mathbb{D}_2 \right) \cap \phi(\mathbb{E}_1)$ is closed  in $\mathbb{E}_2$ and since $\phi$ is continuous,  $p_1(\mathbb{D}_1)=\phi^{-1} \left( p_2(\mathbb{D}_2)\cap \phi(\mathbb{E}_1) \right)$ is closed in $\mathbb{E}_1$.  Now,
$\phi^*_{| \mathbb{E}_2^\flat }$ is surjective and continuous, thus this map is closed and so  is $p_2^\flat$. Therefore  $\phi^*(p_2^\flat(\mathbb{D}_2))$ is closed in $\mathbb{E}_1^\flat$. Thus $p^\flat_1(\mathbb{D}_1)= \phi^* \left( p_2^\flat(\mathbb{D}_2) \right) $ is closed in $\mathbb{E}_1^\flat$. This implies that  $\mathbb{D}_1$ is a closed subspace of $\mathbb{E}_1^\mathfrak{p}$.\\
Consider two elements $(u_1, \phi^*(\alpha_2))$ and $(v_1,\phi^*(\beta_2))$ of $\mathbb{D}_1$.  Since $\mathbb{D}_2$ is a linear Dirac structure,  we have
\[
\begin{array}{rcl}
<<(u_1, \phi^*(\alpha_2)),(v_1,\phi^*(\beta_2))>>_1
&=& \phi^*(\beta_2)(u_1)+\phi^*(\alpha_2)(v_1)\\
&=& <<(\phi(u_1),\alpha_2),(\phi(v_1),\beta_2)>>_2\\
&=&0.
\end{array}
\]
and so $\phi^!({\mathbb{D}}_2)\subset \phi^!({\mathbb{D}}_2)^\perp$. 
But, by assumption, 
 $\mathbb{D}_1\subset\phi^!({\mathbb{D}}_2)$; it follows that $\phi^!({\mathbb{D}}_2)^\perp\subset \mathbb{D}_1^\perp$.  But $\mathbb{D}_1^\perp=\mathbb{D}_1$ ($\mathbb{D}_1$ is a linear Dirac structure), which ends the proof.\\
\noindent 2. We have
$\phi_!(\mathbb{D}_1)=\phi(p_1(\mathbb{D}_1))\times \left( \phi^*_{| \mathbb{E}_2^\flat } \right) ^{-1}(p_1^\flat(\mathbb{D}_1))$.\\
Under the assumption of Point 2., it follows that $\phi$ and $p_1$ are closed. Since $\mathbb{D}_1$ is closed, so is $p_2(\mathbb{D}_2)=\phi(p_1(\mathbb{D}_1))$.  By same arguments, $p_1^\flat(\mathbb{D}_1)$ is closed in $\mathbb{E}_1^\flat$ and $\phi^*_{| \mathbb{E}_2^\flat }$ is continuous. Therefore $p_2(\mathbb{D}_2)= \left( \phi^*_{| \mathbb{E}_2^\flat } \right) ^{-1}(p_1^\flat(\mathbb{D}_1))$ is closed in $\mathbb{E}_2^\flat$. This implies that $\mathbb{D}_2$ is a closed subspace of $\mathbb{E}^\mathfrak{p}$.\\
Consider $(\phi(u_1),\alpha_2)$ and $(\phi(v_1),\beta_2)$ two elements of $\mathbb{D}_2$. As in the previous proof,  we have, since $\mathbb{D}_1$ is a linear Dirac structure:
\[
\begin{matrix}
\begin{array}{rcl}
<<(\phi(u_1), \alpha_2),(\phi(v_1),\beta_2)>>_2
&=&\beta_2(\phi(u_1))+\alpha_2(\phi(v_1))\\
&=&\phi^*(\beta_2)(u_1)+\phi^*(\alpha_2)(v_1)\\
&=&<<(u_1, \phi^*(\alpha_2)),(v_1,\phi^*(\beta_2))>>_1\\
&=&0.
\end{array}
\end{matrix}
\]
This implies that $\mathbb{D}_2\subset \mathbb{D}_2^\perp$.  With the same last  arguments as in the previous proof, we get $\mathbb{D}_2=\phi_!({\mathbb{D}}_1)$.
\end{proof}

\begin{example}
\label{Ex_ProductDirac} 
For $i \in \{1,2\}$, let $\mathbb{E}_i$ be a Banach space and $\mathbb{E}_i^\flat$ be a closed subspace of $\mathbb{E}_i^\prime$. Consider a partial linear Dirac structure 
 $\mathbb{D}_i\subset \mathbb{E}_i\oplus  \mathbb{E}^\flat_i$. We set $\widetilde{\mathbb{E}}_1=\mathbb{E}_1$ and $\widetilde{\mathbb{E}}_2=\mathbb{E}_1\oplus \mathbb{E}_2$ . If we consider 
  $\widetilde{\mathbb{E}}^\flat_2=\mathbb{E} _1^\flat\oplus \mathbb{E}_2^\flat$, then the direct sum\footnote{Cf. Lemma  ~\ref{L_PartialAnulatorProperties}, 4.} $\widetilde{\mathbb{D}}_2=\mathbb{D}_1\oplus \mathbb{D}_2$ is a closed subbunbdle of 
  $\widetilde{\mathbb{E}}_2^\mathfrak{p}=\widetilde{\mathbb{E}}_2\oplus \widetilde{\mathbb{E}}^\flat_2$.\\
Moreover, for $i \in \{1,2\}$, we have a canonical pairing $<.,.>_i$ between $\widetilde{\mathbb{E}}_i^\prime$ and $\widetilde{\mathbb{E}}_i$  defined by
 $<\alpha_1,u_1>_1=\alpha_1(u_1)$ and $<(\alpha_1,\alpha_2),(u_1,u_2)>_2=\alpha_1(u_1)+\alpha_2(u_2)$.  We have seen that  $\widetilde{\mathbb{D}}_2$ is a partial linear Dirac structure on 
 $\widetilde{\mathbb{E}}_2$. By homogeneity in notations, we set $ \widetilde{\mathbb{D}}_1=\mathbb{D}_1$ and we set $\widetilde{p}_i$ (resp. $\widetilde{p}^\flat_i$) the canonical projection of $\widetilde{\mathbb{E}}_i^\mathfrak{p}$ onto $\widetilde{\mathbb{E}}_i$ (resp. $\widetilde{\mathbb{E}}_i^\prime$).
 Let  $\varepsilon:\widetilde{\mathbb{E}}_1 \to \widetilde{\mathbb{E}}_2$ be the canonical inclusion and $\lambda: \widetilde{\mathbb{E}}_2\to \widetilde{\mathbb{E}}_1$ the canonical projection.  We have 
\begin{enumerate}
\item[1.] $ \varepsilon^! \left( \widetilde{\mathbb{D}}_2 \right) 
=
\left\lbrace(u_1,\alpha_1)\in \widetilde{\mathbb{E}}_1^\mathfrak{p}\;:  \left((u_1,0),(\alpha_1, \alpha_2)\right) \in 
\widetilde{\mathbb{D}}_2 \right\rbrace$.

Since  $\widetilde{p}_1(\widetilde{\mathbb{D}}_1)= \widetilde{p}_2 \left( \varepsilon^!\left( \widetilde{\mathbb{D}}_2 \right) \right) $ and $\widetilde{p}^\flat_1(\widetilde{\mathbb{D}}_1)= \widetilde{p}^\flat_1 \left( \varepsilon^!\left( \widetilde{\mathbb{D}}_2 \right) \right)$
 then, from Lemma~\ref{L_PulbackPushward},~1., we have $\widetilde{\mathbb{D}}_1=\varepsilon^! \left( \widetilde{\mathbb{D}}_2 \right) $
\item[2.] 
$ \lambda_! \left( \widetilde{\mathbb{D}}_1 \right) 
=
\left\lbrace 
\left((u_1,u_2),(\alpha_1,\alpha_2)\right)\in \widetilde{\mathbb{E}}_1^\mathfrak{p}\;: (u_1, \alpha_1))\in \mathbb{D}_1
\right\rbrace$.

  Since $\widetilde{\mathbb{D}}_2=\mathbb{D}_1\times \mathbb{D}_2$, this implies that $ \widetilde{\mathbb{D}}_2\subset  \lambda_! \left( \widetilde{\mathbb{D}}_1 \right) $. According to Lemma~\ref{L_PulbackPushward}, 2., we have  $\widetilde{\mathbb{D}}_1=\lambda_! \left( \widetilde{\mathbb{D}}_2 \right) $.
\end{enumerate}
\end{example} 

\begin{definition}
\label{D_AscendingLinearDiracSequence}
Consider an ascending sequence   $ \left( \mathbb{E}_n,\varepsilon_n \right) _{n \in \mathbb{N}}$ of Banach spaces\footnote{This means that $\varepsilon_n:\mathbb{E}_n\to \mathbb{E}_{n+1}$ is the inclusion and $\mathbb{E}_n$ is closed in $\mathbb{E}_{n+1}$.}, provided with a sequence of Pontryagin spaces $ \left( \mathbb{E}_n^\mathfrak{p}=\mathbb{E}_n \oplus \mathbb{E}_n^\flat \right) _{n \in \mathbb{N}}$  and, for each $n\in \mathbb{N}$, we have  a linear Dirac structure $\mathbb{D}_n\subset\mathbb{E}_n^\mathfrak{p}= \mathbb{E}_n\oplus \mathbb{E}_n^\flat$. \\
We say that the sequence of linear Dirac structures $\mathbb{D}_n$ is \emph{an ascending sequence of linear Dirac structures} on $ \left( \mathbb{E}_n,\varepsilon_n \right) _{n \in \mathbb{N}} $ if the following properties are satisfied for all $n\in \mathbb{N}$:
\begin{description}
\item[({\bf ALD1})] 
$\mathbb{E}_n^\flat$ is closed in $\mathbb{E}_n^\prime$ and 
%$\varepsilon^* \left( \mathbb{E}^\flat_{n+1} \right) = \mathbb{E}_n^\flat$;
$\varepsilon_n^* \left( \mathbb{E}^\flat_{n+1} \right) = \mathbb{E}_n^\flat$;
\item[({\bf ALD2})] 
$\mathbb{D}_n$ is the pullback of $\mathbb{D}_{n+1}$.
\end{description}
\end{definition}

\begin{definition}
\label{D_ProjectiveLinearDiracSequence}
Consider a projective sequence   $ \left( \mathbb{E}_n,\lambda_n \right) _{n \in \mathbb{N}}$ of Banach spaces where  
%!$\lambda_n:\mathbb{E}_n\to \mathbb{E}_{n+1}$ 
$\lambda_{n+1} : \mathbb{E}_{n+1} \to \mathbb{E}_n$ 
 is surjective,  provided with a sequence of Pontryagin spaces $ \left( \mathbb{E}_n^\mathfrak{p}=\mathbb{E}_n\oplus\mathbb{E}^\flat_n \right) _{n \in \mathbb{N}} $ and, for each $n\in \mathbb{N}$, we have  a linear Dirac structure $\mathbb{D}_n\subset\mathbb{E}_n^\mathfrak{p}= \mathbb{E}_n\oplus \mathbb{E}_n^\flat$. \\
We say that the sequence of linear Dirac structures $(D_n)_{n \in \mathbb{N}}$ is a \emph{submersive projective sequence of linear Dirac structures} on $(\mathbb{E}_n,\varepsilon_n)$ if the following properties are satisfied for all $n\in \mathbb{N}$:
\begin{description}
\item[({\bf SPLD1})] 
$\mathbb{E}_n^\flat$ is closed in $\mathbb{E}_n^\prime$ and 
%! $\lambda_n^*(\mathbb{E}^\flat_{n+1})\subset  \mathbb{E}_n^\flat$;
$\lambda_{n+1}^*(\mathbb{E}^\flat_n)\subset  \mathbb{E}_{n+1}^\flat$
\item[({\bf SPLD2})]
$\mathbb{D}_{n+1}$ is the pushforward of $\mathbb{D}_{n}$.
\end{description}
\end{definition}
Thus we have:
\begin{theorem}\label{T_LimitSequenceLinearDirac}${}$
\begin{enumerate}
\item[1.]
Let $ \left( \mathbb{D}_n \right) _{n \in \mathbb{N}}$ be an ascending sequence of linear Dirac structures on an ascending sequence $ \left( \mathbb{E}_n,\varepsilon_n \right) _{n \in \mathbb{N}}$ of Banach spaces  provided with a sequence of Pontyagin spaces $ \left( \mathbb{E}_n^\mathfrak{p}=\mathbb{E}_n \oplus \mathbb{E}_n^\flat \right) _{n \in \mathbb{N}} $. Then there exist 
\begin{description}
\item[(i)] a convenient space  $\mathbb{E}=\underrightarrow{\lim}\mathbb{E}_n$; 
\item[(ii)] a Fr\'echet space $\mathbb{E}^\flat =\underleftarrow{\lim}\mathbb{E}^\flat_n$;  
\item[(iii)] a dual pairing $<\;,\;>$ on $\mathbb{E}^\flat\times\mathbb{E}$ defined  in the following way:

for $\alpha=\underleftarrow{\lim}\alpha_n$ and $u=\underrightarrow{\lim}u_n$, if  $i$ is the smallest integer $n$ such that $u$ belongs to $\mathbb{E}_n$, then
\begin{center}
 $<\alpha,u>=\underleftarrow{\lim}_{n\geq i}<\alpha_n,u_{i}>$
 \end{center}
\item[(iv)] a linear Dirac structure $\mathbb{D}\subset \mathbb{E}^\mathfrak{p}=\mathbb{E}\oplus\mathbb{E}^\flat$ defined by
\begin{equation}
\label{eqLinkDNDn+1}
\mathbb{D}
=\underrightarrow{\lim}
 \left( p_n \left( \mathbb{D}_n \right) \right) \times\underleftarrow{\lim}
 \left( p^\flat_n \left( \mathbb{D}_n \right) \right)
\end{equation}
where $p_n$ (resp. $p^\flat_n$) is the projection of $\mathbb{E}^\mathfrak{p}_n $ on $\mathbb{E}_n$ (resp. $\mathbb{E}^\flat_n$);
\item[(v)]
we set $\mathbb{L}_n=p_n(\mathbb{D}_n)$ (resp. $\mathbb{L}^\flat_n=p_n^\flat(\mathbb{D}_n)$). If $\mathbb{L}=\underrightarrow{\lim}\mathbb{L}_n=p(\mathbb{D})$ and $\mathbb{L}^\flat=\underleftarrow{\lim}(\mathbb{L}_n)=p^\flat(\mathbb{D})$, then $P_\mathbb{L}=\underleftarrow{\lim}P_{\mathbb{L}_n}$ is a symmetric bounded map from $\mathbb{L}$ to $\mathbb{L}^\prime$ whose kernel is $\mathbb{D}\cap\mathbb{E}$ and range is $\mathbb{L}^\flat/(\mathbb{D}\cap\mathbb{E}^\flat)$. In particular, the $2$-form  $\Omega_\mathbb{L}(u,v)=<P_\mathbb{L}(u), v)>$ on $\mathbb{L}$ is such that  
\[
\Omega_{\mathbb{L}}= \underleftarrow{\lim}(\Omega_{\mathbb{L}_n}).
\]
\end{description}

\item[2.] Let $ \left( \mathbb{D}_n \right) _{n\in \mathbb{N}}$ be a submersive projective sequence of linear Dirac structures on a submersive projective  sequence $ \left(  \mathbb{E}_n,\lambda_n \right) _{n\in \mathbb{N}}$ of Banach spaces provided with a sequence of Pontryagin spaces $ \left( \mathbb{E}^\mathfrak{p}_n=\mathbb{E}_n\oplus\mathbb{E}^\flat \right) _{n\in \mathbb{N}}$. Then  there exists

\begin{description}
\item[(i)] a Fr\'echet  space  $\mathbb{E}=\underleftarrow{\lim}\mathbb{E}_n$; 
\item[(ii)] a convenient space $\mathbb{E}^\flat =\underrightarrow{\lim}\mathbb{E}^\flat_n$;  
\item[(iii)] a dual pairing $<\;,\;>$ on $\mathbb{E}^\flat\times\mathbb{E}$ defined  in the following way:

for $\alpha=\underrightarrow{\lim}\alpha_n$ and $u=\underleftarrow{\lim}u_n$, if  $i$ is  the smallest integer $n$ such that $\alpha$ belongs to $\mathbb{E}^\flat_n$ then
 $<\alpha,u>=\underleftarrow{\lim}_{n\geq i}<\alpha_{i},u_n>$
\item[(iv)] a linear Dirac structure $\mathbb{D}\subset \mathbb{E}^\mathfrak{p}=\mathbb{E}\oplus\mathbb{E}^\flat$ defined by
\[
\mathbb{D}=\underleftarrow{\lim}
 \left( p_n \left( \mathbb{D}_n \right) \right)\times\underrightarrow{\lim} \left( p^\flat_n \left( \mathbb{D}_n \right) \right)
\]
where $p_n$ (resp. $p^\flat_n$) is the projection of $\mathbb{E}^\mathfrak{p}_n $ on $\mathbb{E}_n$ (resp. $\mathbb{E}^\flat_n$);
\item[(v)]
 We set $\mathbb{L}_n=p_n(\mathbb{D}_n)$ (resp. $\mathbb{L}^\flat_n=p_n^\flat(\mathbb{D}_n)$). If $\mathbb{L}=\underleftarrow{\lim}\mathbb{L}_n=p(\mathbb{D})$ and $\mathbb{L}^\flat=\underrightarrow{\lim}(\mathbb{L}_n)=p^\flat(\mathbb{D})$, then $P_\mathbb{L}=\underrightarrow{\lim}P_{\mathbb{L}_n}$ is a symmetric bounded map from $\mathbb{L}$ to $\mathbb{L}^\prime$ whose kernel is $\mathbb{D}\cap\mathbb{E}$ and range is $\mathbb{L}^\flat/(\mathbb{D}\cap\mathbb{E}^\flat)$. In particular, the $2$-form  $\Omega_\mathbb{L}$ defined by $\Omega_\mathbb{L}(u,v)=-<P_\mathbb{L}(v), u>$ on $\mathbb{L}$ and in fact  $\Omega_\mathbb{L}=\underrightarrow{\lim}\Omega_{\mathbb{L}_n}$.
\end{description}
\end{enumerate}
\end{theorem}

\begin{proof}
The proof of Point 2  uses the same type of arguments  as in the proof of Point 1 by permuting direct and projective limit in a dual way. Therefore, we will only consider the proof of Point 1.\\
Under the notations and assumptions of Point 1,  using Lemma-Definition~\ref{L_PulbackPushward} and Definition~\ref{D_AscendingLinearDiracSequence}, we have the following justifications: 
\begin{description} 
\item[(i)] 
Since $ \left( \mathbb{E}_n,\varepsilon_n \right) $ is an ascending sequence of Banach spaces, then the direct limit   $\mathbb{E}=\underrightarrow{\lim}\mathbb{E}_n$ is a convenient space (cf. \cite{CaPe19} or \cite{CaPe23}, Chap.~5,~4).
\item[(ii)] 
As  $ \left( \mathbb{E}_n,\varepsilon_n \right) $ is an ascending sequence of Banach spaces, then the sequence $ \left\langle( \mathbb{E}^\prime_{n+1},\varepsilon_n^* \right) $ is a projective sequence. But we have $\varepsilon^*_n \left( \mathbb{E}^\flat_{n+1} \right) = \mathbb{E}_n^\flat$, 
so $(\mathbb{E}^\prime_n,{\varepsilon^*_n}_{|\mathbb{E}^\flat_{n+1}})$ is also a projective sequence. Therefore $\mathbb{E}^\flat =\underleftarrow{\lim}\mathbb{E}^\flat_n$ is a Fr\'echet space (cf. \cite{DGV16} or  \cite{CaPe23}, Chap.~4, 4).
\item[(iii)]  
Fix some  $\alpha=\underrightarrow{\lim}\alpha_n\in \mathbb{E}^\flat$ and $u=\underleftarrow{\lim}u_n \in \mathbb{E}$. Since $\mathbb{E}=\displaystyle\bigcup_{n\in \mathbb{N}}\mathbb{E}_n$, let $i$ be the smallest integer $n$ such that $u_n$ belongs to $\mathbb{E}_n$. 
From properties of direct  limits (resp. projective limits), for each $k\in \mathbb{N}$, we have a canonical  injective (resp. surjective) bounded map 
$\widehat{\varepsilon}_k:\mathbb{E}_k\to \mathbb{E}^\flat$ (resp. $\widehat{\varepsilon}_k^*:\mathbb{E}^\flat\to \mathbb{E}_k^\flat$ where  $\widehat{\varepsilon}_k^*$ is the restriction to $\mathbb{E}^\flat$ of  the adjoint $\widehat{\varepsilon}_k^*: \mathbb{E}^\prime \to \mathbb{E}^\prime_k  $).
Then, we have  $\alpha_n=\widehat{\varepsilon}_n^*(\alpha)$ and so
\[
<\alpha, u>
=
\alpha_n(u_i)=\underleftarrow{\lim}_{n\geq i}<\alpha_n,u_i>
\] 
is well defined,  according to the properties of $\varepsilon_k^*$ in restriction to $\mathbb{E}^\flat_{k+1}$.  This  implies the announced result for the pairing $<\;,\;>$.
\item[(iv)] 
From   the proof of Lemma-Definition~\ref{L_PulbackPushward},~1. and Definition~\ref{D_AscendingLinearDiracSequence},  for each $n\in \mathbb{N}$, we have
\begin{center}
$p_n \left( \mathbb{D}_n \right)
= 
\varepsilon_{n+1}^{-1}
\left( p_{n+1}(\mathbb{D}_{n+1}\cap \varepsilon_n(\mathbb{E}_{n})) \right) 
\textrm{ and } 
p_n^\flat\left( \mathbb{D}_n \right)=\varepsilon_n^*
 \left( p_{n+1}^\flat(\mathbb{D}_{n+1}) \right) .$
\end{center}
It follows that $\varepsilon_n (p_n(\mathbb{D}_n))\subset p_{n+1}(\mathbb{D}_{n+1})$. Since, for each $n\in \mathbb{N}$, the spaces $p_n(\mathbb{D}_n)$ and $p_{n}^\flat(\mathbb{D}_{n})$ are closed in $\mathbb{E}_n$ and $\mathbb{E}_n^\flat$ respectively, it follows that $(p_n(\mathbb{D}_n))$ (resp.  $(p_{n}^\flat(\mathbb{D}_{n}))$) is an ascending sequence (resp. projective sequence) of Banach spaces. Therefore 
$\mathbb{D}=\underrightarrow{\lim}(p_n(\mathbb{D}_n))\times\underleftarrow{\lim}(p^\flat_n(\mathbb{D}_n))$ is a closed subspace of $\mathbb{E}^\mathfrak{p}$.
Since $\mathbb{D}_n=\mathbb{D}_n^\perp$ and  from the previous expression of $\mathbb{D}_n$, and the characterization of the pairing $<\;,\;>$, it follows easily that $\mathbb{D}=\mathbb{D}^\perp$.
\item[(v)] 
Let $u=\underrightarrow{\lim}u_n\in \mathbb{L}$ and $\alpha=\underleftarrow{\lim}\alpha_n\in \mathbb{E}^\flat$ such that $(u,\alpha)$ belongs to $\mathbb{D}$.  Then by definition of $P_\mathbb{L}$ we have $P_\mathbb{L}(u)=p^\flat(u,\alpha)$ and this value does not depend on the choice of $\alpha$. From the previous step, we can choose $\alpha$ such that  $(u_n,\alpha_n)\in \mathbb{D}_n$ and so $P_{\mathbb{L}_n}(u_n)=p_n^\flat(u_n,\alpha_n)$. Then we have
\begin{center}
$P_\mathbb{L}(u)=\underleftarrow{\lim} \left( p^\flat_n(u_n,\alpha_n) \right)
=
\underleftarrow{\lim}P_{\mathbb{L}_n}(u_n)$.
\end{center}
On the other hand,
\begin{center}
$\Omega_\mathbb{L}(v,u))= <P_\mathbb{L}(v), u>=-<\underleftarrow{\lim}P_{\mathbb{L}_n}(v_n), u_n>$
\end{center}
This implies that $\Omega_{L_n}=\underleftarrow{\lim}\varepsilon_n^*\Omega_L$,
 which ends the proof of the last assertion.
\end{description}
\end{proof}

\subsection{Ascending sequences of partial almost  Banach Dirac manifolds}

Let $\left( M_n,\varepsilon_n \right) _{n\in\mathbb{N}}$  be an ascending sequence of Banach  manifolds $M_n$ modelled on a Banach space $\mathbb{M}_n$ such that $\mathbb{M}_n$ is supplemented in $\mathbb{M}_{n+1}$.  For each $n\in\mathbb{N}$, we associate  a  Banach subbundle $T^\flat M_n$ of $T^\prime M_n$  whose typical fibre is denoted $\mathbb{M}^\flat_n$.  We also consider the Pontryagin  bundle $T^{\mathfrak{p}}M_n = TM_n \oplus T^\flat M_n$ and a Dirac structure $D_n$ on $M_n$. We denote by $p_n$ and $p^\flat_n$ the canonical projections of $T^{\mathfrak{b}}M_n$ onto $TM_n$ and $T^\flat M$ respectively. 
%Then, from  Proposition 7.11 (1)  \cite{CaPe23} we have

According  to section \ref{__LinearDiracSequences}, we introduce:
\begin{definition}
\label{D_DirectLimitDirac}
Assume that $\left( M_n,\varepsilon_n \right)_{n\in\mathbb{N}}$ is an ascending sequence  of Banach manifolds, each one  provided with a Pontryagin bundle $TM_n \oplus T^\flat M_n$ and a partial Dirac structure $D_n$. 
We will say that $ \left( D_n \right) _{n \in \mathbb{N}}$  is  \emph{an ascending sequence of partial almost  Dirac structures} if, for each $x=\underrightarrow{\lim}x_n$, the following properties are satisfied:
\begin{description}
\item[({\bf ADBM1})]  For all integer $i\leq j$,  the family of fibres $\left((D_j)_{x_i}\right)$ is an ascending sequence of partial linear Dirac structures on  $(T_{x_i}M_j, T_{x_j}\varepsilon_j)$.
\item[({\bf ADBM2})] Around each $x\in M$, there exists a sequence of charts $\left(
U_{i},\phi_{i}\right)  _{i\in\mathbb{N}}$ such that
$(U=\underrightarrow{\lim}(U_{i}),\phi=\underrightarrow{\lim}(\phi_{i}))$ is a
chart of $x$ in $M$, so the bundle projections $p_i $ (resp. $ p_i^\flat$) and $p_{i+1}$ (resp.  $p_{i+1}^\flat$ ) are compatible  with bounding maps  over the chart  $(U_{i},\phi_{i})$ 
and such that
\begin{description}
	\item[\textbf{(DSPPBM1)}]
	$T^{\ast}\varepsilon_{i}(T^{\flat}M_{i+1})\subset T^{\flat}M_{i}$
\end{description} 
(cf. Remark \ref{R_ADBM2}).
\end{description}
\end{definition}

\begin{remark}
\label{R_ADBM2}  
Note that since $\varepsilon_i$ is an injective convenient morphism, then  $T^*\varepsilon_i^*:{T^\prime M_{i+1}}_{| M_i}\to T^\prime M_i$ is a well defined surjective  convenient bundle. From   ({\bf ADL1}) and ({\bf ADBM1}), we have  $T^*\varepsilon_i^* \left( T^\flat M_{i+1} \right) _{| M_i}= T^\flat M_i$ and the assumption ({\bf ADBM2}) means that  the following relations are satisfied:

$p_{i+1}(\varepsilon_i(x_i);T\varepsilon_i(u),\alpha)=p_i(x_i;u,\varepsilon^*_i(\alpha) )$ 

$ p_{i+1}^\flat(\varepsilon_i(x_i); T\varepsilon_i(u),\alpha)=p_i^\flat(x_i;u, \varepsilon_i^*(\alpha)) $.

 for all $x_i\in U_i$, $u\in T_{x_i} M_i$ and  $\alpha\in T_{x_i}^\flat M_{i+1}$.
 
$(D_i)_{x_i}= \{(u_i,\varepsilon_i^*(\alpha_{i+1}))\in \mathbb{E}_i^\mathfrak{p}\;: (\varepsilon_i(u_i), \alpha_{i+1})\in (\mathbb{D}_{i+1})_{x_i}\}$,
 
for all $x_i\in U_i$.
\end{remark}

Then, from \cite{CaPe23}, we have
\begin{proposition}
\label{L_BanachDiracContext} 
Under the previous notations,
  $M=\underrightarrow{\lim}M_n$ is a locally Hausdorff convenient manifold and $T^\flat M=\underleftarrow{\lim}T^\flat M_n$ is a closed subbundle of the subbundle $T^\prime M=\underleftarrow{\lim}T^\prime M_n$ which is  the Fr\'echet cotangent bundle of $M$.
\end{proposition}

According to this Proposition, we have:
\begin{theorem}
\label{T_DirectLimitPartialDirac} 
Let $(D_n)_{n \in \mathbb{N}}$    be   an ascending sequence of partial almost  Dirac structures on an ascending  sequence $ \left( M_n,\varepsilon_n \right) _{n \in \mathbb{N}}$ of Banach manifolds, each one being provided with the Pontryagin bundle 
$T^\mathfrak{p}M=TM_n\oplus T^\flat M_n$. 
Then we have the followings:
\begin{enumerate}
\item[1.] 
Let  $T^\mathfrak{p}M=TM\oplus T^\flat M \equiv TM \times T^\flat M$  be the  Pontryagin bundle associated  to $M$. 
We  have convenient bundle  projections $p: T^\mathfrak{p}M\to TM$ and $p^\flat:T^\mathfrak{p}M\to T^\flat M$ which are given by $p=\underrightarrow{\lim}p_n$ and $p^\flat=\underleftarrow{\lim}p^\flat_n$.
\item[2.] 
$\mathbb{D}=\underrightarrow{\lim}p_n(D_n)\times \underleftarrow{\lim}p^\flat_n({D}_n)$ is an almost partial Dirac structure on $M$.
\item[3.] A dual pairing $<\;,\;>$ on $T^\flat M \times TM$ defined  in the following way

for $\alpha=\underleftarrow{\lim}\alpha_n\in T_x^\flat M$ and $u=\underrightarrow{\lim}u_n\in T_xM$,  let   $i$ be the smallest integer $j$ such that $\alpha$ belongs to $T_{\widehat{\lambda}_i(x)}^\flat M$ then
\begin{center}
 $<\alpha,u>=\underleftarrow{\lim}_{n\geq j}\alpha_n(u_{j})=\underleftarrow{\lim}_{n\geq i}<\alpha_i, u_n>$.
 \end{center}

\item[4.] The Courant bracket 
\begin{equation}\label{eq_DirctLimitCourantBracket}
[(X,\alpha), (Y,\beta)]_C
=\left([X,Y],L_X\beta-L_Y\alpha-\displaystyle\frac{1}{2}d(i_Y\alpha -i_X\beta)\right)
\end{equation}
is well defined for any local sections $(X,\alpha)$ and $(Y,\beta)$ of $T^\mathfrak{p}M$ over some open set $U$ in $M$.\\
 Moreover, if $U=\underrightarrow{\lim}U_n$, if $X=\underrightarrow{\lim}X_n$ (resp. $Y=\underrightarrow{\lim}Y_n$) and $\alpha=\underleftarrow{\lim}\alpha_n$ (resp. $\beta=,\underleftarrow{\lim}\beta_n) )$ where $(X_n,\alpha_n)$ (resp. $ (Y_n,\beta_n)$) is a section  of $T^\mathfrak{p}M_n$ over $U_n$, we have :
\begin{equation}
\label{eq_directLimitCorantSequnce}
[(X,\alpha),(Y,\beta)]_C=\left(\underrightarrow{\lim}[X_n,Y_n],\underleftarrow{\lim} p^\flat_n([(X_n,\alpha_n),(Y_n,\beta_n)]_C)\right)
\end{equation}
\end{enumerate}
\end{theorem}
\bigskip

\begin{corollary}
\label{C_FiniteDirectLimitDirac} 
Let $(D_n)_{n \in \mathbb{N}}$  be an ascending sequence of partial Dirac structures\footnote{That is each $D_n$ is stable under the Courant bracket on $M_n$.} on  an ascending  sequence $(M_n,\varepsilon_n) _{n \in \mathbb{N}}$ of Hilbert paracompact manifolds $M_n$, each one being provided with the Pontryagin bundle  $T^\mathfrak{p}M_n = TM_n\oplus T^\flat M_n$. Then the partial almost Dirac structure $D$ on $M$ is integrable. More precisely we have:
\begin{description} 
\item[$\bullet$] 
Given any $x\in M$, there exists a smallest integer $i$ such that $x_i$ belongs to $M_i$. Then the leaf $L$ of the characteristic distribution of $D$ through $x$ is the union of leaves $L_n$ through $x_n \in M_n$ of the characteristic distribution of $D_n$ for  all $n\geq i$
\item[$\bullet$] 
On each such a leaf $L$, we have a bundle convenient skew symmetric  morphism $P_L=\underleftarrow{\lim}_{n\geq i}(P_{L_n})$ from $TL\to T^\prime L$ whose kernel is $D\cap TL$ and range is $T^\prime L\cap T^\flat M$.  The associated $2$-form $\Omega_L$ is $\underleftarrow{\lim}_{n\geq j} (\Omega_{L_n})$ and it  is a strong pre-symplectic $2$-form on $L$.
\end{description}
In particular, this situation occurs for ascending sequence of Dirac structures on an ascending sequence of finite dimensional manifolds.\\
 \end{corollary}

\begin{remark}
\label{R_DnotDiracStructure}  
In the context of Corollary \ref{C_FiniteDirectLimitDirac}, although the relation (\ref{eq_directLimitCorantSequnce}) may suggest that $D$ is a Dirac structure, this relation does not imply, in general,  that the Courant bracket  in restriction to  $\Gamma(D_{|U})\times \Gamma(D_{|U})$  takes values in $\Gamma(D_{|U})$ for any open set $U$ in $M$. The first  reason is that  any vector field $X$ on an open set $U=\underrightarrow{\lim}U_n$ is not equal to $X=\underrightarrow{\lim}X_n$ where $X_n$ is the restriction to $U_n$ of a vector field of $M_n$. Precisely,  in the previous context of ascending sequences of Banach manifolds,  if $X$ is a smooth vector field defined on an open set $U=\underrightarrow{\lim}U_n$ on $M$, let $X_n$ be the restriction of $X$ to $U_n$. From \cite{Glo05}, Proposition~3.4, $X_n$ is only a local section of ${TM_j}_{| U_n}$ for some $j\geq n$. 
 Another  reason is that the module  $\Gamma(D_{| U})$ of sections of $D_{| U}$  is not  generated by sections of type $X=\underrightarrow{\lim}X_n$. The same type of argument is also true for $1$-forms on $M$. The following example gives  a vector field $X$ on the convenient vector space $\mathbb{R}^\infty=\underrightarrow{\lim} \mathbb{R}^n$ of finite sequences which is neither a direct limit $\underrightarrow{\lim}X_n$ nor a  functional linear of such direct limits:\\

Since the tangent space of $\mathbb{R}^\infty$ is $\mathbb{R}^\infty\times\mathbb{R}^\infty$, a vector field on $\mathbb{R}^\infty$ corresponds to a smooth function $X:\mathbb{R}^\infty\to \mathbb{R}^\infty$.  Any $x\in \mathbb{R}^\infty$ can be seen as a finite sequence $(x_1,\dots,x_n)\in \mathbb{R}^n$ for some $n\in \mathbb{N}$. We set $X(x)=(1,2,\dots,2^i)$  where $i$ is the smallest integer $n$ such that $x$ belongs to some $\mathbb{R}^n$. Then $X_{i}=X_{| \mathbb{R}^i}$ is smooth  and as $X=\underrightarrow{\lim}X_i$ where  $X_i$ is  a section of $\mathbb{R}^i\times \mathbb{R}^{2^i}$, and since $\mathbb{R}^\infty=\underrightarrow{\lim} \mathbb{R}^{2^i}$ so  $X$ is a smooth map from  $\mathbb{R}^\infty$ to  $\mathbb{R}^\infty$ but $X_i$ is not a vector field on $\mathbb{R}^i$. Also $X$ can not be written as a sum
$X_i=\displaystyle\sum_{j=1}^N f_j Z^j$, 
with $Z^j=\underrightarrow{\lim}Z^j_n$ where $Z_n$ is a smooth map from $\mathbb{R}^n\to \mathbb{R}^n$. Indeed if it were so,
we should have, for all $x\in \mathbb{R}^i$, $\displaystyle\sum_{j=1}^N f_j Z^j_i(x)\in \mathbb{R}^{2^i}$  while  $\displaystyle\sum_{j=1}^N f_j Z^j_i(x)$ belongs to $\mathbb{R}^i$.\\
\end{remark}

\begin{proof}[Proof of Theorem~\ref{T_DirectLimitPartialDirac}]${}$\\
1. As in the proof of \cite{CaPe23}, Proposition 7.11, we can show that for  $i\leq j$, if  we set $T_{M_i}^\flat M_j=\displaystyle\bigcup_{x_i\in M_i}T_{x_i}^\flat M_j$, then $T^\flat _{M_i}M=\underleftarrow{\lim}T_{M_i}^\flat M_j$ is a Fr\'echet vector bundle over $M_i$ and $T^\flat M=\underrightarrow{\lim}T_{M_i}^\flat M$ is a convenient  bundle over $M$ and, of course, we have $T^\flat M=\displaystyle\bigcup_{n\in \mathbb{N}, x_i\in M_i}(\underleftarrow{\lim}T_{x_i}^\flat M_j)$. If $T_{M_i}M_j$ is the restriction of  $TM_j$ to $M_i$, it follows that 

$T_{M_i}^\mathfrak{p}M:=T^\mathfrak{p}M_{| M_i}\equiv TM_{| M_i}\times  T^\flat M_{| M_i}\equiv \underrightarrow{\lim}T_{M_i} M_j\times\underleftarrow{\lim}T_{M_i}^\flat M_j$.

Now, in the proof of \cite{CaPe23}, Proposition~7.11, we have shown that
\[ 
TM=\underrightarrow{\lim}T_{M_i}M 
\text{ ~and~ }
 T^\flat M=\underrightarrow{\lim}T_{M_i}^\flat M
\]
which ends partially  the proof of Point 1 (for the completeness of  the proof,  see the end of the proof of Point 2).\\

\noindent 2. For each $i$, over $M_i$, we have $D_i\equiv p_i(D_i)\times p_i^\flat (D_i)$.  From assumption ({\bf ADBM2})  (and so Remark \ref{R_ADBM2}), we have: 
%and Lemma-Definition  Point 1

$p_i(D_i)=(T\varepsilon_i)^{-1}\left(p_{i+1}({D_{i+1}}_{| M_i})\cap T\varepsilon_i( TM_i)\right)$\\ 
and so  

$T\varepsilon_i(p^\flat_i(D_i))\subset p_{i+1}({D_{i+1}}_{| M_i})$, 
i.e $p^\flat_i(D_i)=T^*\varepsilon_i\left(p_{i+1}^\flat({D_{i+1}}_{| M_i})\right)$.

\noindent 
If  we set 
\[
\mathcal{L}_{ij}={p_{j}}_{|M_i}\circ\cdots\circ {p_{i+1}}_{| M_i}\circ p_i (D_i)
\text{ ~and~ }
\mathcal{L}_{ij}^\flat={p_{j}^\flat}_{| M_i}\circ\cdots\circ {p_{i+1}^\flat}_{| M_i}\circ p_i ^\flat(D_i)
\]
the previous relations implies (by induction on $j$) that $ \left( \mathcal{L}_{ij} \right) _{j\geq i})$ (resp. $  \left( \mathcal{L}_{ij}^\flat \right)_{j\geq i})$ is an ascending (resp. projective) sequence of Banach bundles over $M_i$. Therefore if we set $\mathcal{L}_{M_i}=\underrightarrow{\lim}\mathcal{L}_{ij}$ and $\mathcal{L}_{M_i}^\flat=\underleftarrow{\lim}\mathcal{L}_{ij}^\flat$, it follows that 
$D_{M_i}=\mathcal{L}_{M_i}\times\mathcal{L}_{M_i}^\flat$
is a closed subbundle of $T_{M_i}^\mathfrak{p}M$. Moreover, by construction, for each $x_i\in M_i$ we have 
\begin{equation}
\label{eq_partialDiracFibre}
D_{x_i}=\underrightarrow{\lim}_{j\geq i}{p_j}_{x_i}(D_i)\times \underleftarrow{\lim}_{j\geq i}{p_j^\flat}_{x_i}(D_i)
\end{equation}
Thus, from  Theorem \ref{T_LimitSequenceLinearDirac}, it follows that $D_{M_i}$ is a partial almost Dirac structure on the Fr\'echet bundle $T_{M_i}M$. Now, as in the proof of Proposition 7.11 in  \cite{CaPe23}, in order to prove that $TM=\underrightarrow{\lim}T_{M_i}M$ and $T^\prime M=\underrightarrow{\lim}T_{M_i}^\prime M$, we can show that 
$\mathcal{L}=\underrightarrow{\lim}\mathcal{L}_{M_i}$ and $\mathcal{L}^\flat=\underleftarrow{\lim}\mathcal{L}^\flat_{M_i}$ are closed subbundles of $TM$ and $T^\flat M$ respectively.  From this construction, if $x$ belongs 
to $M$, then $x$ belongs to some  $M_i$ and so the relation (\ref{eq_partialDiracFibre}) is true for any $x\in M$, and again, from Theorem \ref{T_LimitSequenceLinearDirac}, it follows 
that $D\equiv \mathcal{L}\times \mathcal{L}^\flat$ is a partial almost Dirac structure on  $TM$ that is on $M$ by convention.\\

The previous construction also  shows that the canonical projections  $p$ and $p^\flat$ from $T^\mathfrak{p}M$ to $TM$ and $T^\flat M$ respectively are also characterized by:

for $x$ in $M$, if $i$ is the smallest integer $n$ such that $x$ belongs to $M_n$  and then $p_x=\underrightarrow{\lim}_{j\geq i}{p_j}_{x}$ and $p_x^\flat=\underleftarrow{\lim}_{j\geq i}{p_j^\flat}_{x}$, which completes the proof of Point 1.\\

\noindent 3. This Point is a direct consequence of  Theorem \ref{T_LimitSequenceLinearDirac} 1.(iii).\\

\noindent 4. Since $TM$ is a convenient Lie algebroid, from Definition \ref{D_CourantBracket} and Proposition \ref{P_PropertyCourantBracket}, the Courant bracket (\ref{eq_DirctLimitCourantBracket}) is well defined.  

Consider an open set  $U=\underrightarrow{\lim}U_n$, and let  $X=\underrightarrow{\lim}X_n$ (resp.  $Y=\underrightarrow{\lim}Y_n$) where $X_n$ (resp. $Y_n$) is a vector field on $U_n$. For each $n$, we have a smooth inclusion $\widehat{\varepsilon}_n$ of the Banach manifold $M_n$ into $M$. Therefore $T\widehat{\varepsilon}_n(X_n)=X_{| U_n}$ and  $T\widehat{\varepsilon}_n(Y_n)=Y_{| U_n}$ and so  $T\widehat{\varepsilon}_n([X_n, Y_n])=[X,Y]_{| U_n}$ which implies that

$[X,Y]=\underrightarrow{\lim}[X_n,Y_n]$.

If $\alpha=\underleftarrow{\lim}\alpha_n$, we have  $\alpha_n=\widehat{\varepsilon}_n^*(\alpha)$.\\
Thus, on the one hand, we obtain
 
  $\widehat{\varepsilon}_n^*d(i_X\alpha)= i_{T\widehat{\varepsilon}_n(X_n)}\widehat{\varepsilon}_n^*(\alpha)=\widehat{\varepsilon}_n^*(\alpha)(X_n)=\alpha_n(X_n)$\\
and so
  
$\widehat{\varepsilon}_n^*d(i_X\alpha)= d\widehat{\varepsilon}_n^*(i_X\alpha)=d\alpha_n(X_n)=d_{i_{X_n}}\alpha_n.$

\noindent  By analog arguments, on the other hand, we obtain $\widehat{\varepsilon}_n^*\left(i_X d\alpha\right)=i_{X_n} d\alpha_n$.\\
Since $L_X=i_X \circ d+d \circ i_X$, it follows  that 
$\widehat{\varepsilon}_n^*(L_X\alpha)=L_{X_n}\alpha_n$.\\
We finally obtain:

$ di_X\alpha=\underleftarrow{\lim}d i_{X_n}\alpha_n$ and $L_X\alpha=\underleftarrow{\lim}L_{X_n}\alpha_n$\\
which clearly implies the last part of Point 4.
\end{proof}

\begin{proof}[Proof of Corollary \ref{C_FiniteDirectLimitDirac}] 
Fix some $x=\underrightarrow{\lim}x_n\in M$. There exists a smallest integer $i$ such that $x$ belongs to $M_i$. Since the characteristic distribution of $p_i(D_i)$ is integrable, denote by  $L_i$ the leaf through $x$. For all $n\geq i$, $x$ belongs to $M_n$ and by the same argument, there exists a leaf $L_n$ of the characteristic distribution of $D_n$ through $x$. Now, 
from Theorem \ref{T_DirectLimitPartialDirac}, for any $y\in L$,  we have the following  ascending sequence of  closed  spaces:
$$T_yL_i\subset T_yL_{i+1}\subset \cdots\subset T_yL_n\subset\cdots\subset  p(D_y)\subset T_y M.$$

  The restriction $\varepsilon_{L_i}$  to $L_i$ is a continuous map into $L_{i+1}$ and since it  is a smooth map from $L_i\to M_{i+1}$,   it follows that  $\varepsilon_{L_i} :L_i\to L_{i+1}$ is smooth\footnote{Using the argument of convenient smoothness, from the continuity of $\varepsilon_{L_i} :L_i\to L_{i+1}$  for any smooth curve $\gamma:\mathbb{R}\to L_i$, the map $\varepsilon_{L_i}\circ \gamma$ must be contained in $L_{i+1}$. As $L_{i+1}$ is a closed submanifold, it follows that  $\varepsilon_{L_i}\circ \gamma$ must be smooth in $L_{i+1}$.}. But $T_x\varepsilon_i$ is an injective map for all $x\in L_i$ the same is true for  $T_x\varepsilon_{L_i}$. 
  By same arguments we can show that  the restriction $\varepsilon_{L_n}$  to $L_n$ of the inclusion $\varepsilon_n$ is a smooth immersion from $L_n$ to $L_{n+1}$.  
  Finally,  the sequence $(L_n)_{n\geq i}$ is a an ascending sequence of Hilbert  manifolds. As each $M_n$ is paracompact,  there always exists a Koszul connection on $TM_n$. 
  Moreover,  $T_xM_n$  is supplemented in $T_xM_{n+1}$ for all $x\in M_n$,  so the assumption \cite{CaPe19}, Theorem 31 is satisfied and so $L=\displaystyle\bigcup_{n\geq i} L_n=\underrightarrow{\lim}L_n$ has a structure of convenient  Hausdorff manifold. Since,  for $y\in L_i$, $p(D_y)=T_yL=\underrightarrow{\lim}T_yL_n$, $L$ in then an integral manifold of $p(D)$ through $x$. It follows that   the characteristic distribution of $D$ is integrable.  \\
  
From Theorem \ref{T_GeometricStructuresOnLeaves},  for each leaf $L_n$ of the characteristic distribution of $D_n$, we have a skew symmetric convenient bundle $P_{L_n}:TL_n\to T^\prime L_n$  whose kernel is $(D_n\cap TM_n)_{| L_n}$ and whose range is the  weak subbundle $T^\flat L_n\cap T^\prime L_n$. In particular, $\Omega_{L_n}(u_n,v_n)=<P_{L_n}(v_n), u_n>$ is a strong pre-symplectic  $2$-form on $L_n$.
Now, from Theorem \ref{T_LimitSequenceLinearDirac}, any $x\in M$ belongs to some $M_n$ and if $i$ is the smallest of such integers, then   $x=\underrightarrow{\lim}_{n\geq i}x_n$ and so $x_n=x$ for all $n\geq i$. From Theorem \ref{T_LimitSequenceLinearDirac},
  we have $(P_L)_x=\underleftarrow{\lim}_{n\geq i}(P_{L_n})_{x_n}$. Therefore, if $u$ belongs to $T_xM$, let $j$ be the smallest integer $k$ such that $u$ belongs to $T_xM_k$ then $u_n=u$ for all $n>j$. Thus we have $(P_L)_x(u)=(P_{L_n})_{x_n}(u)=(P_{L_n})_{x_n}(u_n)$ for all $n\geq j$. Moreover, for  $v\in T_xM$  if $j$ is now the smallest integer $n$ such that $u$ and $v$ belong to $T_xM_n$ such that $v$ belongs to $T_xM_n$,  from Theorem \ref{T_LimitSequenceLinearDirac} we have:
\begin{equation}\label{eq_PL}
<(P_L)_x(u), v>=\underleftarrow{\lim}_{n\geq j}<(P_{L_n})_{x_n}(u_n),v_n>=-\underleftarrow{\lim}_{n\geq j}<(P_{L_n})_{x_n}(v_n), u_n>
\end{equation}
and so  $(P_L)_x$ is a skew symmetric map from $T_xL$  to $T_x^\prime L$.\\
Now, by Theorem \ref{T_DirectLimitPartialDirac} (1), we have $TL=\underrightarrow{\lim}TL_n$ and $T^\flat L=\underleftarrow{\lim} T^\flat L_n$ and so $T^\prime L=\underleftarrow{\lim}T^\prime L_n$. It follows that we obtain a  convenient morphism $P_L=\underleftarrow{\lim}(P_{L_n})$ from $T^\flat L$ to $T^\prime L$ whose range is  $T^\flat L\cap T^\prime L=\underleftarrow{\lim} (T^\flat L_n\cap T^\prime L_n)$.\\

As previously,  $x=\underrightarrow{\lim}x_n\in M$, then $ x$ belongs to some $ M_i$ for  a minimal integer $i$ and $x_n=x_i=x$ for all $n\geq i$. Now we have 
$$D_x\cap T_xM=\{(u,0)\in D_{x_i}, u\in T_{x_i} M\}.$$
 According to the previous notations, $u$ belongs to  $T_xM_j$ for a minimal integer $j$. It follows that $D_x\cap T_xM=\displaystyle\bigcup_{j\geq i}\left((D_j)_{x_i}\cap T_{x_i}M_j\right)$. Thus 
 $ \left( (D_j)_{x_i}\cap T_{x_i}M_j \right) _{j\geq i}$ is an ascending sequence and, as for the construction of $TM$, in one hand, if $D_{M_i}$ denotes the restriction of $D$ to $M_i$, we have $D_{M_i}\cap T_{M_i}M=\underrightarrow{\lim}_{j\geq i} (D_{M_i}\cap T_{M_i} M_j)$ and so $D\cap TM=\underrightarrow{\lim}(D_{M_i}\cap T_{M_i} M)$ which is a closed subbundle of $TL$ since each $(D_j)_{x_i}\cap T_{x_i}M_j$ is contained in $T_{x_i}L_j$ for all $x_i\in L_i$.  

\begin{lemma}
\label{L_SupplementedFinE}
Let  $(\mathbb{E}_n)_{\in \mathbb{N}}$ and $(\mathbb{F}_n)_{\in \mathbb{N}}$ be ascending sequences of Hilbert spaces such that $\mathbb{F}_n$ is a closed subspace of $\mathbb{E}_n$  and $\mathbb{E}_n$ is closed in $\mathbb{E}_{n+1}$, then\\
$\mathbb{F}=\underrightarrow{\lim}\mathbb{F}_n$ is a closed Hilbert subspace in the Hilbert space of  $\mathbb{E}=\underrightarrow{\lim}\mathbb{E}_n$ and so is supplemented.
\end{lemma}
 Note that, since  $L_i$ is a Hilbert manifold and $x_n=x_i=x$ for $n\geq i$, we can apply Lemma~\ref{L_SupplementedFinE}  for  $\mathbb{F}_n=D_{x_n}\cap T_{x_n} M$  and $\mathbb{E}_n=T_{x_n}L_n$ and so  each fibre $D_x\cap T_xM$ must be supplemented in $T_xL$ for each $x\in L$.\\ 
From the above construction of $P_L$, it is easy to check that $D\cap TM$ is the kernel of $P_L$. 
From relation (\ref{eq_PL}), we have
\begin{equation}
\label{eq_OL}
\begin{matrix}
(\Omega_L)_x(u,v)=<P_L(v),u>
=\underleftarrow{\lim}_{n\geq j}<(P_{L_n})_{x_n}(v_n),u_n>\hfill \\
=\underleftarrow{\lim}_{n\geq j} \Omega_{L_n}(T_{x_n}\varepsilon_n(u),T_{x_n}\varepsilon_n(v))
=\underleftarrow{\lim}_{n\geq j} (\Omega_{L_n})_{x_n}.\hfill{}
\end{matrix}
\end{equation}
 
Fix $x=\underrightarrow{\lim}x_n\in M$ and $u$, $v$ and $w$ in $T_xM$. There exists a minimal integer $i$ such that  $x\in M_i$ and $j$ such that  $u, v,w\in T_{x}M_j$. As we have already seen, we have $x_n=x_i=x$ if $n\geq i$ and
if we write $u=\underrightarrow{\lim}u_n$, then $u_n=u_j=u$ for all $n\geq j$.  Of course, we have analog results for $v$ and $w$.\\
From the previous relation (\ref{eq_OL}), we have $\Omega_L= \left( (\widehat{\varepsilon}_n)_{|L_n} \right) ^*\Omega_{L_n}$.  Since  the exterior differential commute with $((\widehat{\varepsilon}_n)_{|L_n})^*$  
we have 
\[
d\Omega_L(u,v,w)
=((\widehat{\varepsilon}_n)_{|L_n})^*d\Omega_{L_n}(u,v,w)=\Omega_{L_n}(u_n,v_n,w_n)=0
\]
since $\Omega_{L_n}$ is closed. It follows that $\Omega_L$ is closed.  By construction, the kernel of $\Omega_L$ is  $D\cap TM$ which is supplemented in each fibre of $TL$ and we have $\Omega_L^\flat(TL)=T^\flat L\cap T^\prime L$ which is a  weak (closed) subbundle of $T^\prime L$, which completes  the proof.
\end{proof}

\subsection[Projective  sequences  of partial almost  Banach Dirac manifolds]{Submersive projective  sequences  of partial almost  Banach Dirac manifolds}

Let $\left( M_n \lambda_n \right)_{n\in\mathbb{N}}$  be a submersive  sequence of Banach  manifolds $M_n$ modelled on the Banach spaces $\mathbb{M}_n$ such that $\lambda_n: {M}_{n+1}\to M_n$ is a surjective submersion. 
For any  $n\in\mathbb{N}$, we associate  a  Banach subbundle $T^\flat M_n$ of $T^\prime M_n$  whose typical fibre is denoted $\mathbb{M}^\flat_n$ and we consider the Pontryagin  bundle $T^{\mathfrak{p}}M_n=TM_n \oplus T^\flat M_n$ and a Dirac structure $D_n$ on $M_n$. We denote by $p_n$ and $p^\flat_n$ the canonical projections of $T^{\mathfrak{p}}M_n$ onto $TM_n$ and $T^\flat M$ respectively. 

\begin{proposition}\label{L_ProjectiveBanachDiracContext} 
Under the previous notations, $M=\underleftarrow{\lim}M_n$ is a locally Hausdorff convenient manifold and we have $TM=\underleftarrow{\lim}TM_n$ which is a Fr\'echet bundle. Each fibre $T_x^\prime M$ of the  cotangent bundle $T^\prime M$ is the topological dual of $T_xM$.  In each $T_x^\prime M$, there exists a direct limit  of type $T_x^\flat M=\underrightarrow{\lim}_{n\geq i}T_{x_i}^\flat M_n$ which gives rise to a convenient  closed subbundle  $T^\flat M$ of $T^\prime M$.
\end{proposition}
\begin{proof}[Sketch of the proof]
From \cite{CaPe23}, Chap. 4, we have  $TM=\underleftarrow{\lim}TM_n$ which is a Fr\'echet bundle.
If $\widehat{\lambda}_i:=\underleftarrow{\lim}_{n\geq i}\lambda_n: M\to M_i$ is the canonical submersion, for each $x\in M$, the sequence $ \left( T_{\widehat{\lambda}_i(x)}^\prime M_i \right) $ is  an ascending sequence of closed  Banach spaces whose direct limit is a convenient space $T_x^\prime M$,which is the topological dual of $T_xM$ and so $T^\prime M=\displaystyle\bigcup_{x\in M} T_x^\prime M$ is the cotangent bundle of $M$ (cf. \cite{CaPe23}, Chap.~7, 5). Moreover, $ \left( T_{\widehat{\lambda}_i(x)}^\flat M_i \right) _{i \in \mathbb{N}} $  is an ascending subsequence of $ \left( T_{\widehat{\lambda}_i(x)}^\prime M_i \right) _{i \in \mathbb{N}} $ and so $T^\flat M=\displaystyle\bigcup_{x\in M} T_x^\flat M$ is a closed convenient subbundle of $T^\prime M$ (cf. proof of Proposition~7.10 in \cite{CaPe23}).
\end{proof}

\smallskip
According to  section \ref{__LinearDiracSequences}, we introduce:

\begin{definition}
\label{D_ProjectiveLimitDirac}
Assume that $ \left( M_n,\lambda_n \right) _{n\in\mathbb{N}}$ is a submersive projective sequence  of Banach manifolds, where each $M_n$ is  provided with a  Pontryagine  bundle $TM_n \oplus T^\flat M_n$ and a partial almost Dirac structure on $D_n$. We will say that $D_n$   is   a submersive sequence of partial almost  Dirac structures,  if for each $x=\underleftarrow{\lim}x_n$, the following properties are satisfied:
\begin{description}
\item[{\bf (SPBM1)}]  
For all integer $i\leq j$,  the family of fibres $\left((D_j)_{x_i}\right)$ is a submersive sequence of partial linear Dirac structures on  $(T_{x_i}M_j, T_{x_j}\lambda_j)$.
\item[{\bf (SPBM2)}]
Around each $x\in M$, there exists a sequence of charts $\left(U_{i},\phi_{i}\right)  _{i\in\mathbb{N}}$ such that
$(U=\underleftarrow{\lim}(U_{i}),\phi=\underleftarrow{\lim}(\phi_{i}))$ is a
chart of $x$ in $M$, so the bundle projections $p_i $ (resp. $ p_i^\flat$) and $p_{i+1}$ (resp.  $p_{i+1}^\flat$ ) are compatible  with bounding maps over chart  $(U_{i},\phi_{i})$ 
and  compatible with Property {\emph{({\bf SPLD1})}}. 
%(cf. Remark \ref{R_ SPBM2}).
\end{description}
\end{definition}

\begin{remark}
\label{R_SPBM2} 
Note that since $\lambda_i$ is a  surjective  convenient morphism, then  $(T\lambda_i)^*:{T^\prime M_i}\to T^\prime M_i$ is a well defined surjective  convenient bundle; from   {\bf (SPDL1)} and {\bf (SPDBM1)}, we have $T^*\lambda_i(T^\flat M_{i})\subset T^\flat M_i$, and the assumption {\bf (PDBM2)} means that  the following relations are satisfied:

$p_{i+1} \left( x_{i+1};u \lambda_i^*\alpha \right)
= 
p_i \left( x_i;T_{x_{i+1}}\lambda_i(u),\alpha \right) $, 

$ p_{i+1}^\flat \left( x_{i+1}; u, \lambda_i^*\alpha \right) =
p_i^\flat \left( x_i;T_{x_{i+1}}\lambda_i(u),\alpha \right) $,

 for all $x_{i+1}\in U_{i+1}, \; u\in T_{x_{i+1}} M_{i+1},\; \alpha\in T^*\lambda_i( T_{x_i}^\flat M_{i})$,\\
 and $(D_{i+1})_{| U_{i+1}}$
 is the pushforward\footnote{Note that since $\lambda_i$ is surjective and  $(U_i)_{i\in \mathbb{N}}$ is a projective sequence of open sets,  $\lambda_i(U_{i+1})$ is on open set of $U_{i}$.} of $(D_{i})_{| \lambda_i(U_{i+1})}$.
\end{remark}

Taking into account Proposition \ref{L_ProjectiveBanachDiracContext}, we have:
\begin{theorem}
\label{T_ProjectiveLimitPartialDirac} 
Let $ \left( D_n \right) _{n \in \mathbb{N}}$  be a submersive projective sequence of partial almost  Dirac structures on a submersive projective sequence of Banach manifolds $ \left( M_n,\lambda_n \right) $ where each manifold $M_n$ is provided with the Pontryagin bundle  $T^\mathfrak{p}M_n=TM_n\oplus T^\flat M_n$. Then we have the followings:
\begin{enumerate}
\item[1.] 
Let  $T^\mathfrak{p}M=TM\oplus T^\flat M \equiv TM \times T^\flat M$  be the  Pontryagin bundle associated  to $M$. We  have convenient bundle  projections $p: T^\mathfrak{p}M\to TM$ and\\ $p^\flat:T^\mathfrak{p}M\to T^\flat M$ which are given by $p=\underleftarrow{\lim}p_n$ and $p^\flat=\underrightarrow{\lim}p^\flat_n$.
\item[2.] 
${D}=\underleftarrow{\lim}(p_n({D}_n))\times\underrightarrow{\lim}(p^\flat_n({D}_n))$ is a partial almost Dirac structure on $M$.
\item[3.] 
A dual pairing $<\;,\;>$ on $T^\flat M \times TM$ defined  in the following way: 

for $\alpha=\underrightarrow{\lim}\alpha_n\in T_x^\flat M$ and $u=\underleftarrow{\lim}u_n\in T_xM$,  let   $i$ be the smallest integer $j$ such that $\alpha$ belongs to $T_{\widehat{\lambda}_i(x)}^\flat M$\footnote{Cf. Proof of Proposition \ref{L_ProjectiveBanachDiracContext}.}  
then
 $$<\alpha,u>=\underleftarrow{\lim}_{n\geq j}\alpha_n(u_{j})=\underleftarrow{\lim}_{n\geq i}<\alpha_i, u_n>$$

\item[4.] 
The Courant bracket 
\begin{equation}\label{eq_ProjectiveLimitCourantBracket}
[(X,\alpha), (Y,\beta)]_C
=\left([X,Y],L_X\beta-L_Y\alpha-\displaystyle\frac{1}{2}d(i_Y\alpha -i_X\beta)\right)
\end{equation}
is well defined for any section $(X,\alpha)$ and $(Y,\beta)$ of $T^\mathfrak{p}M$ defined on some open set $U$ in $M$.\\
 Moreover, on an  open set $U=\underleftarrow{\lim}U_n$,   consider sections  $X=\underleftarrow{\lim}X_n$ (resp. $Y=\underleftarrow{\lim}Y_n$) and $\alpha=\underrightarrow{\lim}\alpha_n$ (resp. $\beta=,\underrightarrow{\lim}\beta_n) )$ where $(X_n,\alpha_n)$ (resp. $ (Y_n,\beta_n)$) is a section  of $T^\mathfrak{p}M_n$ over $U_n$. Then we have :
\begin{equation}
\label{eq_directLimitCourant}
[X,\alpha),(Y,\beta)]_C=\left(\underleftarrow{\lim}[X_n,Y_n],\underrightarrow{\lim} p^\flat_n([(X_n,\alpha_n),(Y_n,\beta_n)]_C)\right).\\
\end{equation}
\end{enumerate}
\end{theorem}

\begin{remark}
\label{R_DProjectivenotDirac}  
For the same arguments as in Remark \ref{R_DnotDiracStructure}, if $(D_n)$ is a sequence of submersive surjections of partial Dirac structures in Theorem \ref{T_ProjectiveLimitPartialDirac}, on $M=\underleftarrow{\lim}M_n$, we only have a partial {\bf almost}  Dirac structure on  $D$, but  in general, $D$  is not a  partial Dirac structure.
\end{remark}

\begin{remark}
\label{R_NoProjectiveLimitChart}
As for ascending sequences of partial Dirac structures, we can hope that, for a submersive projective sequence of partial  Dirac structures $(D_n )$, under some additional assumptions\footnote{  As in the Hilbert context for instance.}, we can show that we have projective limits of  sequences of  characteristic leaves $(L_n)$  which give rise to a leaf of the characteristic distribution of $D$. For such a result, we would need some sufficient conditions under which such a sequence has the projective limit chart property. Unfortunately, we have such a result, contrary to the ascending case of sequences of Banach manifolds. 
\end{remark}

Since the proof of Theorem \ref{T_ProjectiveLimitPartialDirac} uses analog arguments as the ones used in the proof of Theorem \ref{T_DirectLimitPartialDirac}, we only give a sketch of proof.

\begin{proof}
\noindent 1. and 2. 
The first  part of point 1 is a consequence of Proposition \ref{L_ProjectiveBanachDiracContext} and  the proof of the last part of Point 1 and of Point 2 using analog arguments as for the last part of point 1  and Point 2 of Theorem \ref{T_DirectLimitPartialDirac}.\\

\noindent 3. 
This point is a direct consequence of  Theorem~\ref{T_LimitSequenceLinearDirac}, 2.(iii).\\

\noindent 4.  
For the existence of a Courant bracket, the same arguments as for Theorem ~\ref{T_DirectLimitPartialDirac},~4 are used. \\

We now look for the proof of the last part.\\
Consider an open set  $U=\underleftarrow{\lim}U_n$ and  the projective limit $X=\underleftarrow{\lim}X_n$ (resp. the projective limit $Y=\underleftarrow{\lim}Y_n$) where $X_n$ (resp. $Y_n$) is a vector field on $U_n$. For each $n$, we have a smooth submersion $\widehat{\lambda}_n$ of the convenient  manifold $M$ onto the Banach space $M_n$. Therefore $T\widehat{\lambda}_n(X)=X_n$ and  $T\widehat{\lambda}_n(Y)=Y_{n}$ and so  $T\widehat{\lambda}_n([X, Y])=[X_n,Y_n]$, which implies that $[X,Y]=\underleftarrow{\lim}[X_n,Y_n]$.

If $\alpha=\underrightarrow{\lim}\alpha_n$  where $\alpha_n$ is a $1$-form on $U_n$, then  $\alpha=\widehat{\lambda}_n^*(\alpha_n)$. Thus, on the one hand, we obtain:  
 
  $\widehat{\lambda}_n^*(i_{X_n}\alpha_n)= i_{T\widehat{\lambda}_n(X_n)}\widehat{\lambda}_n^*(\alpha_n)=\widehat{\lambda}_n^*(\alpha_n)(X)=\alpha(X)$ 
  
and so we get
  
   $\widehat{\lambda}_n^*d(i_{X_n}\alpha_n)= d\widehat{\lambda}_n^*(i_{X_n}\alpha_n)=d\alpha(X)=d_{i_{X}}\alpha.$

\noindent  By analog arguments, on the other hand, we obtain $\widehat{\lambda}_n^*\left(i_{X_n} d\alpha_n\right)=i_{X} d\alpha$.  Since $L_X=i_X \circ d+d \circ i_X$, it follows  that 
$\widehat{\lambda}_n^*(L_{X_n}\alpha_n)=L_{X}\alpha$. We finally obtain:

$ di_{X}\alpha=\underrightarrow{\lim}d i_{X_n}\alpha_n$ and $L_{X}\alpha=\underrightarrow{\lim}L_{X_n}\alpha_n$
\noindent which clearly implies the last part of Point 4.
\end{proof}

\begin{example}
\label{Ex_BanachProduct} 
We consider a submersive projective sequence  $(M_k)_{k\in \mathbb{N}^\ast}$ of Banach manifolds, where each one is provided with a partial Dirac structure $D_k$, closed Banach subbundle of $T^\mathfrak{p}
M_k=TM_k\displaystyle\oplus T^\flat M_k$ and  $T^\flat M_k$ is a closed subbundle of $T^\prime M_k$. To this sequence, we can  associate the sequence of Banach manifolds 
$ \left( \widetilde{M}_n=\displaystyle\prod_{k=1}^n M_k \right) _{n \in \mathbb{N}^\ast} $.
 If $\lambda_n:\widetilde{M}_{n+1}\to \widetilde{M}_n$ is the canonical projection, then $(\widetilde{M}_n,\lambda_n)$ is a submersive 
projective sequence of Banach manifolds. Then $\widetilde{M}=\displaystyle\prod_{k=1}^\infty M_k= \underleftarrow{\lim}\widetilde{M}_n$ is a Fr\'echet manifold and we have $T\widetilde{M}
=\underleftarrow{\lim}T\widetilde{M}_n=\displaystyle\prod_{k=1}^\infty TM_k \equiv \displaystyle\bigoplus _{k=1}^\infty TM_k$ and  $T^\flat \widetilde{M}=\underrightarrow{\lim}T^\flat \widetilde{M}_n$.
We have a canonical pairing $<.,.>_n$ between $T^\flat \widetilde{M}_n=\displaystyle\prod_{k=1}^n T^\flat M_k\equiv \displaystyle\bigoplus_{k=1}^n T^\flat M_k$ and $T\widetilde{M}_n=\displaystyle\prod_{k=1}^n TM_k\equiv \displaystyle\bigoplus_{k=1}^n TM_k$ given by
\[
<(u_1,\dots,u_n),(\alpha_1,\dots,\alpha_n)>_n=\displaystyle\sum_{k=1}^n\alpha_k(u_k).
\]
for any $\alpha_k\in T_{x_k}^\flat M_k$ and $u_k\in T_{x_k} M_k$; it induces a pairing between $T^\flat\widetilde{M}$ and $T\widetilde{M}$ (cf. Theorem \ref{T_ProjectiveLimitPartialDirac}, 3.).\\
Now,  $\widetilde{D}_n=\displaystyle\bigoplus_{k=1}^n D_k$ is a closed subbundle of $T^\mathfrak{p}\widetilde{M}_n=\displaystyle\bigoplus_{k=1}^nT^\mathfrak{p}M_k$ and, by Lemma \ref{Ex_ProductDirac}, $\widetilde{D}_n$  is a  partial almost  Dirac structure on $\widetilde{M}_n$.   From Example \ref{Ex_ProductDirac}, 2., it follows that $ \left( D_n \right)$ is a submersive projective sequence of partial almost Dirac structures. Therefore $\widetilde{D}=\underrightarrow{\lim}(\tilde{p}_n(\widetilde{D}_n))\times\underleftarrow{\lim}(\tilde{p}^\flat_n(\widetilde{D}_n))$ is a partial  almost Dirac structure on $\widetilde{M}$.\\  
Assume now that each $M_k$ is a Hilbert manifold and then $TM_k$ and $T^\flat M_k$ are Hilbert vector bundles. According to Example~\ref{Ex_SumBanachDiracstructures}, since  the characteristic distribution of $D_k$ is  integrable,  by induction, we obtain that $
\widetilde{D}_n$ can be provided with a strong partial Lie algebroid structure and so is integrable. In fact, (cf. Example~\ref{Ex_SumBanachDiracstructures}), for any $\tilde{x}_n=(x_1,\dots,x_k\dots, x_n)$, if 
$L_k$ is the characteristic leaf though $x_k$ in $M_k$, then $\widetilde{L}_n=\displaystyle\prod_{k=1}^n L_k$ is the characteristic leaf of $\widetilde{D}_n$ through $\tilde {x}$. This implies that the partial almost Dirac structure $D$ on $M$ is integrable and the characteristic leaf through $\tilde{x}=(x_k)$ is $\displaystyle\prod_{k=1}^\infty L_k$ where $L_k$ is the characteristic leaf of $D_k$ through $x_k$.
Now,  by induction, as in Example~\ref{Ex_SumBanachDiracstructures}, on any characteristic leaf $\widetilde{L}_n=\displaystyle\prod_{k=1}^n L_k$, we have a strong pre-symplectic form $\Omega_{\widetilde{L}_n}=\Omega_{L_1}\oplus\cdots\oplus\Omega_{L_n}$.\\

Unfortunately, in general, for any $\tilde{x}=(x_k)_{k\in \mathbb{N}}$, if $\widetilde{L}_n$ is the characteristic leaf through $\tilde{x}_n=(x_1,\dots,x_k,\dots, x_n)$, then the sequence of Hilbert manifolds $ \left( \widetilde{L}_n \right) $   does not have  the projective limit chart property and so $\widetilde{L}=\underleftarrow{\lim}\widetilde{L}_n$ does not have a Fr\'echet manifold structure.

However, consider the particular case where $M_k$ is a Hilbert space for all $k\geq k_0$; for such integers  $k$, let $\mathbb{D}_k$ be  a linear Dirac structure on $M_k$. According to Example~\ref{Ex_RegularDiracmanifolds}, $D_k=\mathbb{D}_k\times M_k$ 
is a partial Dirac structure on the manifold $M_k$ and each characteristic leaf $L_k$ of $D_k$ is of type $\{x_k\}\times p(\mathbb{D}_k)$ and is provided with a strong pre-symplectic form $\Omega_L$.   Given some $\tilde{x}_n=(x_k)_{k\in \mathbb{N}} \in  \widetilde{M}$,  let $ \widetilde{L}_n$  be  
 the characteristic leaf through $\tilde{x}_n=(x_1,\dots,x_k,\dots x_n)$.  Then the sequence of Hilbert manifolds $ \left( \widetilde{L}_n \right) $   has  now  the projective limit chart property around $\tilde{x}$ and so $\widetilde{L}=\underleftarrow{\lim}\widetilde{L}_n$ has  a 
Fr\'echet manifold structure.
 Therefore, for any  $\tilde{x}=(x_k)$, if $\widetilde{L}_n$ is the characteristic leaf through $\tilde{x}_n=(x_1,\dots,x_k,\dots x_n)$, then the sequence of Hilbert manifolds $ \left( \widetilde{L}_n \right) $ has  the projective limit chart property and so $\widetilde{L}
 =\underleftarrow{\lim}\widetilde{L}_n$ has  a Fr\'echet manifold structure which is the  characteristic leaf of $\widetilde{D}$ on $\widetilde{M}$ through $\tilde{x}$. Moreover,  $\widetilde{L}$ is provided with a strong pre-symplectic $\Omega_{\widetilde{L}}$  which, from 
 Theorem \ref{T_LimitSequenceLinearDirac}, 2. (v) is   $\left(\Omega_{\widetilde{L}}\right)_{\tilde{x}}=\underrightarrow{\lim}\left(\Omega_{\widetilde{L}_n}\right)_{\tilde{x}_n}$.
\end{example}

%%%%%%%%%%%%%%%%%%

% \printindex

\bigskip
\begin{minipage}[t]{10cm}
\begin{flushleft}
\small{
\textsc{Fernand Pelletier}
\\*e-mail: fernand.pelletier@univ-smb.fr
\\UMR 5127 CNRS, Universit\'e de Savoie Mont Blanc, LAMA
\\Campus Scientifique
Le Bourget-du-Lac, 73370, France
\\[0.4cm]
\textsc{Patrick Cabau}
\\*e-mail: patrickcabau@gmail.com

}
\end{flushleft}
\end{minipage}

\end{document}